\documentclass[11pt,reqno]{amsart}
\usepackage{amsmath}
\usepackage{amsfonts}
\usepackage{amssymb, amsthm, epsfig}
\usepackage{mathrsfs}
\usepackage{hyperref}
 \usepackage{yhmath}
\usepackage{cite}
\usepackage[utf8]{inputenc}
\usepackage[T1]{fontenc}
\usepackage{geometry}
\usepackage{graphicx}
\usepackage{verbatim}
\usepackage{color}

\theoremstyle{plain}
\newtheorem{thm}{Theorem}[section]

\newtheorem{lem}[thm]{Lemma}
\newtheorem{prop}[thm]{Proposition}

\theoremstyle{definition}
\newtheorem{defn}{Definition}[section]
\theoremstyle{remark}
\newtheorem{rem}{Remark}[section]

\numberwithin{equation}{subsection}

\allowdisplaybreaks

\DeclareSymbolFont{lettersA}{U}{pxmia}{m}{it}
\DeclareMathSymbol{\piup}{\mathord}{lettersA}{"19}

\makeatletter

\newcommand{\Rmnum}[1]{\expandafter\@slowromancap\romannumeral#1@}
\makeatother

\begin{document}
\title[2D steady supersonic ramp flows of BZT fluids]
{Two-dimensional steady supersonic ramp flows of Bethe-Zel'dovich-Thompson fluids}

\author{Geng Lai}\address{Department of Mathematics, Shanghai University,
Shanghai, 200444, P.R. China;  Newtouch Center for Mathematics, Shanghai University, Shanghai, 20044, P.R. China.}
\email{\tt laigeng@shu.edu.cn}

\keywords{BZT fluid; shock-fan-shock composite wave; fan-shock-fan composite; sonic shock; hodograph transformation; characteristic decomposition}
\date{\today}

\begin{abstract}
Two-dimensional (2D) steady supersonic ramp flows are important and well-studied flow patterns in aerodynamics.
In the paper \cite{VKG1}, Vimercati, Kluwick, and Guardone constructed various self-similar composite wave solutions consisting of centered simple waves and oblique shocks for the 2D steady supersonic flow of BZT fluids past compressible and rarefactive ramps.
In the present paper, we study the stabilities of the self-similar fan-shock-fan and shock-fan-shock composite waves in 2D steady supersonic flows of BZT fluids.
In contrast to ideal gases,
the flow downstream  (or upstream) of a shock of a BZT fluid may be sonic in the sense of the flow velocity relative to the shock
front, and the formulation of boundary conditions for the shocks is usually unknown in advance.
In order to study the stabilities of the composite waves,
we establish some a priori estimates of different shock types by comparing the shock speed with the acoustic speed along the shocks and applying Liu's extended entropy condition, and to solve some post-sonic and pre-sonic shock free boundary problems.
The sonic shocks are envelopes of one of the acoustic families of characteristics, and not characteristics. This results in a fact that the flow downstream (or upstream) of a sonic shock is not $C^1$ smooth up to the shock boundary.
We use a weighted characteristic decomposition method and a hodograph transformation method
 to overcome the difficulty caused by the singularity on sonic shocks of 2D steady full Euler equations and potential flow equations, respectively.
Some new iteration schemes are also proposed to solve the post-sonic and pre-sonic shock free boundary problems.
\end{abstract}

\maketitle
\tableofcontents

\section{Introduction}

Two-dimensional (2D) steady supersonic flows past compressible and rarefactive ramps
are fundamental in gas dynamics, as they provide global or local structures in various flow fields.
Referring to Figure \ref{Figure1},
there is an infinite long wedge with a horizontal wall which is straight up to a sharp
corner O and a straight ramp. A supersonic flow arrives with a constant state along
the horizontal wall and turns at O into a new direction.
If the incoming flow is a polytropic ideal gas, the well-known oblique shock and centered simple wave may occur at a compressive and rarefactive ramp, respectively.

\begin{figure}[htbp]
\begin{center}
\includegraphics[scale=0.4]{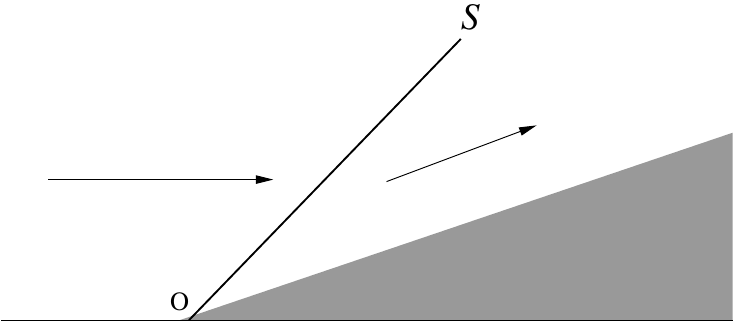}\qquad\qquad\qquad\includegraphics[scale=0.38]{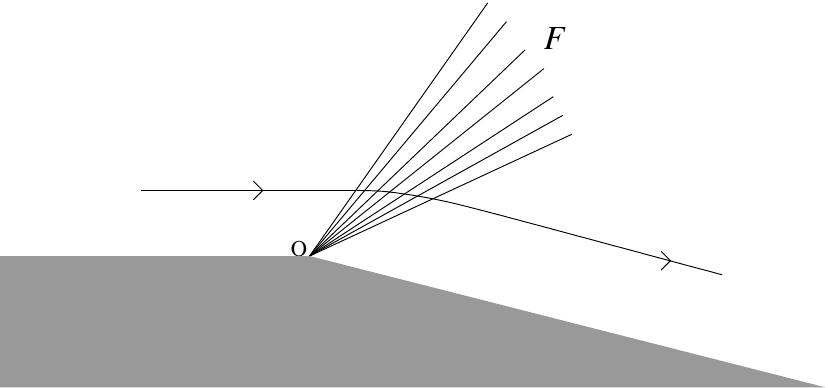}
\caption{\footnotesize 2D steady supersonic flow past a compressive (left) and rarefactive (right) ramp for a polytropic ideal gas.}
\label{Figure1}
\end{center}
\end{figure}

For polytropic ideal gases, the isentropes are always convex downward in the $(\tau, p)$-plane, where $\tau$ is the specific volume and $p$ the pressure.
A natural and interesting question is to determine whether the above result
can be extended to more general gases, for example, gases with nonconvex isentropes.
The possibility that isentropes are nonconvex was first explored independently by Bethe \cite{Bethe} and Zel'dovich \cite{ZE} for van der Waals gases. Later on, Thompson \cite{Thompson1}
introduced a non-dimensional thermodynamic quantity $\mathcal{G}=-\frac{\tau p_{\tau\tau}(\tau, s)}{2p_{\tau}(\tau, s)}$ called the fundamental derivative of gasdynamics to
delineate the dynamic behavior of compressible fluids, where $s$ is the specific entropy.
Fluids with $\mathcal{G}<0$ within a certain thermodynamic region are now commonly referred to as Bethe-Zel'dovich-Thompson (BZT) fluids.
This type of fluids may significantly differ from the polytropic ideal gases. For example, physically admissible rarefactive shocks may occur in gases with nonconvex isentropes; see e.g., \cite{BBKN,Cra,NSMGC,Thompson2,Weyl,ZGC}.

\begin{figure}[htbp]
\begin{center}
\includegraphics[scale=0.39]{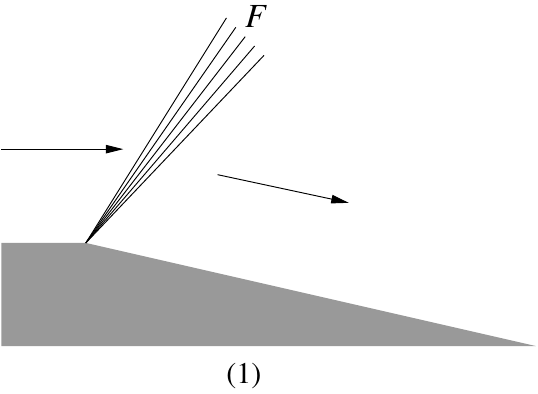}\qquad\includegraphics[scale=0.39]{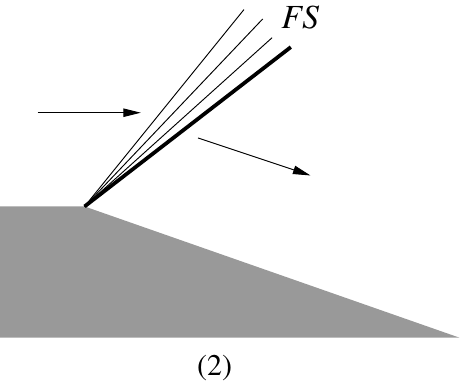}\qquad \includegraphics[scale=0.39]{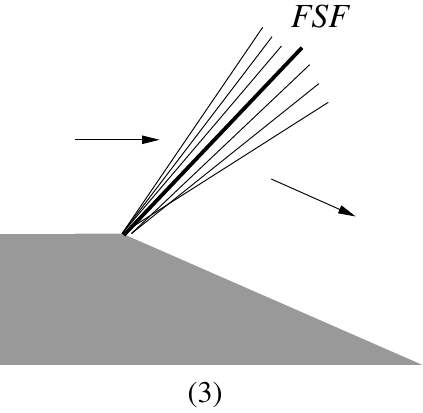} \\
\includegraphics[scale=0.39]{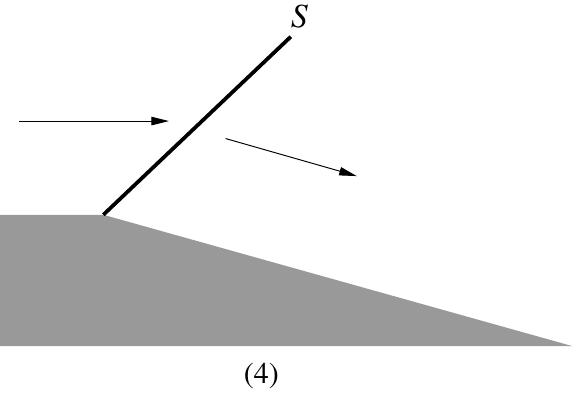}\qquad\includegraphics[scale=0.39]{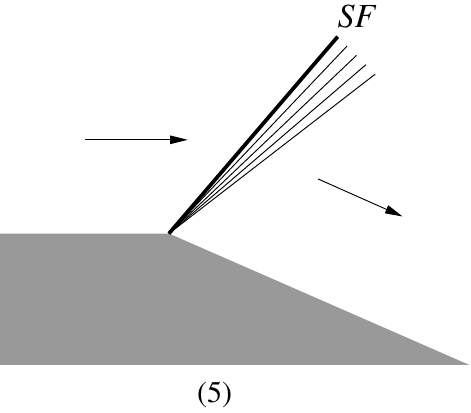}
\caption{\footnotesize 2D steady supersonic flows of BZT fluids past a rarefactive ramp.}
\label{Figure2}
\end{center}
\end{figure}


The 2D steady supersonic ramp problem can be treated as a Riemann-type initial-boundary value problem if the incoming flow is uniform and the ramp is straight. There have been many important progress on the 1D Riemann problem for nonconvex hyperbolic conservation laws; see \cite{ZH,LeFloch1,LeFloch2,Liu1,Liu2,MP,MV,Wen1,Wen2}.
In a recent paper \cite{VKG1}, Vimercati, Kluwick and Guardone constructed self-similar solutions for the 2D steady supersonic flow of BZT fluids past compressible and rarefactive ramps.
For the BZT fluids, many additional wave configurations may possibly occur in the 2D steady supersonic ramp flows.
The wave configuration at a rarefactive ramp may be a fan wave (${\it F}$), fan-shock composite wave (${\it FS}$), fan-shock-fan composite wave (${\it FSF}$), oblique shock wave (${\it S}$), or shock-fan composite wave (${\it SF}$); see Figure \ref{Figure2}.
 The wave configuration at a compressive ramp may be
 an oblique shock wave (${\it S}$), fan-shock composite wave (${\it FS}$), shock-fan-shock composite wave (${\it SFS}$), fan wave (${\it F}$), or fan-shock composite wave (${\it FS}$); see Figure \ref{Figure3}.
Recently, Li $\&$ Sheng \cite{Li-Sheng1,Li-Sheng2} also constructed self-similar solutions to the supersonic ramp problem for a van der Waals gas supplemented by Maxwell's rule.

\begin{figure}[htbp]
\begin{center}
\includegraphics[scale=0.43]{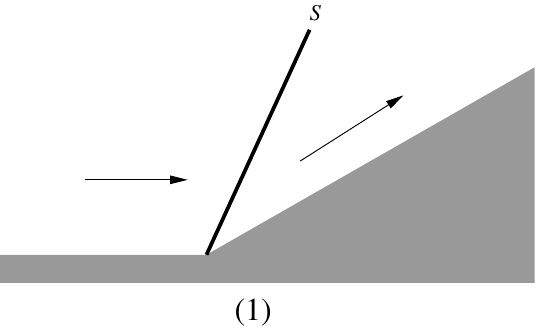}\quad \includegraphics[scale=0.43]{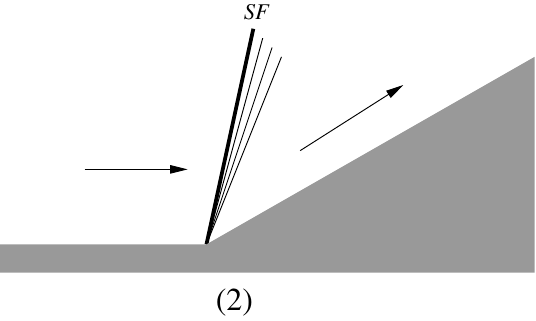}\quad \includegraphics[scale=0.43]{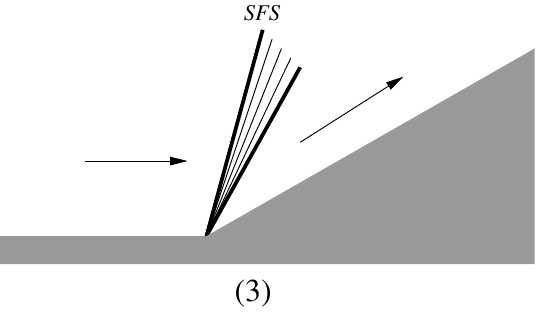} \\
\includegraphics[scale=0.43]{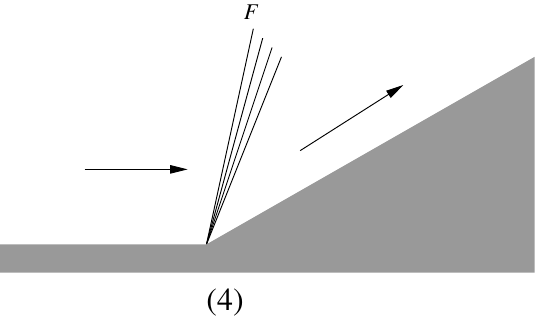}\qquad \includegraphics[scale=0.43]{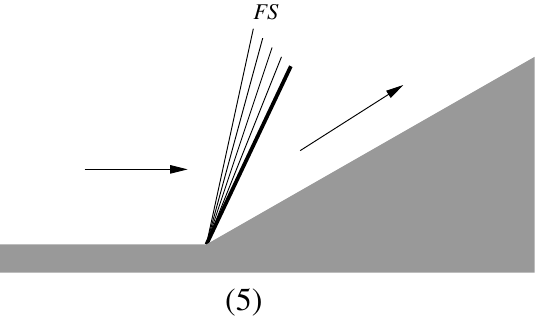}
\caption{\footnotesize 2D steady supersonic flows of BZT fluids past a compressive ramp.}
\label{Figure3}
\end{center}
\end{figure}

We are concerned with the stabilities of the oblique composite waves in 2D steady supersonic ramp flows of BZT fluids.
There have been many important studies on the stability of oblique shocks in steady supersonic flows past compressive ramps and cones for polytropic ideal gases.
For supersonic flows past a curved compressive ramp, the local existence of the corresponding curved shock has been established in \cite{Gu,Li-Yu,Sch}. Yin \cite{Yin} applied a reflected characteristic method to obtain the global existence of a curved shock, provided that the curved ramp is a small perturbation of a straight one.
If the curved ramp is not a small perturbation of a straight one, the global existence of a curved shock was established in \cite{HD,HQ} with the assumption that the Mach number of the incoming flow is sufficiently large.
The local existence of a curved conical shock for supersonic flow past a curved cone was established in
\cite{CLD1,CLD2}.
For the global existence of curved conical shocks, we refer the reader to \cite{CXY,CY1,CY2,LWY1,HZ,XY1,XY2}.
Chen {\it et al.} \cite{CCX,CFf} also obtained the global existence of a curved surface shock for supersonic flow past a three-dimension compressible ramp.
We also refer to \cite{CKXZ,CKZ,CZZ,LL,WZ,Zhang} for the global existence of supersonic flows past Lipschitz perturbed wedges and cones in the frame work of B.V. functions.

In this paper, we study the stabilities of the self-similar fan-shock-fan and shock-fan-shock composite waves
in the 2D steady supersonic ramp flows of BZT fluids; see Figures \ref{Figure2} (3) and \ref{Figure3} (3).
In contrast to ideal gases,
the flow downstream  (or upstream) of a shock of a BZT fluid may be sonic in the sense of the flow velocity relative to the shock
front, and the shock is usually called post-sonic (or pre-sonic) shock (see, e.g., \cite{VKG1,VKG2}) and is also related to sonic phase transition (see \cite{Colombo,Corli}).
For example, in the oblique shock-fan-shock composite wave, the head shock is post-sonic, while the tail shock is pre-sonic;
in the oblique fan-shock-fan composite wave, the shock is double-sonic (see \cite{Cra}).
One main difficulty in investigating the propagation of the shocks of the composite waves lies in that the type of each shock cannot be determined a priori. This differs from the case of ideal gases and makes the formulation of the boundary conditions on the shocks a priori unknown. In order to overcome this difficulty, we establish some a priori estimates for the type of the shocks by comparing the accelerations of the shock wave and the acoustic wave‌ at the corner.
The other difficulty is that curved sonic shocks would occur.
The post-sonic (pre-sonic, resp.) shocks are envelopes of one out of the two families of acoustic wave characteristics of the downstream (upstream, resp.) flow fields, and not characteristics. This will result in a fact that the flow downstream (upstream, resp.) of a post-sonic (pre-sonic, resp.) shock is not $C^1$ smooth up to the shock boundary.
We use a weighted characteristic decomposition method and a hodograph transformation method
 to overcome the difficulty cased by the singularity on sonic shocks of 2D steady full Euler equations and potential flow equations, respectively,
and derive several groups of structural conditions to ensure the existence of flows adjacent up to the sonic shocks. 

The rest of the paper is organized as follows.

In Section 2, we study the stability of the shock-fan-shock composite wave for the 2D steady potential flow equations.
First, we consider the scenario of a uniform incoming flow over a curved ramp. By calculating the curvature of the tail shock, we prove that the tail shock is pre-sonic (transonic, resp.) if and only if the ramp is convex-up (concave-up, resp.). Through solving two distinct types of shock free boundary problems, we establish the local existence of a solution near the corner, as stated in Theorem \ref{Thm39}.
Second, we extend the analysis to the case of a non-uniform incoming flow over a curved ramp. A set of sufficient conditions on the incoming flow are derived to guarantee the existence of a post-sonic head shock. We also find a critical condition on the ramp to establish the existence of a pre-sonic/transonic tail shock. The local existence of a solution with a post-sonic head shock and a pre-sonic/transonic tail shock for the non-uniform incoming flow over a curved ramp is stated in Theorem \ref{thm23}.


In Section 3, we investigate the stability of the fan-shock-fan composite wave in supersonic flows past a rarefactive ramp for the 2D steady full Euler equations.
We consider a straight wedge with a non-uniform incoming flow.
Via a weighted characteristic decomposition method, we establish the existence of a near-corner flow configuration featuring a curved post-sonic shock, under the condition that the incoming flow satisfies certain monotonicity conditions.
The main result of this section is presented in Theorem \ref{main}.

\section{Supersonic flows past a compressive ramp}

\subsection{System of 2D steady potential flow}
In this section we study the stability of shock-fan-shock composite wave for the 2D steady potential flow equations.
The 2D steady Euler equations for potential flow can be written as
\begin{equation}\label{42501}
\left\{
  \begin{array}{ll}
   (\rho u)_x+(\rho v)_y=0,\\[4pt]
  \displaystyle  u_y-v_x=0
  \end{array}
\right.
\end{equation}
supplemented by the Bernoulli's law
\begin{equation}\label{5802}
\frac{1}{2}(u^2+v^2)+h=\mbox{Constant}.
\end{equation}
Here,  $h$ is the specific enthalpy which satisfies
\begin{equation}\label{5801}
h'(\tau)=\tau p'(\tau),
\end{equation}
$(u, v)$ is the flow velocity, $\rho$ is the density, $\tau$ is the specific volume, and $p=p(\tau)$ is the pressure.
A typical example of the BZT fluid is the well-known van der Waals gas.
We take a van der Waals gas equation of state
\begin{equation}\label{8310a}
p=\frac{S}{(\tau-1)^{\gamma}}-\frac{1}{\tau^2},
\end{equation}
where $S$ is a positive constant in corresponding with the entropy, and $\gamma$ is an
adiabatic constant between $1$ and $2$.
We assume that there exist $\tau_1^{i}$ and $\tau_2^{i}$ ($1<\tau_1^{i}<\tau_2^{i}$) such that
\begin{equation}\label{81601}
p'(\tau)<0~\mbox{for}~\tau>1;\quad
 p''(\tau)>0~ \mbox{for}~ \tau\in(1, \tau_1^i)\cup(\tau_2^i, +\infty);
 \quad p''(\tau)<0~ \mbox{for}~ \tau_1^i<\tau<\tau_2^i;
\end{equation}
Under the assumption (\ref{81601}), the specific volume $\tau$  can by (\ref{5802}) be seen as a function of the flow speed $q=\sqrt{u^2+v^2}$. We denote this function by $\tau=\ddot{\tau}(q)$.
\subsubsection{\bf Characteristic equations for potential flow}

For smooth flow, system (\ref{42501}) can be written in the following matrix form
\begin{equation}
\left(
 \begin{array}{cc}
c^{2}-u^{2} & -uv\\
  0 & -1 \\
  \end{array}
  \right)\left(
           \begin{array}{c}
             u \\
             v \\
           \end{array}
         \right)_{x}+\left(
                         \begin{array}{cc}
                          -uv &  c^{2}-v^{2}\\
                           1 & 0 \\
                         \end{array}
                       \right)\left(
                                \begin{array}{c}
                                  u \\
                                  v \\
                                \end{array}
                              \right)_{y}~=~\left(
                                               \begin{array}{c}
                                                 0 \\
                                                 0 \\
                                               \end{array}
                                             \right),
                                             \label{matrix}
\end{equation}
where $c=\sqrt{-\tau^2p'(\tau)}$ is the speed of sound. We set $$\hat{c}(\tau)=\sqrt{-\tau ^2p'(\tau)}
\quad \mbox{and} \quad \ddot{c}(q)=\sqrt{-\ddot{\tau}^2(q)p'(\ddot{\tau}(q))}.$$

The eigenvalues of (\ref{42501}) are determined
by
\begin{equation}
(v-\lambda u)^{2}-c^{2}(1+\lambda^{2})=0,\label{characteristice}
\end{equation}
which yields
\begin{equation}
\lambda=\lambda_{\pm}=\frac{uv\pm
c\sqrt{u^{2}+v^{2}-c^{2}}}{u^{2}-c^{2}}.
\end{equation}
So, if and only if $q>c$, system (\ref{42501}) is hyperbolic and has two families of wave characteristics $C_{\pm}$
defined as the integral curves of
$\frac{{\rm d}y}{{\rm d}x}=\lambda_{\pm}$.

\begin{figure}[htbp]
\begin{center}
\includegraphics[scale=0.58]{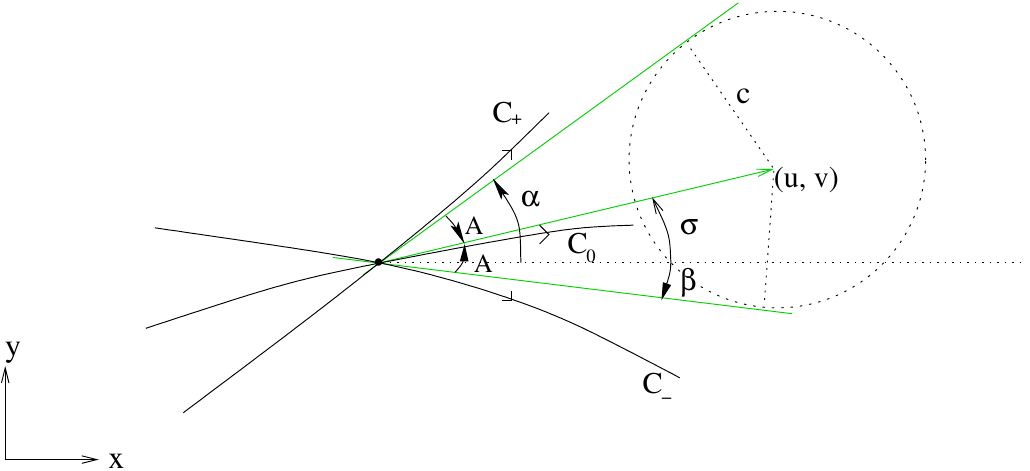}
\caption{ \footnotesize Characteristic angles and characteristic direction.}
\label{Characteristicd}
\end{center}
\end{figure}

In this paper, we will use the method of characteristics to discuss the stability of the composite waves. To this end, we introduce the concept of the direction of the wave characteristics.
The direction of the wave
characteristics is defined as the tangent direction that forms an
acute angle $A$ with the direction of the flow velocity
$(u, v)$; see Figure \ref{Characteristicd}. The $C_+$
characteristic direction forms with the direction of the flow velocity
 the angle $A$ from $(u, v)$ to $C_{+}$ in
the counterclockwise direction, and the $C_-$ characteristic direction forms with the
 direction of the flow velocity the angle $A$ from $(u, v)$ to $C_{-}$
in the clockwise direction.  By computation, we have
\begin{equation}
c^{2}=q^{2}\sin^{2} A.\label{210cqo}
\end{equation}
The angle $A$ is called the
Mach angle.
The $C_{+}$ ($C_{-}$) characteristic angle is defined as the counterclockwise angle from the positive $x$-axis to the $C_{+}$ ($C_{-}$) characteristic direction.
We denote by $\alpha$ and
 $\beta$ the $C_{+}$ and $C_{-}$ characteristic angles,  respectively,
 where $0<\alpha-\beta<\pi$. Let $\sigma$ be the counterclockwise angle from the positive $x$-axis to the direction of the flow velocity.
Then we have
\begin{equation}
\alpha=\sigma+A,\quad \beta=\sigma-A,\quad\sigma=\frac{\alpha+\beta}{2},\quad A=\frac{\alpha-\beta}{2}. \label{tau}
\end{equation}
Therefore, the relations between $(u,v,c)$ and $(\sigma, A, c)$ are
\begin{equation}\label{U}
u=\frac{c\cos\sigma}{\sin A}\quad \mbox{and}\quad v=\frac{c\sin\sigma}{\sin A}.
\end{equation}
By computation we also have
\begin{equation}\label{4402}
q^{2}\cos\alpha\cos\beta=u^2-c^2,\quad q^{2}\sin\alpha\sin\beta=v^2-c^2.
\end{equation}

Multiplying (\ref{matrix}) on the left
by
$(1,\mp c\sqrt{u^{2}+v^{2}-c^{2}})$ and using (\ref{4402}), we get
\begin{equation}
\left\{
  \begin{array}{ll}
  \displaystyle \bar{\partial}_{+}u+\lambda_{-}\bar{\partial}_{+}v
 =0,  \\[10pt]
    \displaystyle  \bar{\partial}_{-}u+\lambda_{+}\bar{\partial}_{-}v=0,
  \end{array}
\right.\label{forma}
\end{equation}
where
$\bar{\partial}_{\pm}$ are characteristic directional derivatives defined by
\begin{equation}\label{72801}
\bar{\partial}_{+}:=\cos\alpha\partial_{x}+\sin\alpha\partial_{y}\quad\mbox{and}\quad \bar{\partial}_{-}:=\cos\beta\partial_{x}+\sin\beta\partial_{y}.
\end{equation}

From (\ref{72801}) we have
\begin{equation}\label{72802}
\partial_{x}=-\frac{\sin\beta\bar{\partial}_{+}-\sin\alpha\bar{\partial}_{-}}{\sin2A},\quad
\partial_{y}=\frac{\cos\beta\bar{\partial}_{+}-\cos\alpha\bar{\partial}_{-}}{\sin2A},
\end{equation}
and
\begin{equation}\label{32403}
\bar{\partial}_{+}+\bar{\partial}_{-}=2\cos A \bar{\partial}_{0},
\end{equation}
where
\begin{equation}\label{82900a}
\bar{\partial}_{0}=\cos\sigma\partial_x+\sin\sigma\partial_y.
\end{equation}

From (\ref{U}) we have
\begin{equation}
\bar{\partial}_{\pm}u=\frac{\cos\sigma}{\sin A}\bar{\partial}_{\pm}c+\frac{c\cos\alpha\bar{\partial}_{\pm}\beta
-c\cos\beta\bar{\partial}_{\pm}\alpha}{2\sin^{2}A},\label{1a}
\end{equation}
\begin{equation}
\bar{\partial}_{\pm}v=\frac{\sin\sigma}{\sin A}\bar{\partial}_{\pm}c
+\frac{c\sin\alpha\bar{\partial}_{\pm}\beta
-c\sin\beta\bar{\partial}_{\pm}\alpha}{2\sin^{2}A}.\label{2a}
\end{equation}
Inserting (\ref{1a}) and (\ref{2a}) into (\ref{forma}), we obtain
\begin{equation}\label{3a}
\bar{\partial}_{+}c=\frac{c}{\sin2A}
(\bar{\partial}_{+}\alpha-\cos2A\bar{\partial}_{+}\beta)
\quad\mbox{and}\quad
\bar{\partial}_{-}c=\frac{c}{\sin2A}
(\cos2A\bar{\partial}_{-}\alpha-\bar{\partial}_{-}\beta).
\end{equation}

Differentiating the Bernoulli's law (\ref{5802}) and
using (\ref{1a}) and (\ref{2a}), we get
\begin{equation}
\left(\frac{1}{\sin^{2}A}+\kappa\right)\bar{\partial}_{\pm}c=
\frac{c\cos A}{2\sin^{3}A}(\bar{\partial}_{\pm}\alpha-\bar{\partial}_{\pm}\beta),\label{bB1a}
\end{equation}
where $$\kappa=\frac{-2p'(\tau)}{\tau p''(\tau)+2p'(\tau)}.$$

Inserting (\ref{bB1a}) into (\ref{3a}), we obtain
\begin{equation}\label{6a}
\bar{\partial}_{+}\alpha=\Omega\cos^{2}A\bar{\partial}_{+}\beta\quad \mbox{and}\quad \bar{\partial}_{-}\beta=\Omega\cos^{2}A\bar{\partial}_{-}\alpha,
\end{equation}
respectively,
where
$$
\Omega=~\varpi(\tau)-\tan^2A\quad \mbox{and}\quad \varpi(\tau)=-\frac{4p'(\tau)+\tau p''(\tau)}{\tau p''(\tau)}.
$$

Combining (\ref{3a}) and (\ref{6a}), we have
\begin{equation}
c\bar{\partial}_{+}\beta=-(1+\kappa)\tan A\bar{\partial}_{+}c=-\frac{p''(\tau)}{2c\rho^4}\tan A\bar{\partial}_{+}\rho,\label{7a}
\end{equation}
\begin{equation}
c\bar{\partial}_{+}\alpha=-\left(\frac{1+\kappa}{2}\right)\Omega\sin2 A\bar{\partial}_{+}c=-\frac{p''(\tau)}{4c\rho^4}\Omega\sin2 A\bar{\partial}_{+}\rho,\label{10a}
\end{equation}
\begin{equation}\label{82301a}
c\bar{\partial}_{-}\alpha=(1+\kappa)\tan A\bar{\partial}_{-}c=\frac{p''(\tau)}{2c\rho^4}\tan A\bar{\partial}_{-}\rho,
\end{equation}
\begin{equation}
c\bar{\partial}_{-}\beta=\left(\frac{1+\kappa}{2}\right) \Omega\sin2 A\bar{\partial}_{-}c=\frac{p''(\tau)}{4c\rho^4}\Omega\sin2 A\bar{\partial}_{-}\rho.\label{8a}
\end{equation}

From (\ref{1a}), (\ref{2a}), and (\ref{7a})--(\ref{8a}) we have
\begin{equation}\label{11a}
\bar{\partial}_{+}u=\kappa\sin\beta\bar{\partial}_{+}c=c\tau\sin\beta\bar{\partial}_{+}\rho,\quad\bar{\partial}_{-}u=-\kappa\sin\alpha\bar{\partial}_{-}c=-c\tau\sin\alpha\bar{\partial}_{-}\rho,
\end{equation}
\begin{equation}\label{72804a}
\bar{\partial}_{+}v=-\kappa\cos\beta\bar{\partial}_{+}c=-c\tau\cos\beta\bar{\partial}_{+}\rho,
\quad\bar{\partial}_{-}v=\kappa\cos\alpha\bar{\partial}_{-}c=c\tau\cos\alpha\bar{\partial}_{-}\rho.
\end{equation}

\subsubsection{\bf Riemann invariants}
The Riemann invariants of (\ref{42501}) are defined as
\begin{equation}\label{Riemanni}
r_{\pm}(\sigma,q)=\sigma\pm\int^{q}\frac{\sqrt{q^{2}-c^{2}}}{qc}dq,
\end{equation}
in which $c=\ddot{c}(q)$.
In view of the Riemann invariants, we have
\begin{equation}\label{52801}
\left\{
  \begin{array}{ll}
    \bar{\partial}_{+}r_{-}=0,  \\[2pt]
     \bar{\partial}_{-}r_{+}=0.
  \end{array}
\right.
\end{equation}

It follows from
\begin{equation}
\left(
  \begin{array}{cc}
    \displaystyle\frac{\partial r_{+}}{\partial \sigma} & \displaystyle\frac{\partial r_{+}}{\partial q}   \\[8pt]
    \displaystyle\frac{\partial r_{-}}{\partial \sigma}  & \displaystyle\frac{\partial r_{-}}{\partial q}
  \end{array}
\right)\left(
         \begin{array}{cc}
          \displaystyle \frac{\partial \sigma}{\partial r_{+}}  & \displaystyle\frac{\partial \sigma}{\partial r_{-}} \\[8pt]
           \displaystyle\frac{\partial q}{\partial r_{+}} & \displaystyle \frac{\partial q}{\partial r_{-}}
         \end{array}
       \right)=\left(
                 \begin{array}{cc}
                   1 & 0 \\[12pt]
                   0 & 1 \\
                 \end{array}
               \right)
\end{equation}
that
\begin{equation}
\frac{\partial\sigma}{\partial r_{\pm}}=\frac{1}{2}\quad \mbox{and} \quad
\frac{\partial q}{\partial
r_{\pm}}=\pm \frac{qc}{2\sqrt{q^{2}-c^{2}}}=\pm \frac{q\sin A}{2\cos A}.\label{1201}
\end{equation}

Thus, by (\ref{U}) we have
\begin{equation}\label{6603}
\frac{\partial u}{\partial r_{+}}=-\frac{q\sin\beta}{2\cos A}, \quad \frac{\partial u}{\partial r_{-}}=-\frac{q\sin\alpha}{2\cos A}, \quad \frac{\partial v}{\partial r_{+}}=\frac{q\cos\beta}{2\cos A}, \quad \mbox{and} \quad
\frac{\partial v}{\partial r_{-}}=\frac{q\cos\alpha}{2\cos A}.
\end{equation}
Combining these with the Bernoulli's law (\ref{5802}) we also have
\begin{equation}\label{6606}
\frac{\partial \rho}{\partial r_{-}}=\frac{\rho }{\sin (2A)}\quad \mbox{and}\quad \frac{\partial \rho}{\partial r_{+}}=-\frac{\rho }{\sin (2A)}.
\end{equation}

From (\ref{6603}) we also have
\begin{equation}\label{62401}
\frac{\partial r_{+}}{\partial u}=\frac{\cos\alpha}{c}, \quad
\frac{\partial r_{+}}{\partial v}=\frac{\sin\alpha}{c}, \quad \frac{\partial r_{-}}{\partial u}=-\frac{\cos\beta}{c},\quad  \frac{\partial r_{-}}{\partial v}=-\frac{\sin\beta}{c}.
\end{equation}

By the Bernoulli's law (\ref{5802}), we have
\begin{equation}
\frac{\partial c}{\partial q}=-\frac{q}{c\kappa(\tau)}.\label{dc}
\end{equation}
Combining this with $c=q\sin A$, we get
\begin{equation}
\frac{\partial  A}{\partial q}=-\frac{q}{\sqrt{q^{2}-c^{2}}}\left(\frac{c}{q^{2}}+\frac{1}{c\kappa}\right).
\end{equation}
Thus, we have
\begin{equation}\label{omegar}
\frac{\partial A}{\partial r_{\pm}}=\frac{\partial  A}{\partial q}\cdot\frac{\partial q}{\partial
r_{\pm}}=\mp\frac{q^{2}c}{2(q^{2}-c^{2})}\left(\frac{c}{q^{2}}+\frac{1}{c\kappa}\right),
\end{equation}
\begin{equation}\label{alphar}
\frac{\partial(\sigma\pm A)}{\partial r_{\mp}}=\frac{1}{2}+\frac{q^{2}c}{2(q^{2}-c^{2})}\left(\frac{c}{q^{2}}+\frac{1}{c\kappa}\right)
=-\frac{q^2\tau p''(\tau)}{4(q^2-c^2)p'(\tau)},
 \end{equation}
and
 \begin{equation}\label{alphar1}
\frac{\partial(\sigma\pm A)}{\partial r_{\pm}}=\frac{1}{2}-\frac{q^{2}c}{2(q^{2}-c^{2})}\left(\frac{c}{q^{2}}+\frac{1}{c\kappa}\right)=1+\frac{q^2\tau p''(\tau)}{4(q^2-c^2)p'(\tau)}.
 \end{equation}

In terms of the Riemann invariants, we have the characteristic decompositions (see Lai \cite{Lai1})
\begin{equation}
\left\{
  \begin{array}{ll}
  \bar{\partial}_{+}\bar{\partial}_{-}r_{-}=\mathcal{R}\big(\bar{\partial}_{-}r_{-}
  -\cos2A\bar{\partial}_{+}r_{+}\big)\bar{\partial}_{-}r_{-},  \\[4pt]
   \bar{\partial}_{-}\bar{\partial}_{+}r_{+}=-\mathcal{R}\big(\bar{\partial}_{+}r_{+}
   -\cos2A\bar{\partial}_{-}r_{-}
\big)\bar{\partial}_{+}r_{+}.
  \end{array}
\right.\label{cdr}
\end{equation}
where
$\mathcal{R}=-\frac{q^2\tau p''(\tau)}{4(q^{2}-c^{2})p'(\tau)\sin 2A}$.


\subsubsection{\bf Hodograph transformation}
Let the hodograph transformation be
$$
T:  (x, y)\rightarrow (r_{+}, r_{-})
$$
for (\ref{52801}).
If the Jacobian
$$
j(r_{+}, r_{-}; x, y)=\frac{\partial(r_{+}, r_{-})}{\partial(x, y)}=\frac{\partial r_{+}}{\partial{x}}\frac{\partial r_{-}}{\partial{y}}-\frac{\partial r_{-}}{\partial{x}}\frac{\partial r_{+}}{\partial{y}}\neq 0
$$
for a solution $(r_{+}, r_{-})(x, y)$ of (\ref{52801}), we may consider $x$ and $y$ as functions of $r_{+}$ and $r_{-}$. From
\begin{equation}
\frac{\partial r_{+}}{\partial{x}}=j\frac{\partial y}{\partial r_{-}}, \quad \frac{\partial r_{+}}{\partial{y}}=-j\frac{\partial x}{\partial r_{-}}, \quad \frac{\partial r_{-}}{\partial{x}}=-j\frac{\partial y}{\partial r_{+}}, \quad \frac{\partial r_{-}}{\partial{y}}=j\frac{\partial x}{\partial r_{+}},
\end{equation}
we then see that $x(r_{+}, r_{-})$ and $y(r_{+}, r_{-})$ satisfy the equations
\begin{equation}\label{HT}
\left\{
  \begin{array}{ll}
\hat{\partial}_{+}y=\lambda_{+}\hat{\partial}_{+}x, \\[2pt]
 \hat{\partial}_{-}y=\lambda_{-}\hat{\partial}_{-}x,
  \end{array}
\right.
\end{equation}
where
$$
\hat{\partial}_{+}:=-\frac{\partial}{\partial r_{+}}\quad \mbox{and}\quad \hat{\partial}_{-}:=\frac{\partial}{\partial r_{-}}.
$$

From (\ref{HT}), we have $\hat{\partial}_{+}(\lambda_{-}\hat{\partial}_{-}x)=\hat{\partial}_{-}(\lambda_{+}\hat{\partial}_{+}x)$.  This deduces the following characteristic decompositions:
\begin{equation}\label{cdh}
\left\{
  \begin{array}{ll}
    (\lambda_{+}-\lambda_{-})\hat{\partial}_{-}\hat{\partial}_{+}x=\hat{\partial}_{+}\lambda_{-}\hat{\partial}_{-}x-
\hat{\partial}_{-}\lambda_{+}\hat{\partial}_{+}x,  \\[4pt]
     (\lambda_{+}-\lambda_{-})\hat{\partial}_{+}\hat{\partial}_{-}x=\hat{\partial}_{+}\lambda_{-}\hat{\partial}_{-}x-
\hat{\partial}_{-}\lambda_{+}\hat{\partial}_{+}x.
  \end{array}
\right.
\end{equation}

We compute
$\hat{\partial}_{\pm}\rho=\rho_x\hat{\partial}_{\pm}x+\rho_y\hat{\partial}_{\pm}y$. Thus, by (\ref{6606}) and (\ref{HT}) we have
\begin{equation}\label{53002}
\bar{\partial}_{+}\rho=\frac{\cos\alpha\hat{\partial}_{+}\rho}{\hat{\partial}_{+}x}=\frac{\rho\cos\alpha}{\sin (2A)\hat{\partial}_{+}x}\quad \mbox{and}\quad
\bar{\partial}_{-}\rho=\frac{\rho\cos\beta\hat{\partial}_{-}}{\hat{\partial}_{-}x}=\frac{\rho\cos\beta}{\sin (2A)\hat{\partial}_{-}x}.
\end{equation}

\subsection{Oblique waves in 2D steady supersonic flows past a compressible ramp}
In this part we discuss oblique waves of the 2D steady potential flow equations (\ref{42501}) for the van der Waals gas (\ref{8310a}).
Under assumption (\ref{81601}),
there exist $\tau_{1}^a$ and $\tau_{2}^a$, where $1<\tau_1^a<\tau_{1}^{i}<\tau_{2}^{i}<\tau_{2}^a$, such that
\begin{equation}\label{5901}
p'(\tau_{1}^a)=p'(\tau_{2}^{i})\quad \mbox{and}\quad p'(\tau_{1}^{i})=p'(\tau_{2}^a);
 \end{equation}
see Figure \ref{Figure9}.
\begin{figure}[htbp]
\begin{center}
\includegraphics[scale=0.48]{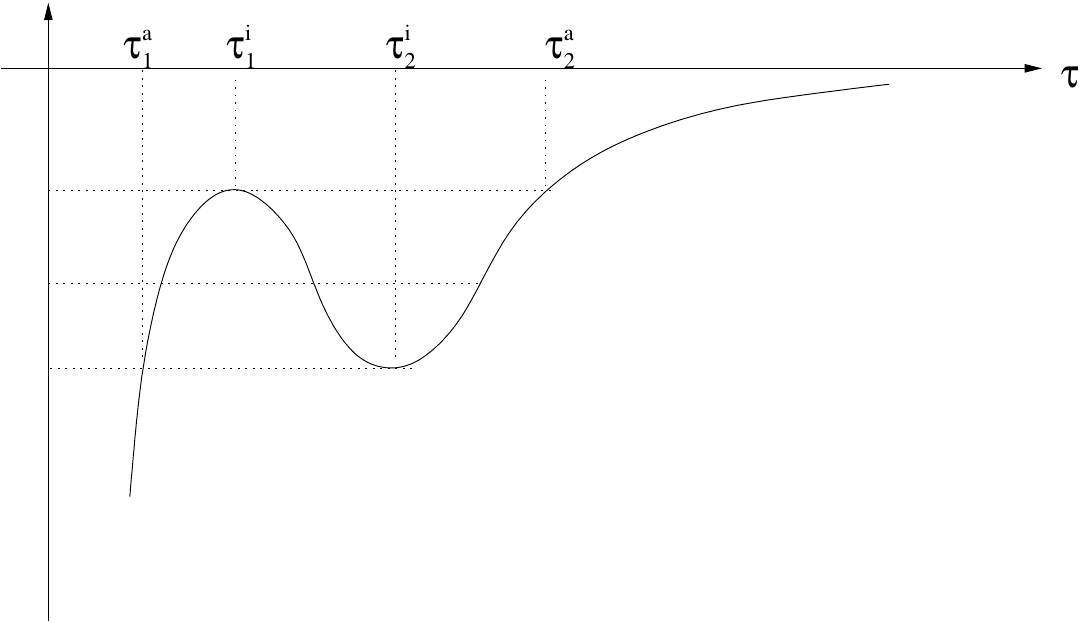}
\caption{\footnotesize The graph of $p'(\tau)$.}
\label{Figure9}
\end{center}
\end{figure}

\subsubsection{\bf Oblique shocks in 2D steady potential flows}
We now discuss the oblique shocks of the 2D steady potential flow equations
We assume the shock line in the $(x, y)$-plane to be straight and the states on both its sides to be constant.
We use subscripts `$f$' and `$b$' to denote the front side and the back side states  of the shock, respectively.

 We denote by $N$ and $L$
the components of $(u, v)$ normal and tangential to the shock direction, respectively, and $\phi$ the inclination angle of the shock front.
Then we have
\begin{equation}\label{62604}
L=u\cos \phi+v\sin \phi, \quad N=u\sin \phi- v\cos \phi.
\end{equation}
From (\ref{62604}) we have
\begin{equation}\label{52304}
u=N\sin\phi+L\cos\phi\quad \mbox{and}\quad v=L\sin\phi-N\cos\phi.
\end{equation}

The Rankine-Hugoniot conditions for the oblique shock can be written in the form:
\begin{equation}\label{RH}
\left\{
  \begin{array}{ll}
    m:=\rho_f N_f=\rho_b N_b, \\[4pt]
    L_{f}=L_b,\\[4pt]
   N_f^2+2h(\tau_f)= N_b^2+2h(\tau_b).
  \end{array}
\right.
\end{equation}

From the first and third relations of (\ref{RH}), we have
\begin{equation}\label{m}
m^2=-\frac{2h(\tau_f)-2h(\tau_b)}{\tau_f^2-\tau_b^2}.
\end{equation}

We confine ourselves to compressive shocks, i.e., $\tau_f>\tau_b$.
We also assume that the oblique shocks satisfy the
Liu's extended entropy condition (\cite{Liu1}):
\begin{equation}\label{Liu}
\frac{2h(\tau_f)-2h(\tau_b)}{\tau_f^2-\tau_b^2}<\frac{2h(\tau_f)-2h(\tau)}{\tau_f^2-\tau^2}\quad \mbox{for~~all}\quad \tau\in(\tau_b, \tau_f).
\end{equation}

\begin{prop}\label{prop1}
There exists a unique $\tau_{c}\in (\tau_1^i, +\infty)$ such that
\begin{equation}
p'(\tau_1^i)=\frac{2h(\tau_c)-2h(\tau_1^i)}{\tau_c^2-(\tau_1^i)^2}<p'(\tau_c).
\end{equation}
Moreover, $\tau_c>\tau_2^a$ and
\begin{equation}\label{72309}
p'(\tau_1^i)>\frac{2h(\tau)-2h(\tau_1^i)}{\tau^2-(\tau_1^i)^2}\quad \mbox{for~any}\quad \tau\in (\tau_1^i, \tau_c).
\end{equation}
\end{prop}
\begin{proof}
Firstly, by the Cauchy mean value theorem one knows that if there exists a $\tau_c>\tau_1^i$ such that $p'(\tau_1^i)=\frac{2h(\tau_c)-2h(\tau_1^i)}{\tau_c^2-(\tau_1^i)^2}$, then $\tau_c>\tau_2^a$ and $\frac{2h(\tau_c)-2h(\tau_1^i)}{\tau_c^2-(\tau_1^i)^2}<p'(\tau_c)$.

Set $f_1(\tau)=\frac{2h(\tau)-2h(\tau_1^i)}{\tau^2-(\tau_1^i)^2}-p'(\tau_1^i)$.
Then we have
$f_1'(\tau)=\frac{2\tau}{\tau^2-(\tau_1^i)^2}\big(p'(\tau)-\frac{2h(\tau)-2h(\tau_1^i)}{\tau^2-(\tau_1^i)^2}\big)$.
Moreover, by the Cauchy mean value theorem we have $p'(\tau)-\frac{2h(\tau)-2h(\tau_1^i)}{\tau^2-(\tau_1^i)^2}>0$ for $\tau>\tau_2^a$. Thus, we have $f_1'(\tau)>0$ for $\tau>\tau_2^a$.
Combining this with $\lim\limits_{\tau\rightarrow +\infty}f_1(\tau)=-p'(\tau_1^i)>0$ and $f_1(\tau_2^a)<0$, we know that there exits a unique $\tau_c>\tau_1^i$ such that $p'(\tau_1^i)=\frac{2h(\tau_c)-2h(\tau_1^i)}{\tau_c^2-(\tau_1^i)^2}$.

Set $g_1(\tau)=2h(\tau)-2h(\tau_1^i)-(\tau^2-(\tau_1^i)^2)p'(\tau_{1}^i)$. Then we have $g_1(\tau_c)=g_1(\tau_{1}^i)=0$ and $g_1'(\tau)=2\tau (p'(\tau)- p'(\tau_1^i))$.
Then by $\tau_{c}>\tau_2^i$, we have $g_1(\tau)<0$ for $\tau\in (\tau_{1}^i, \tau_c)$. This yields (\ref{72309}).
This completes the proof.
\end{proof}

\begin{prop}\label{prop2}
Suppose $\tau_f\in (\tau_c, +\infty)$. Then,
\begin{equation}\label{43012}
p'(\tau)<\frac{2h(\tau)-2h(\tau_f)}{\tau^2-\tau_f^2}<p'(\tau_f)\quad \mbox{for}\quad \tau\in (1, \tau_f).
\end{equation}
Moreover, for any  $\tau_b\in (1, \tau_f)$ there holds the Liu's extended entropy condition (\ref{Liu}).
\end{prop}
\begin{proof}
Firstly, by the Cauchy mean value theorem we have $\frac{2h(\tau)-2h(\tau_f)}{\tau^2-\tau_f^2}<p'(\tau_f)$ for $\tau\in (1, \tau_f)$.

Set
\begin{equation}\label{52106}
f_2(\tau)=\frac{2h(\tau_f)-2h(\tau)}{\tau_f^2-\tau^2}-p'(\tau).
\end{equation}
 Then we have
\begin{equation}\label{52102}
f_2'(\tau)=\frac{2\tau f(\tau)}{\tau_f^2-\tau^2}-p''(\tau).
\end{equation}
From the proof of Proposition \ref{prop1} we have $f_2(\tau_1^i)>0$.
Combining this with (\ref{52102}) one has $f_2(\tau)>0$ for $\tau<\tau_2^i$.
By the Cauchy mean value theorem we have $f_2(\tau_2^i)>0$. Combining this with
(\ref{52102}) one has $f_2(\tau)>0$ for $\tau_2^i<\tau<\tau_f$. This completes the proof for (\ref{43012}).

For any fixed $\tau_b\in (1, \tau_f)$,
we set
\begin{equation}\label{72310}
g_2(\tau)=(2h(\tau_f)-2h(\tau))(\tau_f^2-\tau_b^2)-(2h(\tau_f)-2h(\tau_b))(\tau_f^2-\tau^2).
\end{equation}
Then we have
$
g_2(\tau_f)=g_2(\tau_b)=0
$
and
\begin{equation}\label{72311}
\begin{aligned}
g_2'(\tau)&=2\tau(\tau_f^2-\tau_b^2)\left(-p'(\tau)+\frac{2h(\tau_f)-2h(\tau_b)}{\tau_f^2-\tau_b^2}\right)
\\
&=-2\tau(\tau_f^2-\tau_b^2)\left(p'(\tau)-\frac{2h(\tau)-2h(\tau_f)}{\tau^2-\tau_f^2}\right)-\frac{2\tau g_2(\tau)}{\tau_f^2-\tau^2}.
\end{aligned}
\end{equation}
Combining (\ref{72310}), (\ref{72311}) and (\ref{43012}), we have $g_2(\tau)>0$ for $\tau\in (\tau_b, \tau_f)$. This immediately imply (\ref{Liu}).
This completes the proof.
\end{proof}

\begin{prop}\label{prop3}
(Post-sonic shock)
For any $\tau_f\in (\tau_2^i, \tau_c)$, there exists a unique $\tau_{po}\in(\tau_1^i, \tau_2^i)$ such that when $\tau_b=\tau_{po}(\tau_f)$ there hold
\begin{equation}\label{430121}
p'(\tau_{b})=\frac{2h(\tau_{b})-2h(\tau_f)}{\tau_{b}^2-\tau_f^2}<p'(\tau_f)
\end{equation}
and
the Liu's extended entropy condition (\ref{Liu}).
Moreover, we have
\begin{equation}\label{72501}
\frac{{\rm d} \tau_{po}}{{\rm d}\tau_f}<0\quad \mbox{for}\quad \tau_f\in (\tau_2^i, \tau_c).
 \end{equation}
\end{prop}
\begin{proof}
Let $f_2(\tau)$ be defined in (\ref{52106}). By the Cauchy mean value theorem we have $f_2(\tau_2^i)>0$. By (\ref{72309}) we have $f_2(\tau_1^i)<0$.
 Therefore, by (\ref{52102}) we know that there exists a unique $\tau_{po}\in (\tau_1^i, \tau_2^i)$  such that $p'(\tau_{po})=\frac{2h(\tau_{po})-2h(\tau_f)}{\tau_{po}^2-\tau_f^2}$.
Since $\tau_{po}\in (\tau_1^i, \tau_2^i)$,  by the Cauchy mean value theorem we have $\frac{2h(\tau_{po})-2h(\tau_f)}{\tau_{po}^2-\tau_f^2}<p'(\tau_f)$.

Next, we prove that if $\tau_b=\tau_{po}$ then (\ref{Liu}) holds. Set $g_3(\tau)=2h(\tau_f)-2h(\tau)-(\tau_f^2-\tau^2)p'(\tau_{b})$. Then we have $g_3(\tau_f)=g_3(\tau_{b})=0$ and $g_3'(\tau)=2\tau (p'(\tau_{b})- p'(\tau))$.
Then by $\tau_{b}\in (\tau_1^i, \tau_2^i)$, we have $g_3(\tau)>0$ for $\tau\in (\tau_{b}, \tau_f)$. This yields (\ref{Liu}).

By a direct computation we have
$
p''(\tau_{po})(\tau_f^2-\tau_{po}^2)\frac{{\rm d} \tau_{po}}{{\rm d}\tau_f}=2\tau_f\big(p'(\tau_f)-p'(\tau_{po})\big)
$.
Then by (\ref{430121}) we have $\frac{{\rm d} \tau_{po}}{{\rm d}\tau_f}<0$.
This completes the proof.
\end{proof}


\begin{prop}
Suppose  $\tau_f\in (\tau_2^i, \tau_c)$. Then for any $\tau_b\in (\tau_{po}(\tau_f), \tau_f)$ there hold
$$
p'(\tau_{b})<\frac{2h(\tau_{b})-2h(\tau_f)}{\tau_{b}^2-\tau_f^2}<p'(\tau_f)
$$
and the Liu's extended entropy condition (\ref{Liu}).
\end{prop}
\begin{proof}
The proof is similar to that of Proposition \ref{prop2}, we omit the details.
\end{proof}

\begin{prop}\label{prop4}
(Pre-sonic shock) For any $\tau_f\in (\tau_1^i, \tau_2^i)$, there exists a unique $\tau_{pr}\in (1, \tau_f)$ such that
when $\tau_b=\tau_{pr}(\tau_f)$ there hold
\begin{equation}\label{430123}
p'(\tau_{b})<\frac{2h(\tau_{b})-2h(\tau_f)}{\tau_{b}^2-\tau_f^2}=p'(\tau_f)
\end{equation}
and the Liu's extended entropy condition (\ref{Liu}). Moreover, $\frac{{\rm d}\tau_{pr}}{{\rm d}\tau_f}<0$ for $\tau_f\in (\tau_1^i, \tau_2^i)$.
\end{prop}
\begin{proof}
Set $g_4(\tau)=2h(\tau_f)-2h(\tau)-p'(\tau_f)(\tau_f^2-\tau^2)$. Then we have $g_4(\tau_f)=0$ and $g_4'(\tau_f)=2\tau(p'(\tau_f)-p'(\tau))$.
Then by $\tau_f\in (\tau_1^i, \tau_2^i)$ we know that there exists a unique $\tau_{pr}\in (1, \tau_f)$ such that $g_4(\tau_{pr})=0$, and $\tau_{pr}<\tau_1^i$. This implies $\frac{2h(\tau_{pr})-2h(\tau_f)}{\tau_{pr}^2-\tau_f^2}=p'(\tau_f)$. Moreover, by the Cauchy mean value theorem we have
$\frac{2h(\tau_{pr})-2h(\tau_f)}{\tau_{pr}^2-\tau_f^2}>p'(\tau_{pr})$.

By a direct computation we have $2\tau(p'(\tau_f)-p'(\tau_{pr}))\frac{{\rm d}\tau_{pr}}{{\rm d}\tau_{f}}=p''(\tau_f)(\tau_f^2-\tau_{pr}^2)$. Then by (\ref{430123}) we have  $\frac{{\rm d}\tau_{pr}}{{\rm d}\tau_f}<0$ for $\tau_f\in (\tau_1^i, \tau_2^i)$.

Set $g_5(\tau)=2h(\tau_f)-2h(\tau)-(\tau_f^2-\tau^2)p'(\tau_{f})$. Then we have $g_5(\tau_f)=g_5(\tau_{b})=0$ and $g_5'(\tau)=2\tau (p'(\tau_{f})- p'(\tau))$.
Then by $\tau_{f}\in (\tau_1^i, \tau_2^i)$, we have $g_5(\tau)>0$ for $\tau\in (\tau_{b}, \tau_f)$. This yields  (\ref{Liu}).
This completes the proof.
\end{proof}

\begin{prop}\label{prop5}
Suppose $\tau_f\in (\tau_2^i, \tau_c)$. Then,
\begin{equation}\label{430125}
p'(\tau)<\frac{2h(\tau)-2h(\tau_f)}{\tau^2-\tau_f^2}<p'(\tau_f)\quad \mbox{for}\quad \tau\in\big(1, \tau_{pr}\big(\tau_{po}(\tau_f)\big)\big).
\end{equation}
Moreover, for any $\tau_b\in (1, \tau_{pr}(\tau_{po}(\tau_f)))$
there holds the Liu's extended condition (\ref{Liu}).
\end{prop}
\begin{proof}
We denote $\tau_{1}=\tau_{po}(\tau_f)$ and $\tau_{2}= \tau_{pr}(\tau_1)$. We have
$$
p'(\tau_2)<\frac{2h(\tau_{1})-2h(\tau_2)}{\tau_{1}^2-\tau_2^2}=
p'(\tau_{1})=\frac{2h(\tau_{1})-2h(\tau_f)}{\tau_{1}^2-\tau_f^2}<p'(\tau_f).
$$
Hence, we get
$$
p'(\tau_2)<\frac{2h(\tau_{f})-2h(\tau_2)}{\tau_{f}^2-\tau_s^2}
<p'(\tau_f).
$$

Set
\begin{equation}
f_3(\tau)=\frac{2h(\tau_f)-2h(\tau)}{\tau_f^2-\tau^2}-p'(\tau_f)\quad \mbox{and}\quad f_2(\tau)=\frac{2h(\tau_f)-2h(\tau)}{\tau_f^2-\tau^2}-p'(\tau).
\end{equation}
Then we have
$$
f_3(\tau_2)<0, \quad f_2(\tau_2)>0, \quad
f_3'(\tau)=\frac{2\tau f_2(\tau)}{\tau_f^2-\tau^2},\quad \mbox{and} \quad f_2'(\tau)=\frac{2\tau f_2(\tau)}{\tau_f^2-\tau^2}-p''(\tau).
$$
Combining these with $\tau_2<\tau_1^i$, we get
$$
f_3(\tau)<0\quad \mbox{and}\quad f_2(\tau)>0 \quad \mbox{for}\quad 1<\tau<\tau_2.
$$
This completes the proof for (\ref{430125}).

For any $\tau_b\in (1, \tau_2)$, let the function $g_{2}(\tau)$ be defined in (\ref{72310}).
Then we have
\begin{equation}\label{72401}
g_2(\tau_f)=g_2(\tau_b)=0, \quad g_2'(\tau_b)>0, \quad \mbox{and}\quad g_2'(\tau_f)<0.
\end{equation}

By Proposition \ref{prop3} we have $p'(\tau)-\frac{2h(\tau)-2h(\tau_f)}{\tau^2-\tau_f^2}<0$ for $\tau\in (\tau_1, \tau_f)$.
Then by (\ref{72311}) and $g_2(\tau_b)=0$ we have $g_{2}(\tau)>0$ for $\tau\in [\tau_1, \tau_f]$ and $g_2'(\tau_1)<0$.

Suppose that there exists a $\tau\in (\tau_b, \tau_f)$ such that $g_2(\tau)=0$. Then by (\ref{72311}), (\ref{430125}), and  (\ref{72401}), there exist at least three points in $(\tau_b, \tau_1)$ such that $g_2'=0$ at these points. This is impossible, since $\tau_1<\tau_2^i$.
Then we have $g_2(\tau)>0$ for $\tau\in (\tau_b, \tau_f)$. This yields  (\ref{Liu}).
This completes the proof of the proposition.
\end{proof}

\subsubsection{\bf Oblique wave curve}
We consider (\ref{42501}) with the following initial-boundary conditions:
\begin{equation}\label{RBDa}
\left\{
  \begin{array}{ll}
    (u, v)(0, y)=(u_0, 0), & \hbox{$y>0$;} \\[4pt]
 v=u\tan\theta_{w}, & \hbox{$y=x\tan\theta_w$, $x>0$,}
  \end{array}
\right.
\end{equation}
where $\theta_w\in(0, \frac{\pi}{2})$ is the ramp angle.
Let $c_0=\ddot{c}(u_0)$ and $\tau_0=\ddot{\tau}(u_0)$.
We look for a self-similar solution to the problem (\ref{42501}), (\ref{RBDa}) for  $u_0>c_0$.
We divide the discussion into the following four cases: (1) $\tau_0>\tau_c$; (2) $\tau_2^i<\tau_0\leq \tau_c$; (3) $\tau_1^i<\tau_0<\tau_2^i$; (4) $1<\tau_0<\tau_1^i$.
We only discuss the case of $\tau_2^i<\tau_0\leq \tau_c$, since the other cases can be discussed similarly.

\begin{figure}[htbp]
\begin{center}
\includegraphics[scale=0.5]{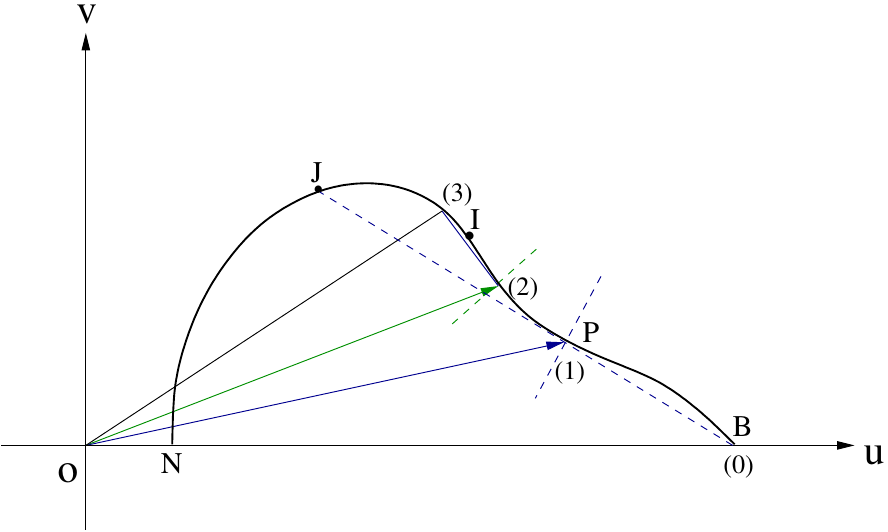}\qquad\quad  \includegraphics[scale=0.55]{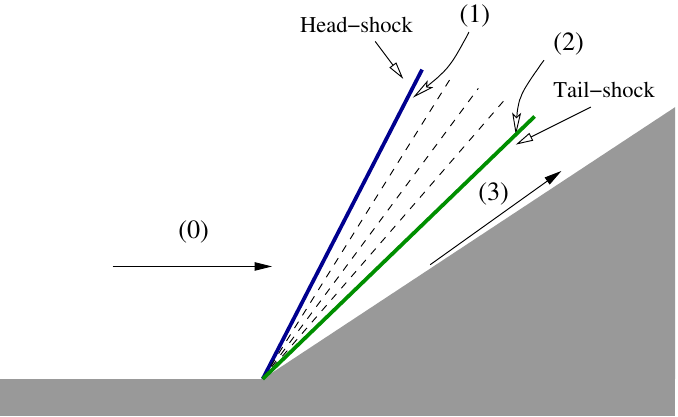}
\caption{\footnotesize Left: an oblique wave curve through $(u_0, 0)$ in the $(u, v)$-plane; right: an oblique shock-fan-shock composite wave.}
\label{Figure10}
\end{center}
\end{figure}

\noindent{\bf (i). Oblique shock wave curve (I).}
We first look for a self-similar solution with a single oblique shock. We denote by $(u_b^h, v_b^h)$ the velocity of the flow downstream of the shock and by $\phi$ the inclination angle of the shock.
We denote by $\tau_b^h$ the specific volume of the flow downstream of the shock. We shall show that $\phi$ and $(u_b^h, v_b^h)$ are functions of $\tau_b^h \in [\tau_{po}(\tau_0), \tau_0)$, where the function $\tau_{po}$ is defined in Proposition \ref{prop3}.

Let
$$
N_f^h=\tau_0\sqrt{-\frac{2h(\tau_b^h)-2h(\tau_0)}{(\tau_b^h)^2-\tau_0^2}}.
$$
We also assume
\begin{equation}\label{8801}
u_0>N_f^h\quad \mbox{for}\quad  \tau_b^h \in [\tau_{po}(\tau_0), \tau_0).
\end{equation}
Then we have
$$
\phi=\arcsin\Big(\frac{N_f^h}{u_0}\Big).
$$

In view of (\ref{52304}), we have
\begin{equation}\label{8802}
u_b^h=N_b^h\sin\phi+L_b^h\cos\phi \quad \mbox{and}\quad v_b^h=-N_b^h\cos\phi+L_b^h\sin\phi,
\end{equation}
where
$$
N_b^h= \frac{\tau_b^hN_f^h}{\tau_0}\quad \mbox{and} \quad L_f^h=L_b^h=u_0\cos\phi.
$$
So,
under the condition (\ref{8801}), we can define the one-parametric curve $(u, v)=(u_b^h, v_b^h)(\tau_b^h)$, $\tau_b^h\in [\tau_{po}(\tau_0), \tau_0)$ in the $(u, v)$-plane.
This curve is a portion of the shock polar of the state $(u_0, 0)$; see the arc $\wideparen{\mathrm{BP}}$ in Figure \ref{Figure10} (left). That is to say, for any  $(u, v)\in \wideparen{\mathrm{BP}}$, the state $(u, v)$ can be connected to the state $(u_0, 0)$ on the right by an oblique shock.

From Proposition \ref{prop3}
we also have
\begin{equation}\label{8310b}
q_{po}>N_{b}^{h}(\tau_{po}(\tau_0))=\hat{c}(\tau_{po}(\tau_0)),
\end{equation}
where
\begin{equation}\label{8803}
q_{po}=\sqrt{(u_b^{h})^2(\tau_{po}(\tau_0))+(v_b^h)^2(\tau_{po}(\tau_0))}.
\end{equation}




\vskip 4pt

\noindent{\bf (ii). Oblique shock-fan composite wave curve.}
The incoming flow $(u_0, 0)$ can also be deflected by an oblique shock-fan composite wave consisting of a post-sonic shock and a centered simple wave.
We now construct the oblique shock-fan composite wave curve of the state $(u_0, 0)$.

Set
\begin{equation}\label{6602}
(u_{po}, v_{po})=(u_b^h(\tau_{po}(\tau_0)),v_b^h(\tau_{po}(\tau_0)))\quad \mbox{and}\quad
\sigma_{po}=\arctan\Big(\frac{v_{po}}{u_{po}}\Big).
\end{equation}
We then define the functions
$$
\hat{\sigma}(\tau):=\sigma_{po}-\int^{\hat{q}(\tau)}_{q_{po}}\frac{\sqrt{q^{2}-c^{2}}}{qc}{\rm d}q \quad \mbox{and}\quad \hat{\alpha}(\tau)
:=\hat{\sigma}(\tau)+\arcsin\Big(\frac{\hat{c}(\tau)}{\hat{q}(\tau)}\Big), \quad \tau_1^i\leq \tau<\tau_{po}(\tau_0),
$$
where $\hat{q}(\tau)$ is determined by the Bernoulli's law
$$
\frac{1}{2}\hat{q}^2(\tau)+h(\tau)=\frac{1}{2}q_{po}^2+h(\tau_{po}(\tau_0)).
$$

A direct computation yields
\begin{equation}\label{9501a}
\hat{\sigma}'(\tau)<0\quad \mbox{and}\quad \hat{\alpha}'(\tau)>0\quad \mbox{for} \quad \tau_1^i\leq \tau<\tau_{po}(\tau_0).
\end{equation}
(Remark: the second inequality of (\ref{9501a}) can also be proved by (\ref{alphar}), since $r_{+}(\hat{\sigma}(\tau), \hat{q}(\tau))=\mbox{Const.}$)
This implies
$$
\sigma_{po}<\hat{\sigma}(\tau)<\hat{\alpha}(\tau)<\phi_{po} \quad \mbox{for} \quad \tau_1^i\leq \tau<\tau_{po}(\tau_0),
$$
where $$\phi_{po}=\sigma_{po}+\arcsin \left(\frac{\ddot{c}(q_{po})}{q_{po}}\right).$$
By the definition of $\tau_{po}(\tau_0)$ we actually have $\phi_{po}=\arcsin\big(\frac{N_f^h(\tau_{po}(\tau_0))}{u_0}\big)$.

Let $\tau=\hat{\tau}(\alpha)$, $\hat{\alpha}(\tau_1^i)<\alpha<\phi_{po}$ be the inverse function of $\alpha=\hat{\alpha}(\tau)$, $\tau_1^i\leq \tau<\tau_{po}(\tau_0)$.
We then define
 \begin{equation}\label{72403}
 (\tilde{u}, \tilde{v})(\theta):= (\hat{u}, \hat{v})(\hat{\tau}(\theta)), \quad \hat{\alpha}(\tau_1^i)<\theta<\phi_{po}.
 \end{equation}
Then
$$
(u, v)=(\tilde{u}, \tilde{v})(\arctan(y/x)), \quad\hat{\alpha}(\tau_1^i)<\arctan(y/x)<\phi_{po}
$$
is actually a centered simple wave solution of  (\ref{42501}) on the triangle region $\{(x, y)\mid (x,y)=(r\cos\theta, r\sin\theta), r>0, \hat{\alpha}(\tau_1^i)<\theta<\phi_{po}\}$.

We define the functions
$$
\hat{u}(\tau):=(\hat{q}\cos\hat{\sigma})(\tau)\quad \mbox{and}\quad \hat{v}(\tau):=(\hat{q}\sin\hat{\sigma})(\tau), \quad \tau_1^i\leq \tau<\tau_{po}(\tau_0).
$$
We call the one-parametric curve
$(u, v)=(\hat{u}, \hat{v})(\tau)$,  $\tau_1^i\leq \tau<\tau_{po}(\tau_0)$
the shock-fan composite wave curve of the state $(u_0, 0)$; see the arc $\wideparen{\mathrm{PI}}$ in Figure \ref{Figure10} (left).
This curve
is also a $\Gamma_{-}$ characteristic issued from the point $\mathrm{P}=(u_{po}, v_{po})$
in the hodograph plane, i.e.,
\begin{equation}\label{9010a}
r_{+}(u, v)=\mbox{Const.}, \quad (u, v)\in \wideparen{\mathrm{PI}};
\end{equation}
see Courant $\&$ Friedrichs \cite{CaF}.
For any  $(u, v)\in \wideparen{\mathrm{PI}}$, the state $(u, v)$ can be connected to the state $(u_0, 0)$ on the right by an oblique shock-fan composite wave.

By a direct computation we also have
\begin{equation}\label{90901b}
\hat{u}'(\tau)\cos(\hat{\alpha}(\tau))+\hat{v}'(\tau)\sin(\hat{\alpha}(\tau))=0\quad \mbox{and}\quad
\hat{u}(\tau)\sin(\hat{\alpha}(\tau))-\hat{v}(\tau)\cos(\hat{\alpha}(\tau))=\hat{c}(\tau)
\end{equation}
for $\tau_1^i\leq \tau<\tau_{po}(\tau_0)$.


\vskip 4pt

\noindent{\bf (iii). Oblique shock-fan-shock composite wave curve.}
The incoming flow $(u_0, 0)$ can also  be deflected by an oblique shock-fan-fan composite wave consisting of a post-sonic head shock, a centered simple wave, and a pre-sonic tail shock.
The inclination angle and back side state of the head shock are $\phi_{po}$ and $(u_{po}, v_{po})$, respectively.

Assume that the specific volume on the front side of the tail shock is $\tau_{f}^t$, where $\tau_{f}^t\in(\tau_1^i, \tau_{po}(\tau_0))$. Then the inclination angle
of the tail shock is $\hat{\alpha}(\tau_f^t)$,
 the state on the front side of the tail shock is $(\hat{u}, \hat{v})(\tau_f^t)$, and the specific volume on the back side of the tail shock is $\tau_{pr}(\tau_f^t)$, where the function $\tau_{pr}$ is defined in Proposition \ref{prop4}.

We denote by $({u}_b^t, {v}_b^t)$ the state on the back side of the tail shock.
Then we have
\begin{equation}\label{72405}
\left\{
  \begin{array}{ll}
    {u}_b^t(\tau_f^t)={N}_b^t(\tau_f^t)\sin\hat{\alpha}(\tau_f^t)+{L}_b^t(\tau_f^t)\cos\hat{\alpha}(\tau_f^t), \\[4pt]
{v}_b^t(\tau_f^t)={L}_b^t(\tau_f^t)\sin\hat{\alpha}(\tau_f^t)-{N}_b^t(\tau_f^t)\cos\hat{\alpha}(\tau_f^t),
  \end{array}
\right.
\end{equation}
where $N_b^t(\tau_f^t)=\frac{\tau_{pr}(\tau_f^t)N_f^t(\tau_f^t)}{\tau_f^t}$, $L_b^t(\tau_f^t)=L_f^t(\tau_f^t)$,
$N_f^t(\tau_f^t)=\hat{u}(\tau_f^t)\sin\hat{\alpha}(\tau_f^t)-\hat{v}(\tau_f^t)\cos\hat{\alpha}(\tau_f^t)$, and
$L_f^t(\tau_f^t)=\hat{u}(\tau_f^t)\cos\hat{\alpha}(\tau_f^t)+\hat{v}(\tau_f^t)\sin\hat{\alpha}(\tau_f^t)$

We call the one-parametric curve $u={u}_b^t(\tau_f^t)$,
$v={v}_b^t(\tau_f^t)$, $\tau_{f}^t\in(\tau_1^i, \tau_{po}(\tau_0))$ the shock-fan-shock composite wave curve of the state $(u_0, 0)$; see the arc $\wideparen{\mathrm{IJ}}$ in Figure \ref{Figure10} (left).
That is to say, for any  $(u, v)\in \wideparen{\mathrm{IJ}}$, the state $(u, v)$ can be connected to the state $(u_0, 0)$ on the right by an oblique shock-fan-shock composite wave.

\begin{rem}\label{9010b}
For the needs of the following discussion, we represent the curve $\wideparen{\mathrm{IJ}}$ by
$$
H(u, v)=0.
$$
This implies that if the front side state of an oblique pre-sonic shock lies in the arc $\wideparen{\mathrm{PI}}$ then the back side state satisfies $H(u, v)=0$.
\end{rem}

\vskip 4pt

\noindent{\bf (iv). Oblique shock wave curve (II).}
We now discuss the equations $(u, v)=(u_b^h, v_b^h)(\tau_b^h)$ for $\tau_b^h<\tau_{pr}(\tau_{po}(\tau_0))$.
There exists a $1<\tau_{n}<\tau_{pr}(\tau_{po}(\tau_0))$ such that $v_b^{h}(\tau_n)=0$ and $v_b^h(\tau_b^h)>0$ for $\tau_n<\tau_b^h<\tau_{pr}(\tau_{po}(\tau_0))$. The one-parametric curve $(u, v)=(u_b^h, v_b^h)(\tau_b^h)$, $\tau_n\leq \tau_b^h<\tau_{pr}(\tau_{po}(\tau_0))$ is the other portion of the shock polar; see the arc $\wideparen{\mathrm{JN}}$  in Figure \ref{Figure10} (left).

\subsubsection{\bf Self-similar shock-fan-shock composite wave solution}
We define
$$
{\sigma}_b^t(\tau_f^t):=\arctan\left(\frac{{v}_b^t(\tau_f^t)}{{u}_b^t(\tau_f^t)}\right), \quad \tau_f^t\in [\tau_1^i, \tau_{po}(\tau_0)].
$$
Let
\begin{equation}\label{8310c}
{\sigma}_m=\min\limits_{\tau_f^t\in [\tau_1^i, \tau_{po}(\tau_0)]}{\sigma}_b^t(\tau_f^t)\quad \mbox{and}\quad
{\sigma}_{_M}=\max\limits_{\tau_f^t\in [\tau_1^i, \tau_{po}(\tau_0)]}{\sigma}_b^t(\tau_f^t).
\end{equation}
Then when ${\sigma}_m< \theta_w< {\sigma}_{_M}$, there exists a $\tau_w\in (\tau_1^i, \tau_{po}(\tau_0))$
 such that ${v}_b^t(\tau_w)={u}_b^t(\tau_w)\tan\theta_w$, and the problem (\ref{42501}), (\ref{RBDa}) admits a self-similar shock-fan-shock composite wave solution with the form
\begin{equation}\label{71901}
(u, v)=\left\{
         \begin{array}{ll}
           (u_0, 0), & \hbox{$\phi_{po}<\theta<\frac{\pi}{2}$;} \\[3pt]
        (\tilde{u}, \tilde{v})(\theta), & \hbox{$\phi_{pr}<\theta<\phi_{po}$;}\\[3pt]
 ({u}_b^t(\tau_w), {v}_b^t(\tau_w)), & \hbox{$\theta_{w}<\theta<\phi_{pr}$,}
         \end{array}
       \right.
\end{equation}
where $\theta=\arctan(y/x)$ and $\phi_{pr}=\hat{\alpha}(\tau_w)$; see Figure \ref{Figure10} (right).
In Figure \ref{Figure10},  the symbols $(0)$, $(1)$, $(2)$ and $(3)$ represent the states $(u_0, 0)$, $(u_{po}, v_{po})$, $(\hat{u}(\tau_w), \hat{u}(\tau_w))$ and $({u}_b^t(\tau_w), {v}_b^t(\tau_w))$, respectively.

\vskip 4pt

\subsection{Stability of the shock-fan-shock composite wave: case I}

\subsubsection{\bf Initial-boundary value problem (IBVP)}
We now discuss the stability of the shock-fan-shock composite wave solution (\ref{71901}).
We first consider the scenario of a uniform incoming flow over a curved ramp.
So, we consider (\ref{42501}) with the following initial-boundary conditions:
\begin{equation}\label{RBDD}
\left\{
  \begin{array}{ll}
    (u, v)(0, y)=(u_0, 0), & \hbox{$y>0$;} \\[4pt]
 \displaystyle \left(\frac{v}{u}\right)(x,y)=w'(x), & \hbox{$y=w(x)$, $x>0$,}
  \end{array}
\right.
\end{equation}
where $w(x)\in C^2[0,+\infty)$ and $w(0)=0$. Set $\theta_w=\arctan(w'(0))$. We assume that $u_0>c_0$, $\tau_2^i<\tau_0< \tau_c$, and ${\sigma}_m<\theta_w<{\sigma}_{_M}$, where the constants $\sigma_m$ and $\sigma_{_M}$ are defined in (\ref{8310c}).

If $w''(x)\equiv 0$ for $x>0$, then the IBVP (\ref{42501}), (\ref{RBDD}) has a self-similar shock-fan-shock composite wave solution with the form
(\ref{71901}). 
We next discuss the case of $w''(0)\neq0$.
We make the following assumptions:
\begin{description}
  \item[(H1)] the flow downstream of the pre-sonic shock is also supersonic, i.e.,
\begin{equation}\label{71902}
({u}_b^t)^2(\tau_w)+({v}_b^t)^2(\tau_w)>\hat{c}^2(\tau_{pr}(\tau_w));
\end{equation}
  \item[(H2)] the angle between the tangent vector $-(({u}_b^t)'(\tau_w), ({v}_b^t)'(\tau_w))$ of the fan-shock-fan wave curve $\wideparen{\mathrm{IJ}}$ at the point $({u}_b^t(\tau_w), {v}_b^t(\tau_w))$ and the vector $({u}_b^t(\tau_w), {v}_b^t(\tau_w))$ is an acute angle; see the angle $\nu$ in Figure \ref{Superbzt}.
\end{description}

\begin{figure}[htbp]
\begin{center}
\includegraphics[scale=0.48]{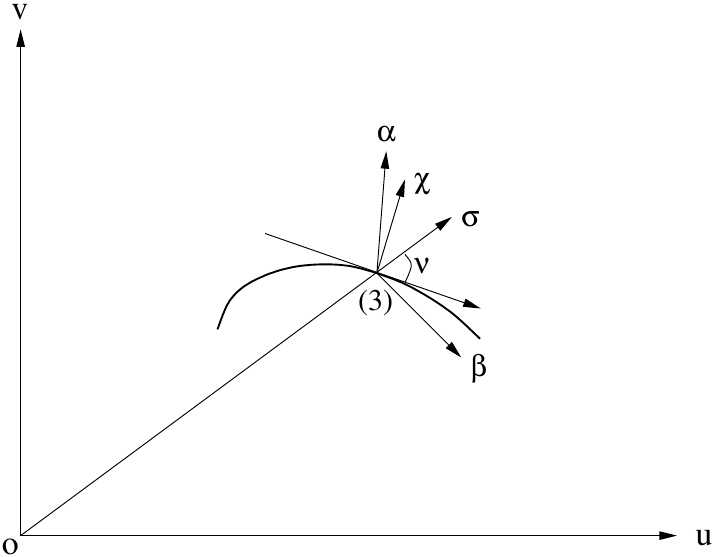}
\caption{\footnotesize The normal and tangent vectors of the arc $\wideparen{\mathrm{IJ}}$ at the point $(3)=({u}_b^t(\tau_w), {v}_b^t(\tau_w))$.}
\label{Superbzt}
\end{center}
\end{figure}

\begin{rem}\label{72407}
It is easy to see $({u}_b^t(\tau_1^i), {v}_b^t(\tau_1^i))=(\hat{u}(\tau_1^i), \hat{v}(\tau_1^i))$.
By (\ref{9501a}) we have $\hat{q}(\tau_1^i)>\ddot{c}(\hat{q}(\tau_1^i))$, and hence $({u}_b^t)^2(\tau_1^i)+({v}_b^t)^2(\tau_1^i)>\hat{c}^2(\tau_{pr}(\tau_1^i))$.
 Using the Rankine-Hugoniot conditions, we can see that the wave curve $\wideparen{\mathrm{IJ}}\cup \wideparen{\mathrm{PI}}$ is at least $C^1$ smooth at the point $\mathrm{I}=(\hat{u}(\tau_1^i), \hat{v}(\tau_1^i))$. So, the angle between the tangent vector $-(({u}_b^t)'(\tau_1^i), ({v}_b^t)'(\tau_1^i))$ of the fan-shock-fan wave curve $\wideparen{\mathrm{IJ}}$ at the point $\mathrm{I}$ and the vector $({u}_b^t(\tau_i), {v}_b^t(\tau_i))$ is equal to $\frac{\pi}{2}-(\hat{\alpha}-\hat{\sigma})(\tau_1^i)<\frac{\pi}{2}$.
So, when $\theta_w$ is close to ${\sigma}_b^t(\tau_1^i)$,
assumptions {\bf (H1)--(H2)} hold.
\end{rem}

The main result is stated as follows.
\begin{thm}\label{Thm39}
Assume $u_0>c_0$, $\tau_2^i<\tau_0< \tau_c$, ${\sigma}_m<\theta_w<{\sigma}_{_M}$, and $w''(0)\neq0$.
Assume as well that assumptions {\bf (H1)--(H2)} hold.
Then there is a small $\delta>0$ such that the IBVP (\ref{42501}), (\ref{RBDD}) admits a piecewise smooth solution on the domain
$\{(x, y)\mid w(x)<y<\delta\tan\phi_{po}, 0<x<\delta\}$. The solution contains a pre-sonic head shock, a simple wave, and a tail shock. Moreover, if $w''(0)>0$ then the tail shock is transonic; if  $w''(0)<0$ then the tail shock is pre-sonic. See Figure \ref{Figure11}.
\end{thm}

\begin{figure}[htbp]
\begin{center}
\includegraphics[scale=0.65]{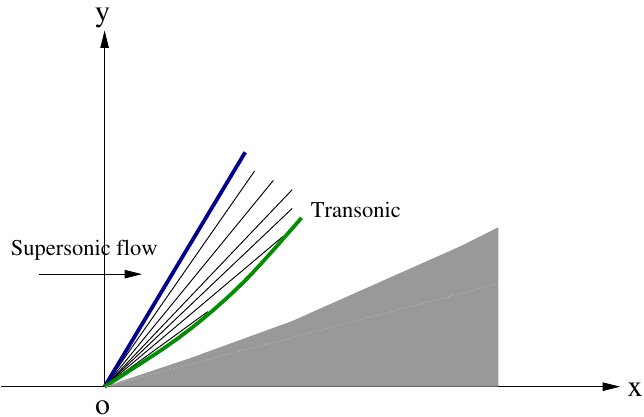}\qquad \quad\includegraphics[scale=0.63]{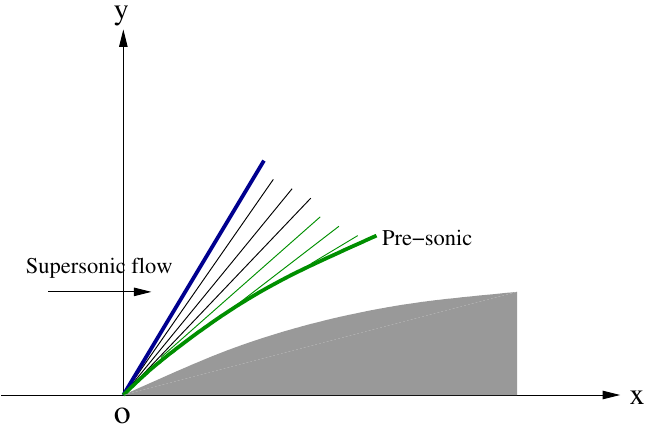} 
\caption{\footnotesize Wave structure of the solution to the IBVP (\ref{42501}), (\ref{RBDD}).}
\label{Figure11}
\end{center}
\end{figure}


\subsubsection{\bf Curvature of the tail shock}
In fact, the head shock and fan wave of the composite wave will not be perturbed by the curved ramp,
while the tail shock will be affected. When $x>0$, there are the following two cases about the tail shock:
\begin{enumerate}
  \item The tail shock enters into the fan wave region $$\Xi_1=\big\{(x, y)\mid x\tan\phi_{pr}\leq y\leq x\tan\phi_{po}, x>0\big\}$$ and becomes a transonic shock, and the flow upstream of the tail shock is the fan wave $(u,v)=(\tilde{u}, \tilde{v})(\arctan(y/x))$, where the function $(\tilde{u}, \tilde{v})$ is defined in (\ref{72403}); see Figure \ref{Figure11} (left).
  \item The tail shock does not enter into the fan wave region $\Xi_1$ and the flow upstream of the tail shock is also a priori unknown; see Figure \ref{Figure11} (right).
\end{enumerate}
The formulation of the boundary conditions on the tail shock for the above two cases are different. In order to determine which case will happen near the origin, we shall compute the curvature of the tail shock at the origin.

We use $\phi$ and the subscript ``$f$" to denote the inclination angle and the front side state of the tail shock, respectively.
The Rankine-Hugoniot conditions along the tail shock are
\begin{equation}\label{52401a}
\left\{
  \begin{array}{ll}
    (v -v_f)\sin\phi+(u -u_f)\cos\phi=0, \\[4pt]
    (\rho  u -\rho_fu_f)\sin\phi-  (\rho  v -\rho_fv_f)\cos\phi=0,
  \end{array}
\right.
\end{equation}
where we still use $(u, v, \rho)$ to denote the flow on the back side of the tail shock.
Eliminating $\phi$ from (\ref{52401a}), we get
\begin{equation}\label{Polar}
G(u,v, u_f, v_f)= (\rho  u -\rho_fu_f)(u -u_f)+  (\rho  v -\rho_fv_f)(v -v_f)=0.
\end{equation}

In terms of the Riemann invariants, (\ref{Polar}) can be written in the form
\begin{equation}\label{80901}
F(r_{+}, r_{-}, r_{+}^{f}, r_{-}^{f})=0,
\end{equation}
where $r_{\pm}^f=r_{\pm}(u_f, v_f)$.


By the first relation of (\ref{52401a}) we have that along the tail shock,
\begin{equation}\label{8804}
(\bar{\partial}_{_\Gamma}v -\bar{\partial}_{_\Gamma}v_f)\sin\phi+(v -v_f)\cos\phi\bar{\partial}_{_\Gamma}\phi+
(\bar{\partial}_{_\Gamma}u -\bar{\partial}_{_\Gamma}u_f)\cos\phi-(u -u_f)\sin\phi\bar{\partial}_{_\Gamma}\phi=0,
\end{equation}
where $\bar{\partial}_{_\Gamma}$ denotes the directional derivative along the shock.
From example,
if the shock can be represented by a function $y=y_s(x)$, then $\phi=\arctan \big(y_s'(x)\big)$,
\begin{equation}\label{20101}
 \bar{\partial}_{_\Gamma}u=\cos\phi\frac{{\rm d}u(x, y_s(x))}{{\rm d}x}\quad \mbox{and}\quad \bar{\partial}_{_\Gamma}v=\cos\phi\frac{{\rm d}v(x, y_s(x))}{{\rm d}x}
\end{equation}
along the shock.
If the downstream/upstream flow, denoted by $(u_{b/f}, v_{b/f})$ for convenience, is $C^1$ smooth up to the shock boundary, then we have
$$
\bar{\partial}_{_\Gamma}(u_{b/f}, v_{b/f})=\cos\phi\partial_{x}(u_{b/f}, v_{b/f})+\sin\phi\partial_{y}(u_{b/f}, v_{b/f})
$$
along the shock.

Since the supersonic flow in a region adjacent to a straight characteristic must be a simple wave,
for both cases the flow upstream of the tail shock is a simple wave, i.e.,
\begin{equation}\label{6601}
r_{+}(u_f, v_f):=\arctan\Big(\frac{v_f}{u_f}\Big)+\int_{q_{po}}^{q_f}\frac{\sqrt{q^{2}-c^{2}}}{qc}{\rm d}q\equiv \sigma_{po},
\end{equation}
where $q_f=\sqrt{u_f^2+v_f^2}$, and the constants $q_{po}$ and $\sigma_{po}$ are defined in (\ref{8803}) and (\ref{6602}), respectively.
 Therefore, by (\ref{6603}) we have
$$
\bar{\partial}_{_\Gamma}u_f=-\frac{q_f\sin\alpha_f}{2\cos A_f}\bar{\partial}_{_\Gamma}r_{-}^{f}\quad \mbox{and} \quad
\bar{\partial}_{_\Gamma}v_f=\frac{q_f\cos\alpha_f}{2\cos A_f}\bar{\partial}_{_\Gamma}r_{-}^{f},
$$
where $r_{-}^f:=\arctan(\frac{v_f}{u_f})-\int_{q_{po}}^{q_f}\frac{\sqrt{q^{2}-c^{2}}}{qc}{\rm d}q$.
Inserting these into (\ref{8804}), we get that along the tail shock,
\begin{equation}\label{52402}
\sin\phi\bar{\partial}_{_\Gamma}v +\cos\phi\bar{\partial}_{_\Gamma}u
+\big[(v -v_f)\cos\phi-(u -u_f)\sin\phi\big]\bar{\partial}_{_\Gamma}\phi
+\frac{q_f\sin(\alpha_f-\phi)}{2\cos A_f}\bar{\partial}_{_\Gamma}r_{-}^{f}=0.
\end{equation}

We decompose $\bar{\partial}_{_\Gamma}$ in two characteristic directions
\begin{equation}\label{72103}
\bar{\partial}_{_\Gamma}=t_{+}\bar{\partial}_{+}+t_{-}\bar{\partial}_{-},
\end{equation}
where
\begin{equation}\label{80601}
t_{+}=\frac{\sin(\phi-\beta )}{\sin(2A )}\quad \mbox{and}\quad t_{-}=\frac{\sin(\alpha -\phi)}{\sin(2A )}.
\end{equation}
Then by (\ref{11a}) and (\ref{72804a}) we have
\begin{equation}
\bar{\partial}_{_\Gamma}u =\kappa (t_{+}\sin\beta \bar{\partial}_{+}c- t_{-}\sin\alpha \bar{\partial}_{-}c)
\quad \mbox{and}\quad  \bar{\partial}_{_\Gamma}v =-\kappa (t_{+}\cos\beta \bar{\partial}_{+}c- t_{-}\cos\alpha \bar{\partial}_{-}c).
\end{equation}
Hence, (\ref{52402}) can be written as
\begin{equation}\label{52501}
\kappa t_{+}\sin(\beta-\phi)\bar{\partial}_{+}c+
\kappa t_{-}\sin(\phi-\alpha)\bar{\partial}_{-}c+(N_f-N)\bar{\partial}_{_\Gamma}\phi=\mathcal{L}_1\bar{\partial}_{_\Gamma}r_{-}^{f},
\end{equation}
where $\mathcal{L}_1=-\frac{q_f\sin(\alpha_f-\phi)}{2\cos A_f}$.

From (\ref{Polar}) we have that along the tail shock,
\begin{equation}\label{6609}
\begin{aligned}
&G_{u}\bar{\partial}_{_\Gamma}u+G_v\bar{\partial}_{_\Gamma}v-(\rho u-\rho_fu_f)\bar{\partial}_{_\Gamma}u_f-
(\rho v-\rho_fv_f)\bar{\partial}_{_\Gamma}v_f\\&\qquad\qquad\qquad -(u-u_f)\bar{\partial}_{_\Gamma}(\rho_fu_f)-(v-v_f)\bar{\partial}_{_\Gamma}(\rho_fv_f)=0,
\end{aligned}
\end{equation}
where
$$
G_u=\Big(\rho-\frac{\rho u^2}{c^2}\Big)(u-u_f)+(\rho u-\rho_fu_f)-v(v-v_f)\frac{\rho u}{c^2}
$$
and
$$
G_v=\Big(\rho-\frac{\rho v^2}{c^2}\Big)(v-v_f)+(\rho v-\rho_fv_f)-u(u-u_f)\frac{\rho v}{c^2}.
$$

\begin{rem}
The vector $(G_u, G_v)$ is actually a normal vector of the shock polar of the state $(u_f, v_f)$ at the point $(u, v)$.
\end{rem}

Using (\ref{6603}), (\ref{52304}), (\ref{RH}) we have
\begin{equation}\label{6608}
\begin{aligned}
&(\rho u-\rho_fu_f)\bar{\partial}_{_\Gamma}u_f+
(\rho v-\rho_fv_f)\bar{\partial}_{_\Gamma}v_f\\=&
\big[(\rho v-\rho_fv_f)\cos\alpha_f-(\rho u-\rho_fu_f)\sin\alpha_f\big]\frac{q_f\bar{\partial}_{_\Gamma}r_{-}^{f}}{2\cos A_f}
\\=&L_f(\rho-\rho_f)\sin(\phi-\alpha_f)\frac{q_f\bar{\partial}_{_\Gamma}r_{-}^{f}}{2\cos A_f}.
\end{aligned}
\end{equation}

From (\ref{6603}) and (\ref{6606}) we have
\begin{equation}
\bar{\partial}_{_\Gamma}(\rho_fu_f)=\left(\frac{\rho_fu_f}{\sin(2A_f)}-\frac{\rho_fq_f\sin\alpha_f}{2\cos A_f}\right)\bar{\partial}_{_\Gamma}r_{-}^{f}
\end{equation}
and
\begin{equation}
\bar{\partial}_{_\Gamma}(\rho_fv_f)=\left(\frac{\rho_fv_f}{\sin(2A_f)}+\frac{\rho_fq_f\cos\alpha_f}{2\cos A_f}\right)\bar{\partial}_{_\Gamma}r_{-}^{f}.
\end{equation}
Thus, we get
\begin{equation}\label{6607}
\begin{aligned}
&(u-u_f)\bar{\partial}_{_\Gamma}(\rho_fu_f)+(v-v_f)\bar{\partial}_{_\Gamma}(\rho_fv_f)\\
~=~&(N-N_f)\Big(\sin\phi\bar{\partial}_{_\Gamma}(\rho_fu_f)-\cos\phi\bar{\partial}_{_\Gamma}(\rho_fv_f)\Big)\\~=~&
(N-N_f)\left(\frac{\rho_fu_f\sin\phi}{\sin(2A_f)}-\frac{\rho_fq_f\sin\alpha_f\sin\phi}{2\cos A_f}
-\frac{\rho_fv_f\cos\phi}{\sin(2A_f)}-\frac{\rho_fq_f\cos\alpha_f\cos\phi}{2\cos A_f}\right)\bar{\partial}_{_\Gamma}r_{-}^{f}
\\~=~&(N-N_f)\left(\frac{\rho_fq_f\sin(\phi-\sigma_f)}{\sin(2A_f)}-\frac{\rho_fq_f\cos(\alpha_f-\phi)}{2\cos A_f}\right)\bar{\partial}_{_\Gamma}r_{-}^{f}.
\end{aligned}
\end{equation}

Combining (\ref{6609}), (\ref{6608}), and (\ref{6607}) we get
$$
\begin{aligned}
\frac{G_u}{\sqrt{G_u^2+G_v^2}}\bar{\partial}_{_\Gamma}u+\frac{G_v}{\sqrt{G_u^2+G_v^2}}\bar{\partial}_{_\Gamma}v
=\mathcal{L}_2\bar{\partial}_{_\Gamma}r_{-}^{f}.
\end{aligned}
$$
where
$$
\mathcal{L}_2=\frac{1}{\sqrt{G_u^2+G_v^2}}\left[\frac{q_f(\rho-\rho_f)L_f\sin(\phi-\alpha_f)}{2\cos A_f}+\rho_fq_f(N-N_f)\left(\frac{\sin(\phi-\sigma_f)}{\sin(2A_f)}-\frac{\cos(\alpha_f-\phi)}{2\cos A_f}\right)\right].
$$

Set $(\cos\chi, \sin\chi)=\Big(\frac{G_u}{\sqrt{G_u^2+G_v^2}}, \frac{G_v}{\sqrt{G_u^2+G_v^2}}\Big)$. Then, $\chi$ can be seen as a function of $(u, v, u_f, v_f)$.  As in (\ref{52501}) we have
\begin{equation}\label{52505}
\kappa t_{+}\sin(\beta-\chi)\bar{\partial}_{+}c+
\kappa t_{-}\sin(\chi-\alpha)\bar{\partial}_{-}c=\mathcal{L}_2\bar{\partial}_{_\Gamma}r_{-}^{f}.
\end{equation}

Since the tail shock is compressive and the flow downstream of the tail shock is subsonic relative to the shock front, we have $N<N_f$ and $N<c$.
Then we have
\begin{equation}\label{52506}
\begin{aligned}
&G_u\sin\phi-G_v\cos\phi\\=~&
\Big(\rho-\frac{\rho u^2}{c^2}\Big)(u-u_f)\sin\phi-v(v-v_f)\frac{\rho u}{c^2}\sin\phi
-\Big(\rho-\frac{\rho v^2}{c^2}\Big)(v-v_f)\cos\phi+u(u-u_f)\frac{\rho v}{c^2}\cos\phi\\=~&
\rho (N-N_f)-\frac{\rho u^2}{c^2}(u-u_f)\sin\phi-(v-v_f)\frac{\rho uv}{c^2}\sin\phi
+\frac{\rho v^2}{c^2}(v-v_f)\cos\phi+(u-u_f)\frac{\rho uv}{c^2}\cos\phi
\\=~&
\rho (N-N_f)-\frac{\rho u^2}{c^2}(N-N_f)\sin^2\phi+\frac{2\rho uv}{c^2}(N-N_f)\sin\phi\cos\phi
-\frac{\rho v^2}{c^2}(N-N_f)\cos^2\phi
\\=~&
\rho (N-N_f)-\frac{\rho}{c^2}(N-N_f)(u\sin\phi-v\cos\phi)^2
\\=~& \rho(N-N_f)\Big(1-\frac{N^2}{c^2}\Big)~<~0.
\end{aligned}
\end{equation}

Combining (\ref{52501}) and (\ref{52505}), we have
\begin{equation}\label{6612}
\mathcal{A}\bar{\partial}_{_\Gamma}\phi=\mathcal{B}(\bar{\partial}_{-}c-\bar{\partial}_{+}c)+\mathcal{C}\bar{\partial}_{_\Gamma}r_{-}^{f},
\end{equation}
where
\begin{equation}\label{902a}
\begin{aligned}
\mathcal{A}&=(N-N_f)[t_{+}\sin(\beta-\chi)+t_{-}\sin(\chi-\alpha)]\\&=\frac{N-N_f}{\sin 2A}
[\sin(\phi-\beta)\sin(\beta-\chi)+\sin(\alpha-\phi)\sin(\chi-\alpha)],
\end{aligned}
\end{equation}
\begin{equation}\label{902b}
\mathcal{B}=t_{-}t_{+}[\sin(\phi-\alpha)\sin(\beta-\chi)-\sin(\chi-\alpha)\sin(\beta-\phi)]=t_{+}t_{-}
\sin(\chi-\phi)\sin(2A),
\end{equation}
and
$$
\mathcal{C}=\mathcal{L}_1[t_{+}\sin(\beta-\chi)+t_{-}\sin(\chi-\alpha)]-\mathcal{L}_2[t_{+}\sin(\beta-\phi)+t_{-}\sin(\phi-\alpha)].
$$
Actually, $\mathcal{A}$, $\mathcal{B}$, $\mathcal{C}$ can be seen as functions of $(u, v, u_f, v_f, \phi)$, i.e.,
$$
\mathcal{A}=\mathcal{A}(u, v, u_f, v_f, \phi),\quad \mathcal{B}=\mathcal{B}(u, v, u_f, v_f, \phi), \quad \mathcal{C}=\mathcal{C}(u, v, u_f, v_f, \phi).
$$

By $N<c$ we have $\sigma<\phi<\alpha$ and accordingly
\begin{equation}\label{52602}
0<\sin(\alpha-\phi)<\sin(\phi-\beta).
\end{equation}
So, if $N<N_f$ and $N<c<q$ at some point on the tail shock then by (\ref{52506}) we have
\begin{equation}\label{6610}
\mathcal{B}>0
\end{equation}
at this point.

By the result on oblique shock-fan-shock composite wave curve in Section 2.2.2, we have
\begin{equation}\label{90901a}
G\big({u}_b^t(\tau_t^f), {v}_b^t(\tau_t^f), \hat{u}(\tau_t^f), \hat{v}(\tau_t^f)\big)=0\quad \mbox{for}\quad \tau_t^f\in (\tau_1^i, \tau_{po}(\tau_0)).
\end{equation}
As in (\ref{52506}), we can use (\ref{90901b}) to deduce
$$
G_{u_f}\big({u}_b^t(\tau_w), {v}_b^t(\tau_w), \hat{u}(\tau_w), \hat{v}(\tau_w)\big)\hat{u}'(\tau_w)+
G_{v_f}\big({u}_b^t(\tau_w), {v}_b^t(\tau_w), \hat{u}(\tau_w), \hat{v}(\tau_w)\big)\hat{v}'(\tau_w)=0.
$$

Combining this with (\ref{90901a}), we have
$$
G_{u}\big({u}_b^t(\tau_w), {v}_b^t(\tau_w), \hat{u}(\tau_w), \hat{v}(\tau_w)\big)(u_b^t)'(\tau_w)+
G_{v}\big({u}_b^t(\tau_w), {v}_b^t(\tau_w), \hat{u}(\tau_w), \hat{v}(\tau_w)\big)(v_b^t)'(\tau_w)=0.
$$
This also indicates that  shock polar and the oblique fan-shock wave curve of the state  $(\hat{u}(\tau_w), \hat{v}(\tau_w))$ is tangent at the point $({u}_b^t(\tau_w), {v}_b^t(\tau_w))$.

Let $\nu$ be
the angle between the tangent vector $-(({u}_b^t)'(\tau_w), ({v}_b^t)'(\tau_w))$ of the fan-shock-fan wave curve $\wideparen{\mathrm{IJ}}$ at the point $({u}_b^t(\tau_w), {v}_b^t(\tau_w))$ and the vector $({u}_b^t(\tau_w), {v}_b^t(\tau_w))$; see Figure \ref{Superbzt}.
Then we have $\chi=\sigma-\nu+\frac{\pi}{2}$.
 By assumption {\bf (H2)} we have $0<\nu<\frac{\pi}{2}$. Thus
\begin{equation}\label{52603}
 |\sin(\chi-\alpha)|=|\cos (A+\nu)|<\cos(\nu-A)=-\sin(\beta-\chi)\quad \mbox{at}\quad \mathrm{O};
\end{equation}
see Figure \ref{Superbzt}.
Combining (\ref{52602}) and (\ref{52603}) and recalling $N<N_f$, we have
\begin{equation}\label{6611}
\mathcal{A}>0\quad \mbox{at}\quad \mathrm{O}.
\end{equation}



\subsubsection{\bf Formulation and resolution of the IBVP for $w''(0)>0$}
In order to solve the IBVP (\ref{42501}), (\ref{RBDD}) for $w''(0)>0$, we consider (\ref{42501}) on an angular domain $R(\delta)=\{(x, y)\mid  0\leq x\leq \delta, w(x)\leq y\leq \psi(x)\}$,
where $y=\psi(x)$ ($\psi(0)=0$) is an unknown curve representing the tail shock. We prescribe the following  boundary conditions:

\noindent
on $y=\psi(x)$,
\begin{equation}\label{6621}
(\rho  u -\rho_fu_f)(u -u_f)+  (\rho  v -\rho_fv_f)(v -v_f)=0
\end{equation}
and
\begin{equation}\label{6622}
\frac{{\rm d}\psi(x)}{{\rm d}x}=-\frac{u-u_f}{v-v_f}, \quad \psi(0)=0,
\end{equation}
where $(u_f, v_f)(x, y)=(\tilde{u}, \tilde{v})(\arctan(y/x))$;

\noindent
on $y=w(x)$,
\begin{equation}\label{6623}
v(x,y)=u(x,y)w'(x).
\end{equation}

We are going to
 give an a priori estimate that the tail shock will go into the fan wave region $\Xi_1$ from the origin $\mathrm{O}$.
In order to prove this, we shall use (\ref{6612}) to prove $\bar{\partial}_{_{\Gamma}}\phi>0$ at $\mathrm{O}$.

Since
$r_{-}^{f}$ has a multi-valued singularity at $\mathrm{O}$, we take the coordinate transformation
$$
x=x, \quad y=x\tan\theta.
$$
Then the fan wave region $\Xi_1$ is mapped into the region $\Upsilon_1=\{(x, \theta)\mid x>0,~\phi_{pr}\leq\theta\leq\phi_{po}\}$ in the $(x, \theta)$-plane.

In terms of the $x$ and $\theta$ coordinates,  we have
$$
\bar{\partial}_{_\Gamma}=\cos\phi\partial_x+\frac{\sin(\phi-\theta)}{x\sec\theta}\partial_{\theta}.
$$
This implies that the tail shock in the $(x, \theta)$-plane is the integral curve of $$\frac{{\rm d}\theta}{{\rm d}x}=\frac{\sin(\phi-\theta)}{x\sec\theta\cos\phi}, \quad \theta\mid_{x=0}~=~\phi_{pr}.$$

By (\ref{6612}) we have that along the tail shock,
\begin{equation}\label{73001a}
\mathcal{A}\bar{\partial}_{_\Gamma}(\phi-\theta)=\mathcal{B}(\bar{\partial}_{-}c-\bar{\partial}_{+}c)+
\frac{\mathcal{C}\sin(\phi-\theta)}{x\sec\theta}\tilde{r}_{-}'(\theta)-\mathcal{A}\frac{\sin(\phi-\theta)}{x\sec\theta},
\end{equation}
where $\tilde{r}_{-}(\theta)=\arctan\big(\frac{\tilde{v}(\theta)}{\tilde{u}(\theta)}\big)-\int_{q_{po}}^{\tilde{q}(\theta)}\frac{\sqrt{q^{2}-c^{2}}}{qc}{\rm d}q$ and $\tilde{q}(\theta)=\sqrt{\tilde{u}^2(\theta)+\tilde{v}^2(\theta)}$.

In view of (\ref{7a})--(\ref{8a}), and (\ref{32403}), we have
$$
\bar{\partial}_{0}\sigma=\Big(\frac{\bar{\partial}_{+}+\bar{\partial}_{-}}{2\cos A}\Big)\Big(\frac{\alpha+\beta}{2}\Big)=\frac{\bar{\partial}_{-}c-\bar{\partial}_{+}c}{2\kappa q}.
$$
Thus, by the slip boundary condition (\ref{6623}) we have that on the ramp,
\begin{equation}\label{52502}
\bar{\partial}_{-}c-\bar{\partial}_{+}c=
2\kappa q\bar{\partial}_{0}\sigma=\frac{2\kappa qw''(x)}{(1+[w'(x)]^2)^{\frac{3}{2}}}.
\end{equation}

At the point $(x, \theta)=(0, \phi_{pr})$, we have $\tau_f=\tau_w$, $(u_f, v_f)=(\hat{u}(\tau_w), \hat{v}(\tau_w))$, and $\phi=\alpha_f=\phi_{pr}$. So, we have $\mathcal{L}_1(0, \phi_{pr})=\mathcal{L}_2(0, \phi_{pr})=0$.
Therefore, by (\ref{6610}), (\ref{6611}), (\ref{73001a}), (\ref{52502}), and $w''(0)>0$ we can see that if $\bar{\partial}_{-}c-\bar{\partial}_{+}c$ is continuous along the tail shock, then
$$
\bar{\partial}_{_\Gamma}(\phi-\theta)>0\quad \mbox{for}\quad (x, \theta)=(0, \phi_{pr}).
$$
This implies an a priori estimate that if the problem (\ref{42501}), (\ref{6621})--(\ref{6623}) admits a local classical solution then when $\delta$ is sufficiently small $\psi(x)>x\tan\phi_{pr}$ for $0<x<\delta$.

\begin{lem}
Assume $w''(0)>0$. Then,
under assumptions {\bf (H1)} and {\bf (H2)},
there is a small $\delta>0$ such that the free boundary problem  (\ref{42501}), (\ref{6621})--(\ref{6623}) admits a classical solution on $R(\delta)$. Moreover, the shock front satisfies $\psi''>0$.
\end{lem}

\begin{proof}
Problem (\ref{42501}), (\ref{6621})--(\ref{6623}) is a typical free boundary problem proposed in Li and Yu \cite{Li-Yu} (Chapter 3.2). Assumption {\bf (H2)} can deduce the solvability condition (see (2.28) in Chapter 3 in \cite{Li-Yu}) of this typical free boundary problem.
\end{proof}

For convenience we denote the solution of the typical free boundary problem  (\ref{42501}), (\ref{6621})--(\ref{6623}) by $(u, v)=(u_d, v_d)(x,y)$.
Let
\begin{equation}\label{72408}
(u, v)=\left\{
         \begin{array}{ll}
           (u_0, 0), & \hbox{$x\tan\phi_{po}<y<\delta\tan\phi_{po}$, $0<x<\delta$;} \\[4pt]
           (\tilde{u}, \tilde{v})(\arctan(y/x)), & \hbox{$\psi(x)<y<x\tan\phi_{po}$, $0<x<\delta$;} \\[4pt]
           (u_d, v_d)(x, y), & \hbox{$w(x)<y<\psi(x)$, $0<x<\delta$.}
         \end{array}
       \right.
\end{equation}
Then, when $w''(0)>0$ the function $(u, v)(x,y)$ defined in (\ref{72408}) is a piecewise smooth local solution to the IBVP (\ref{42501}), (\ref{RBDD}).

\subsubsection{\bf Formulation and resolution of the IBVP for $w''(0)<0$}
From (\ref{6612}) and (\ref{52502}) we can see that
when $w''(0)<0$, the tail shock does not enter into the fan wave region $\Xi_1$. The tail shock must be pre-sonic, since  there are only one type of characteristics of the upstream flow can reach the tail shock. In this case the front side states of the tail shock is also a priori unknown.  However, by (\ref{9010a}), (\ref{6601}), and Remark \ref{9010b} we know that the back side states of the tail shock must satisfy $H(u, v)=0$.


For $\varrho>0$, we let $B_{\varrho}:=\{(u, v)\mid (u-{u}_b^t(\tau_w))^2+ (v-{v}_b^t(\tau_w))^2<\varrho^2\}$.
By the convexity of $\wideparen{\mathrm{PI}}$ we know that when $\varrho$ is sufficiently small,
for any $(u, v)\in \wideparen{\mathrm{IJ}}\cap B_{\varrho}$ there corresponds a unique $(\hat{u}_f, \hat{v}_f)\in\wideparen{\mathrm{PI}}$, such that
$$
u={u}_b^{t}(\ddot{\tau}(\hat{q}_{f}))\quad \mbox{and} \quad v={v}_b^t(\ddot{\tau}(\hat{q}_{f})),
$$
where $\hat{q}_{f}=\sqrt{\hat{u}_{f}^2+\hat{v}_{f}^2}$.
Moreover,  $\hat{u}_{f}({u}_b^t(\tau_w), {v}_b^t(\tau_w))=\hat{u}(\tau_w)$
and  $\hat{v}_{f}({u}_b^t(\tau_w), {v}_b^t(\tau_w))=\hat{v}(\tau_w)$.

In order to find a pre-sonic shock wave solution, we consider (\ref{42501}) on an angular domain $R(\delta)=\{(x, y)\mid  0\leq x\leq \delta, w(x)\leq y\leq \psi(x)\}$,
where $y=\psi(x)$ ($\psi(0)=0$) is an unknown curve representing the pre-sonic shock. We prescribe the following  boundary conditions:

\noindent
on $y=\psi(x)$,
\begin{equation}\label{6641}
H(u, v)=0
\end{equation}
and
\begin{equation}\label{6642}
\frac{{\rm d}\psi(x)}{{\rm d}x}=-\frac{u-\hat{u}_f(u, v)}{v-\hat{v}_f(u, v)}, \quad \psi(0)=0;
\end{equation}
\noindent
on $y=w(x)$,
\begin{equation}\label{6643}
v(x,y)=u(x,y)w'(x).
\end{equation}

\begin{lem}
Assume $w''(0)<0$. Then,
under assumptions {\bf (H1)} and {\bf (H2)}, there is a small $\delta>0$ such that the free boundary problem  (\ref{42501}), (\ref{6641})--(\ref{6643}) admits a classical solution on $R(\delta)$. Moreover, the shock front satisfies $\psi''<0$.
\end{lem}

\begin{proof}
Problem (\ref{42501}), (\ref{6641})--(\ref{6643}) is also a typical free boundary value problem. Assumption {\bf (H2)} can deduce a solvability condition (see (2.28) in Chapter 3 in \cite{Li-Yu}) of this typical free boundary problem. From (\ref{6612}) 
we can get $\psi''<0$.
\end{proof}

For convenience we denote the solution of the typical free boundary problem  (\ref{42501}), (\ref{6641})--(\ref{6643}) by $(u, v)=(u_d, v_d)(x,y)$.
Since the tail shock is pre-sonic, the front side state of the shock is
$\big(\hat{u}_f(u_d(x,y), v_d(x,y)), \hat{v}_f(u_d(x,y), v_d(x,y))\big)$. Since $\psi''<0$, for any $(x,y)\in \Xi_2:=\{(x,y)\mid \psi(x)<y<x\tan\phi_{pr}, 0<x<\delta\}$ there exists a unique $\bar{x}\in (0, x)$ such that
$
\frac{y-\psi(\bar{x})}{x-\bar{x}}=\psi'(\bar{x})
$.
We define
$$
(u_1, v_1)(x,y)=\big(\hat{u}_f(u_d(\bar{x},\psi(\bar{x})), v_d(\bar{x},\psi(\bar{x}))), \hat{v}_f(u_d(\bar{x},\psi(\bar{x})), v_d(\bar{x},\psi(\bar{x})))\big), \quad (x,y)\in \Xi_2.
$$
Then $(u, v)=(u_1, v_1)(x,y)$ is actually a simple wave solution of the system (\ref{42501}) on $\Xi_2$.

We define
\begin{equation}\label{72409}
(u, v)=\left\{
         \begin{array}{ll}
           (u_0, 0), & \hbox{$x\tan\phi_{po}<y<\delta\tan\phi_{po}$, $0<x<\delta$;} \\[4pt]
             (\tilde{u}, \tilde{v})(\arctan(y/x)), & \hbox{$x\tan\phi_{pr}<y<x\tan\phi_{po}$, $0<x<\delta$;} \\[4pt]
(u_1, v_1)(x, y), & \hbox{$\psi(x)<y<x\tan\phi_{pr}$, $0<x<\delta$;} \\[4pt]
           (u_d, v_d)(x, y), & \hbox{$w(x)<y<\psi(x)$, $0<x<\delta$.}
         \end{array}
       \right.
\end{equation}
Then, when $w''(0)<0$ the function $(u, v)(x,y)$ defined in (\ref{72409}) is a piecewise smooth local solution to the IBVP (\ref{42501}), (\ref{RBDD}).

This completes the proof of Theorem \ref{Thm39}.
\vskip 4pt

\subsection{Stability of the shock-fan-shock composite wave: case II}
\subsubsection{\bf Problem and main results}
We consider (\ref{42501}) with the following initial-boundary conditions:
\begin{equation}\label{RBDD1}
\left\{
  \begin{array}{ll}
    (u, v)(0, y)=(\check{u}, \check{v})(y), & \hbox{$y>0$;} \\[6pt]
 \displaystyle \left(\frac{v}{u}\right)(x,y)=w'(x), & \hbox{$y=w(x)$, $x>0$,}
  \end{array}
\right.
\end{equation}
where $(\check{u}, \check{v})(y)\in C^1[0, +\infty)$, $w(x)\in C^2[0,+\infty)$, $\check{v}(0)=0$,  and $w(0)=0$. Set $u_0=\check{u}(0)$ and $\theta_w=\arctan(w'(0))$.
We assume $u_0>c_0$, $\tau_2^i<\tau_0\leq \tau_c$, and ${\sigma}_m<\theta_w<{\sigma}_M$, where $c_0=\ddot{c}(u_0)$ and $\tau_0=\ddot{\tau}(u_0)$.

We purpose to construct a piecewise smooth solution with a curved post-sonic head shock
to the IBVP (\ref{42501}), (\ref{RBDD1}).
To this aim, we also make the following assumption:
\begin{description}
  \item[(H3)] $\frac{{\rm d}}{{\rm d}y}r_{-}(\check{u}(y), \check{v}(y))=0$ and $\frac{{\rm d}}{{\rm d}y}r_{+}(\check{u}(y), \check{v}(y))<0$ for $y\geq 0$.
\end{description}

\begin{rem}
Assumption {\bf (H3)} is a sufficient (not necessary) condition for the appearance of a post-sonic head shock.
\end{rem}

The main results are stated as follows.

\begin{thm}\label{thm23}
Assume $u_0>c_0$, $\tau_2^i<\tau_0\leq \tau_c$, and ${\sigma}_m<\theta_w<{\sigma}_M$.
Assume furthermore that assumptions {\bf (H1)--(H3)} hold. Then there exists a constant $\zeta_*$ depending only on $\check{u}(0)$,  $\check{v}'(0)$, and  $\check{u}'(0)$ such that the following conclusion hold.
\begin{itemize}
  \item If  $w''(0)>\zeta_*$ then the IBVP (\ref{42501}), (\ref{RBDD1}) admits a local piecewise smooth solution with a curved pre-sonic head shock, a centered wave, and a curved transonic tail shock; see Figure \ref{Figure13} (left).
  \item If  $w''(0)<\zeta_*$ then the IBVP (\ref{42501}), (\ref{RBDD1}) admits a local piecewise smooth solution with a curved pre-sonic head shock, a centered wave, and a curved post-sonic tail shock; see Figure \ref{Figure13} (right).
\end{itemize}
\end{thm}


\begin{figure}[htbp]
\begin{center}
 \quad \includegraphics[scale=0.69]{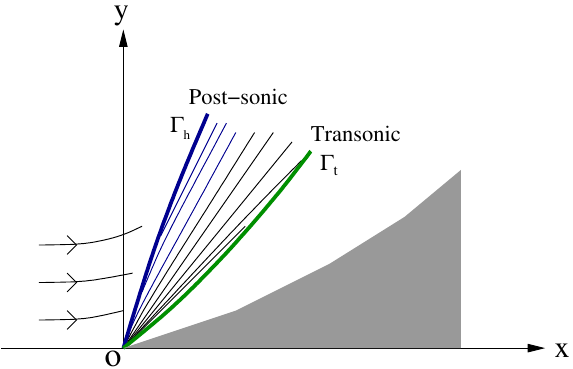}  \quad\quad\qquad\includegraphics[scale=0.59]{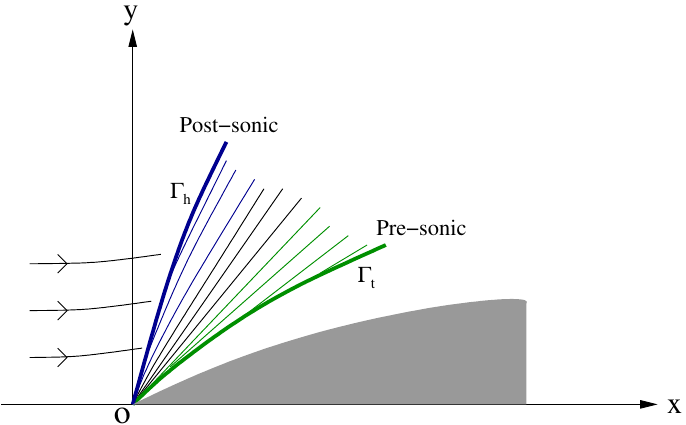}
\caption{\footnotesize Wave structure of the solution to the IBVP (\ref{42501}), (\ref{RBDD1}).}
\label{Figure13}
\end{center}
\end{figure}

\subsubsection{\bf Existence of a post-sonic head shock}
We first consider (\ref{42501}) with the data
\begin{equation}\label{72101}
 (u, v)(0, y)=(\check{u}, \check{v})(y), \quad y\geq 0.
\end{equation}
\begin{lem}
There exists a small $\delta_0>0$ such that the Cauchy problem (\ref{42501}), (\ref{72101}) admits a classical solution on the domain  $\Sigma_1(\delta_0)=\{(x, y)\mid x\tan\alpha_0\leq y\leq \delta_0\tan\alpha_0, 0\leq x\leq\delta_0\}$, where $y=x\tan\alpha_0$ is a straight $C_{+}$ characteristic line issued from the origin and $\alpha_0=\arcsin\frac{c_0}{u_0}$.
Moreover, under assumption {\bf (H3)} the solution satisfies
\begin{equation}\label{72104}
\tau_2^i<\tau<\tau_c, \quad\bar{\partial}_{+}r_{+}<0, \quad \mbox{and}\quad \bar{\partial}_{-}r_{-}=0.
\end{equation}
\end{lem}
\begin{proof}
The local existence follows from Li and Yu \cite{Li-Yu}.
When $\delta_0$ is sufficiently small, we have $\tau_2^i<\tau<\tau_c$ in $\Sigma_1(\delta_0)$, and
$$
(\bar{\partial}_{-}r_{-})(0, y)=0\quad\mbox{and}\quad
(\bar{\partial}_{+}r_{+})(0, y)=2\sin\alpha_0\cdot \frac{{\rm d}}{{\rm d}y}r_{+}(\check{u}(y), \check{v}(y))<0
$$
for $0\leq y\leq \delta_0\tan\alpha_0$.
Then by (\ref{cdr}) one has (\ref{72104}).
This completes the proof.
\end{proof}

For convenience, we use $(u, v)=(u_1, v_1)(x,y)$ to denote the local solution of the Cauchy problem (\ref{42501}), (\ref{72101}).

We look for a post-sonic head shock $\Gamma_h$: $y=\psi_h(x)$, where $\psi_h(0)=0$.
Let $\phi$ be the inclination angle of the head shock. Let $\bar{\partial}_{_\Gamma}$ be the derivative along $\Gamma_h$. 
We still use the subscripts `$f$' and `$b$' to denote the states on the front and back sides of the shock, respectively.
Then, along $\Gamma_h$ there hold the following relations:
\begin{equation}\label{72503}
\begin{aligned}
&(u_f, v_f)=(u_1, v_1),\quad q_f=\sqrt{u_f^2+v_f^2}, \quad \tau_f=\ddot{\tau}(q_f), \quad \tau_b=\tau_{po}(\tau_f), \quad N_f=\tau_f\sqrt{-\frac{h(\tau_b)-h(\tau_f)}{\tau_b^2-\tau_f^2}},\\& \phi=\arctan\Big(\frac{v_f}{u_f}\Big)+\arcsin\Big(\frac{N_f}{q_f}\Big)=\alpha_b,\quad L_f=L_b=\sqrt{u_f^2+v_f^2-N_f^2}, \quad N_b=\frac{\tau_bN_f}{\tau_f},\\& \qquad u_b=N_b\sin\phi+L_b\cos\phi, \quad v_b=L_b\sin\phi-N_b\cos\phi, \quad \frac{{\rm d}\psi_h}{{\rm d}x}=\tan\phi.
\end{aligned}
\end{equation}
Actually, by (\ref{72503}) we can obtain the post-sonic shock $\Gamma_h$ and the states on both sides of it.

By (\ref{72103}) and (\ref{72104}) we have
\begin{equation}\label{72105}
\bar{\partial}_{_\Gamma}r_{-}^{f}=0\quad \mbox{and}\quad \bar{\partial}_{_\Gamma}r_{+}^{f}<0\quad \mbox{along}\quad \Gamma_{h},
\end{equation}
where $r_{\pm}^{f}=r_{\pm}(u_f, v_f)$.

\begin{lem}
(Structural conditions)
Under assumption {\bf (H3)}, there hold
\begin{equation}\label{72204}
\bar{\partial}_{_\Gamma}r_{+}^{b}<0\quad \mbox{and}\quad \bar{\partial}_{_\Gamma}r_{-}^{b}<0\quad \mbox{along}\quad \Gamma_h,
\end{equation}
where $r_{\pm}^{b}=r_{\pm}(u_b, v_b)$.
\end{lem}
\begin{proof}
From (\ref{U}) and (\ref{62604}) we have
$$
m=\rho  N =\rho  q  (\cos\sigma \sin\phi-\sin\sigma \cos\phi).
$$
Thus, we have
\begin{equation}\label{4250103a}
\bar{\partial}_{_\Gamma} m=\frac{N }{q }\bar{\partial}_{_{\Gamma}}(\rho q )-\rho L \bar{\partial}_{_\Gamma}\sigma +\rho L \bar{\partial}_{_\Gamma} \phi\quad \mbox{along}\quad {\Gamma_h}.
\end{equation}

From the Bernoulli's law (\ref{5801}) we have
\begin{equation}\label{62010a}
q \bar{\partial}_{_\Gamma} q +\tau p'(\tau)\bar{\partial}_{_\Gamma} \tau =0\quad \mbox{along}\quad {\Gamma_h}.
\end{equation}
This yields
$$
\bar{\partial}_{_\Gamma} (\rho q )=\frac{1}{q }(q ^2+\tau ^2p'(\tau))\bar{\partial}_{_\Gamma}\rho \quad \mbox{along}\quad {\Gamma_h}.
$$
Inserting this into (\ref{4250103a}), we known that along ${\Gamma_h}$,
\begin{equation}\label{250102a}
\bar{\partial}_{_\Gamma} \phi=\frac{\bar{\partial}_{_\Gamma} m}{\rho L }-\frac{N }{\rho L q ^2}(q ^2+\tau ^2p'(\tau))\bar{\partial}_{_\Gamma}\rho +\bar{\partial}_{_\Gamma}\sigma
\end{equation}
and
\begin{equation}\label{42804b}
\bar{\partial}_{_\Gamma} \sigma =\bar{\partial}_{_\Gamma} \phi-\frac{\bar{\partial}_{_\Gamma} m}{\rho L }+\frac{N }{\rho L q ^2}(q ^2+\tau ^2p'(\tau))\bar{\partial}_{_\Gamma}\rho.
\end{equation}
We also have
\begin{equation}\label{50701b}
\bar{\partial}_{_\Gamma}u=\bar{\partial}_{_\Gamma}(q\cos\sigma)=\cos\sigma\bar{\partial}_{_\Gamma}q
-q\sin\sigma\bar{\partial}_{_\Gamma}\sigma,
\end{equation}
\begin{equation}\label{50702b}
\bar{\partial}_{_\Gamma}v=\bar{\partial}_{_\Gamma}(q\sin\sigma)=\sin\sigma\bar{\partial}_{_\Gamma}q
+q\cos\sigma\bar{\partial}_{_\Gamma}\sigma.
\end{equation}

By (\ref{50701b})--(\ref{50702b}) we have
\begin{equation}\label{42801b}
\begin{aligned}
&\cos\alpha_b \bar{\partial}_{_\Gamma} u_b+\sin\alpha_b \bar{\partial}_{_\Gamma} v_b~=~
\cos A_b \bar{\partial}_{_\Gamma} q_b+q_b\sin A_b \bar{\partial}_{_\Gamma} \sigma_b\quad \mbox{along}\quad {\Gamma_h}.
\end{aligned}
\end{equation}
From (\ref{62010a}), we have
\begin{equation}\label{42803b}
\bar{\partial}_{_\Gamma} q_b=-\tau_b q_b \sin^2 A_b \bar{\partial}_{_\Gamma} \rho_b \quad \mbox{along}\quad {\Gamma_h}.
\end{equation}
From (\ref{42804b}) one has
\begin{equation}\label{42804ba}
\bar{\partial}_{_\Gamma} \sigma_b=\bar{\partial}_{_\Gamma} \phi-\frac{\bar{\partial}_{_\Gamma} m}{\rho_bL_b}+\frac{N_b}{\rho_bL_bq_b^2}\big(q_b^2+\tau_b^2p'(\tau_b)\big)\bar{\partial}_{_\Gamma}\rho_b\quad \mbox{along}\quad {\Gamma_h}.
\end{equation}

Inserting (\ref{42803b}) and  (\ref{42804ba}) into  (\ref{42801b}) and recalling that $N_b=c_b=q_b\sin A_b$ for the post-sonic shock, we get
$$
\cos\alpha_b \bar{\partial}_{_\Gamma} u_b+\sin\alpha_b \bar{\partial}_{_\Gamma} v_b=
q_b\sin A_b \left(\bar{\partial}_{_\Gamma} \phi-\frac{\bar{\partial}_{_\Gamma} m }{\rho_fL_f}\right)+
q_b\sin A_b\Big(\frac{1}{\rho_fL_f}-\frac{1}{\rho_bL_b}\Big)\bar{\partial}_{_\Gamma} m\quad \mbox{along}\quad {\Gamma_h}.
$$

We compute
\begin{equation}\label{72203}
\begin{aligned}
\bar{\partial}_{_\Gamma} \phi-\frac{\bar{\partial}_{_\Gamma} m }{\rho_fL_f}&=
-\frac{N_f }{\rho_f L_f q_f^2}(q_f^2+\tau_f^2 p'(\tau_f))\bar{\partial}_{_\Gamma}\rho_f +\bar{\partial}_{_\Gamma}\sigma_f\\&=\frac{N_f(q_f^2-c_f^2 )}{ L_f q_f^2\sin (2A_f)}\bar{\partial}_{_\Gamma}r_{+}^f +\frac{\bar{\partial}_{_\Gamma}r_{+}^f}{2}
\end{aligned}
\end{equation}
and
$$
\begin{aligned}
\Big(\frac{1}{\rho_fL_f}-\frac{1}{\rho_bL_b}\Big) \bar{\partial}_{_\Gamma}m ~=~&\frac{\tau_f-\tau_b}{2\rho_fN_fL_f}\Bigg[\frac{2\tau_f}{\tau_f^2-\tau_b^2}\left(-p'(\tau_f)+\frac{2h(\tau_f)-2h(\tau_b)}{\tau_f^2-\tau_b^2}\right)
\partial_{_\Gamma}\tau_f\\&
\qquad\qquad\quad-\frac{2\tau_b}{\tau_f^2-\tau_b^2}\left(-p'(\tau_b)+\frac{2h(\tau_f)-2h(\tau_b)}{\tau_f^2-\tau_b^2}\right)
\partial_{_\Gamma}\tau_b\Bigg]\\~=~&\frac{\tau_f^2}{(\tau_f+\tau_b)N_fL_f}\left(-p'(\tau_f)+\frac{2h(\tau_f)-2h(\tau_b)}{\tau_f^2-\tau_b^2}\right)
\partial_{_\Gamma}\tau_f
\\~=~&\frac{c_f^2-N_f^2}{(\tau_f+\tau_b)\rho_fN_fL_f\sin(2A_f)}
\bar{\partial}_{_\Gamma}r_{+}^f,
\end{aligned}
$$
where we use (\ref{m}).
Thus, by (\ref{62401}) and (\ref{72105}) and $N_f>c_f$ we have
$$
\begin{aligned}
\bar{\partial}_{_\Gamma}r_{+}^{b}&=\frac{\cos\alpha_b \bar{\partial}_{_\Gamma} u_b+\sin\alpha_b \bar{\partial}_{_\Gamma} v_b}{c_b}\\[2pt]&<\frac{(c_f^2-N_f^2)\bar{\partial}_{_\Gamma}r_{+}^f}{N_fL_f\sin(2A_f)}
+\frac{N_f(q_f^2-c_f^2 )\bar{\partial}_{_\Gamma}r_{+}^f }{ L_f q_f^2\sin (2A_f)} +\frac{\bar{\partial}_{_\Gamma}r_{+}^f}{2}<\frac{\bar{\partial}_{_\Gamma}r_{+}^f}{2}<0\quad \mbox{along}\quad \Gamma_h.
\end{aligned}
$$

Using (\ref{6606}), (\ref{72501}), and (\ref{72105}) we have
$$
\bar{\partial}_{_\Gamma} \rho_b=-\frac{\bar{\partial}_{_\Gamma} \tau_b}{\tau_b^2}=
-\frac{\tau_{po}'(\tau_f)\bar{\partial}_{_\Gamma} \tau_f}{\tau_b^2}=
-\frac{\tau_{po}'(\tau_f)\bar{\partial}_{_\Gamma}r_{+}^f}{\tau_b^2\rho_f\sin (2A_f)}<0\quad \mbox{along}\quad \Gamma_h.
$$
Combining this with (\ref{72203}) we get
$$
\begin{aligned}
\bar{\partial}_{_\Gamma}r_{-}^{b}&=\frac{-(\cos\beta_b \bar{\partial}_{_\Gamma} u_b+\sin\beta_b \bar{\partial}_{_\Gamma} v_b)}{c_b}\\&= \bar{\partial}_{_\Gamma} \phi-\frac{\bar{\partial}_{_\Gamma} m }{\rho_bL_b}+
\tau_b\sin(2 A_b)\bar{\partial}_{_\Gamma} \rho_b
\\&= \bar{\partial}_{_\Gamma} r_{+}^{b}+
\tau_b\sin(2 A_b)\bar{\partial}_{_\Gamma} \rho_b<0\quad \mbox{along}\quad \Gamma_h.
\end{aligned}
$$
This completes the proof.
\end{proof}


\subsubsection{\bf The flow downstream of the post-sonic shock}
Let
$$
\wideparen{\mathrm{O_1B_1}}=\{(r_{+}, r_{-})\mid (r_{+}, r_{-})=({r}_{+}^{b}(x, y), {r}_{-}^{b}(x,y)), ~ (x,y)\in \Gamma_h\}
$$
be a curve in the $(r_{+}, r_{-})$-plane; see Figure \ref{Figure14} (right).
Then by (\ref{72204}), the mapping
$$
(r_{+}, r_{-})=({r}_{+}^{b}(x, y), {r}_{-}^{b}(x,y)), \quad (x,y)\in \Gamma_h
$$
has an inverse noted by
\begin{equation}\label{61706}
(x, y)=(x_b(r_{+}, r_{-}), y_b(r_{+}, r_{-})), \quad (r_{+}, r_{-})\in \wideparen{\mathrm{O_1B_1}}.
\end{equation}

\begin{figure}[htbp]
\begin{center}
\includegraphics[scale=0.46]{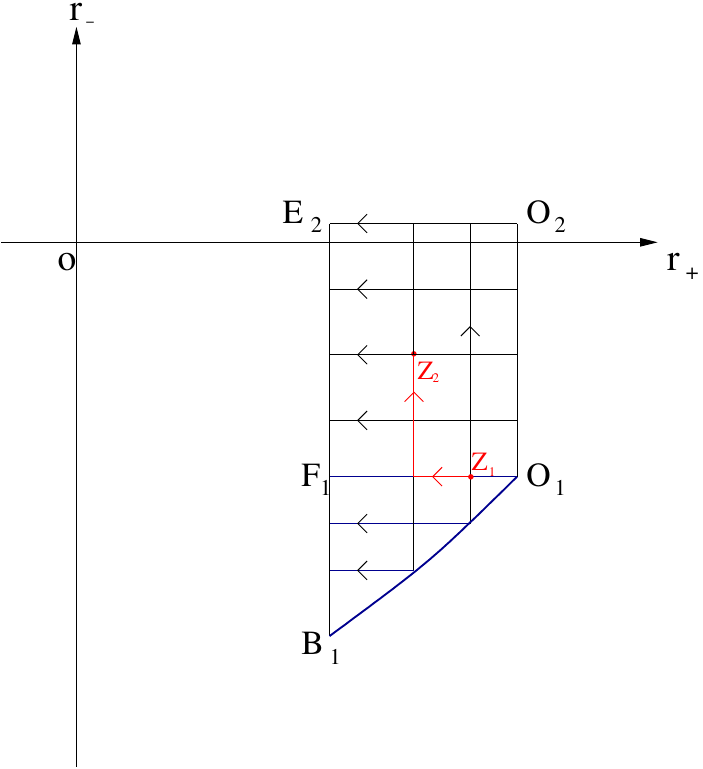}\qquad\qquad\qquad \includegraphics[scale=0.46]{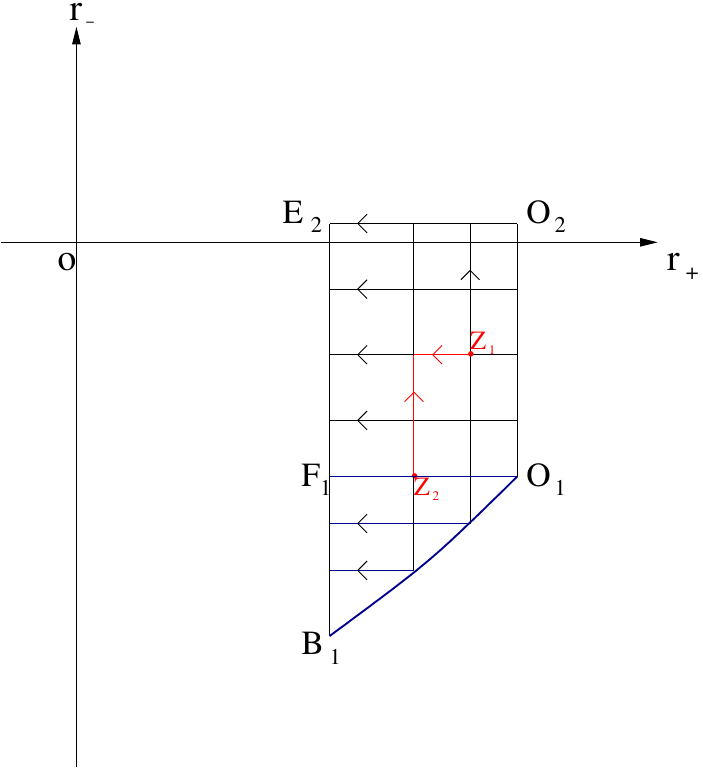}
\caption{\footnotesize Flow states between the post-sonic head shock and the transonic tail shock.}
\label{Figure14}
\end{center}
\end{figure}

In view of (\ref{72204}), the curve $\wideparen{\mathrm{O_1B_1}}$ can be represented by a function
$$
r_{-}=f_h(r_{+}), \quad r_{+}^{po}-\eta_0\leq r_{+}\leq r_{+}^{po},
$$
where $f_h(r_{+}^{po})= r_{-}^{po}$, $(r_{+}^{po},  r_{-}^{po})=(r_{+}(u_{po}, v_{po}), r_{-}(u_{po}, v_{po}))$, the state $(u_{po}, v_{po})$ is defined in (\ref{6602}), and $\eta_0$ is small positive constant.

Let
$$
\overline{\mathrm{O_1O_2}}=\big\{(r_{+}, r_{-})\mid r_{+}=r_{+}^{po},~~ r_{-}^{po}\leq r_{-}\leq r_{-}^{pr}\big\},
$$
where $r_{-}^{pr}=r_{-}(\hat{u}(\tau_w), \hat{v}(\tau_w))$.
We set
\begin{equation}\label{62501}
(x, y)=(0,0)\quad \mbox{on}\quad \overline{\mathrm{O_1O_2}}.
\end{equation}

We  consider
 (\ref{HT}) with the boundary conditions (\ref{61706}) and (\ref{62501}). The linear  problem (\ref{HT}),  (\ref{61706}) and (\ref{62501}) admits a classical solution denoted by $(x, y)=(x_{2}, y_{2})(r_{+}, r_{-})$ on the domain $\Delta_h\cup \Delta_2$, where
 $$\Delta_{h}=\big\{(r_{+}, r_{-})\mid f_h(r_{+})\leq r_{-}\leq r_{-}^{po},~ r_{+}^{po}-\eta_0\leq r_{+}\leq r_{+}^{po}\big\}$$
and
$$
\Delta_2=\big\{(r_{+}, r_{-})\mid r_{-}^{po}< r_{-}\leq r_{-}^{pr},~ r_{+}^{po}-\eta_0\leq r_{+}\leq r_{+}^{po}\big\}.
$$

\begin{lem}
Assume $\eta_0$ to be sufficiently small. Then the mapping $(x, y)=(x_{2}, y_{2})(r_{+}, r_{-})$, $(r_{+}, r_{-})\in (\Delta_h\cup \Delta_2)\setminus \overline{\mathrm{O_1O_2}}$ is invertible.
\end{lem}
\begin{proof}
 Firstly, when $\eta_0$ is sufficiently small,
 $$
\tau_1^i< \tau<\tau_2^i\quad \mbox{for}\quad (r_{+}, r_{-})\in\Delta_h\cup \Delta_2.
 $$
 Then by (\ref{alphar}) and (\ref{alphar1}) we have
 \begin{equation}\label{72205}
 \hat{\partial}_{\pm}\lambda_{-}>0\quad \mbox{and}\quad
  \hat{\partial}_{\pm}\lambda_{+}<0\quad \mbox{for}\quad (r_{+}, r_{-})\in\Delta_h\cup \Delta_2.
 \end{equation}

 In view of (\ref{9501a}),
when $\eta_0$ is sufficiently small we have 
 \begin{equation}\label{72506}
 \lambda_{-}<\lambda_{+}\quad \mbox{for}\quad (r_{+}, r_{-})\in\Delta_h\cup \Delta_2.
\end{equation}


 Since the post-sonic shock $\Gamma_h$ is an integral curve of $\frac{{\rm d}y}{{\rm d}x}=\lambda_{+}(u_b(x,y), v_b(x,y))$, we have
$$
 \hat{\partial}_{+}y-f_h'(r_{+}) \hat{\partial}_{-}y=\lambda_{+}\big(\hat{\partial}_{+}x-f_h'(r_{+}) \hat{\partial}_{-}x\big)\quad\mbox{on}\quad \wideparen{\mathrm{O_1B_1}}.
$$
 Combining this with (\ref{HT}), we get
 \begin{equation}\label{61711}
  \hat{\partial}_{-}x=0 \quad\mbox{on}\quad \wideparen{\mathrm{O_1B_1}}.
 \end{equation}
 Furthermore, by (\ref{72204}) and $\bar{\partial}_{_\Gamma}x=\cos\phi=\cos\alpha_b$ we have
 \begin{equation}\label{61712}
  \hat{\partial}_{+}x>0 \quad\mbox{on}\quad \wideparen{\mathrm{O_1B_1}}.
 \end{equation}

By (\ref{62501}) we have
\begin{equation}\label{72502}
  \hat{\partial}_{-}x=0\quad \mbox{on}\quad \overline{\mathrm{O_1O_2}}.
\end{equation}

 Thus, by (\ref{cdh}), (\ref{72205}), (\ref{72506}), (\ref{61711})--(\ref{72502}) one has
\begin{equation}\label{72208}
   \hat{\partial}_{-}x>0\quad \mbox{and}\quad   \hat{\partial}_{+}x>0\quad \mbox{in}\quad (\Delta_h\cup \Delta_2)\setminus (\wideparen{\mathrm{O_1B_1}}\cup\overline{\mathrm{O_1O_2}}).
 \end{equation}
 Combining this with (\ref{HT}) we have
 \begin{equation}\label{72209}
   \hat{\partial}_{-}y=  \lambda_{-}\hat{\partial}_{-}x\quad \mbox{and}\quad    \hat{\partial}_{+}y=  \lambda_{+}\hat{\partial}_{+}x\quad \mbox{in}\quad (\Delta_h\cup \Delta_2)\setminus (\wideparen{\mathrm{O_1B_1}}\cup\overline{\mathrm{O_1O_2}}).
 \end{equation}

Let $\mathrm{Z}_1=(r_{+}^1, r_{-}^1)$ and $\mathrm{Z}_2=(r_{+}^2, r_{-}^2)$
be arbitrary two points in $(\Delta_h\cup \Delta_2)\setminus (\wideparen{\mathrm{O_1B_1}}\cup\overline{\mathrm{O_1O_2}})$. There are the following four cases: (1) $r_{+}^2<r_{+}^1$ and $r_{-}^2>r_{-}^1$; (2) $r_{+}^2<r_{+}^1$ and $r_{-}^2<r_{-}^1$; (3) $r_{+}^2=r_{+}^1$ and $r_{-}^2>r_{-}^1$; (4) $r_{+}^2<r_{+}^1$ and $r_{-}^2=r_{-}^1$.

For Case (1), see Figure \ref{Figure14} (left),  by (\ref{72208}) we have
\begin{equation}\label{9601a}
x(\mathrm{Z}_2)-x(\mathrm{Z}_1)=\underbrace{\int_{r_{+}^2}^{r_{+}^1}\hat{\partial}_{+}x(r_{+}, r_{-}^1) {\rm d}r_{+}}_{>0}+
\underbrace{\int_{r_{-}^1}^{r_{-}^2}\hat{\partial}_{-}x(r_{+}^2, r_{-}) {\rm d}r_{-}}_{>0}>0.
\end{equation}

For Case (2),  see Figure \ref{Figure14} (right), we have
\begin{equation}\label{9602b}
x(\mathrm{Z}_2)-x(\mathrm{Z}_1)=\int_{r_{+}^2}^{r_{+}^1}\hat{\partial}_{+}x(r_{+}, r_{-}^1) {\rm d}r_{+}-
\int_{r_{-}^2}^{r_{-}^1}\hat{\partial}_{-}x(r_{+}^2, r_{-}) {\rm d}r_{-}
\end{equation}
and
\begin{equation}
y(\mathrm{Z}_2)-y(\mathrm{Z}_1)=\int_{r_{+}^2}^{r_{+}^1}\lambda_{+}(r_{+}, r_{-}^1)\hat{\partial}_{+}x(r_{+}, r_{-}^1) {\rm d}r_{+}-
\int_{r_{-}^2}^{r_{-}^1}\lambda_{-}(r_{+}^2, r_{-})\hat{\partial}_{-}x(r_{+}^2, r_{-}) {\rm d}r_{-}.
\end{equation}
By (\ref{72205}) we have
\begin{equation}
\lambda_{+}(r_{+}, r_{-}^1)>\lambda_{+}(r_{+}^2, r_{-}^1)>\lambda_{-}(r_{+}^2, r_{-}^1)\quad \mbox{for}\quad r_{+}^2<r_{+}<r_{+}^1,
\end{equation}
and
\begin{equation}\label{9602c}
\lambda_{-}(r_{+}^2, r_{-})<\lambda_{-}(r_{+}^2, r_{-}^1)\quad \mbox{for}\quad r_{-}^2<r_{-}<r_{-}^1.
\end{equation}
Combining (\ref{72208}) and (\ref{9602b})--(\ref{9602c}), we know that $x(\mathrm{Z}_2)-x(\mathrm{Z}_1)$ and  $y(\mathrm{Z}_2)-y(\mathrm{Z}_1)$ can not be zero at the same time for Case (2).

By (\ref{72209}) we have $x(\mathrm{Z}_2)\neq x(\mathrm{Z}_1)$ for cases (3) and (4).
Since $\mathrm{Z}_1$ an $\mathrm{Z}_2$ can be arbitrary, there are no two points from $(r_{+}, r_{-})\in (\Delta_h\cup \Delta_2)\setminus (\wideparen{\mathrm{O_1B_1}}\cup\overline{\mathrm{O_1O_2}})$ map to one point in the $(x,y)$-plane. This completes the proof.
\end{proof}

Let
$$
\Sigma_h=\big\{(x, y)\mid (x, y)=(x_{2}, y_{2})(r_{+}, r_{-}), ~ (r_{+}, r_{-})\in \Delta_h\big\}
$$
and
$$
\Sigma_2=\big\{(x, y)\mid (x, y)=(x_{2}, y_{2})(r_{+}, r_{-}), ~ (r_{+}, r_{-})\in \Delta_2\big\}.
$$
We denote the inverse function of $(x, y)=(x_{2}(r_{+}, r_{-}), y_{2}(r_{+}, r_{-}))$, $(r_{+}, r_{-})\in (\Delta_h\cup \Delta_2)\setminus \overline{\mathrm{O_1O_2}}$ by
$$
(r_{+}, r_{-})=(\tilde{r}_{+}(x,y), \tilde{r}_{-}(x,y)), \quad (x,y)\in (\Sigma_h\cup\Sigma_2)\setminus\mathrm{O}.
$$

Let
$$
(u_2, v_2)(x,y)=\big(u(\tilde{r}_{+}(x,y), \tilde{r}_{-}(x,y)), v(\tilde{r}_{+}(x,y), \tilde{r}_{-}(x,y))\big),  \quad (x,y)\in (\Sigma_h\cup\Sigma_2)\setminus\mathrm{O}.
$$
Then  $(u, v)=(u_2, v_2)(x, y)$ satisfies (\ref{42501}) on the domain $(\Sigma_h\cup\Sigma_2)\setminus\Gamma_{h}$, and $(u_2, v_2)(x, y)=(u_b, v_b)(x,y)$ for $(x,y)\in \Gamma_h$.
Moreover, the solution in  $\Sigma_2$ is a $C_{+}$ type centered wave with $\mathrm{O}$ as the center point.

\subsubsection{\bf Existence of a transonic tail shock}
In order to obtain a transonic tail shock to the IBVP (\ref{42501}), (\ref{RBDD1}), we consider (\ref{42501}) on an angular domain $R(\delta)=\{(x, y)\mid  0\leq x\leq \delta, w(x)\leq y\leq \psi_t(x)\}$,
where $y=\psi_t(x)$ ($\psi_t(0)=0$) is an unknown curve representing the tail shock. We prescribe the following  boundary conditions:

\noindent
on $y=\psi_t(x)$,
\begin{equation}\label{66211}
(\rho  u -\rho_fu_f)(u -u_f)+  (\rho  v -\rho_fv_f)(v -v_f)=0
\end{equation}
and
\begin{equation}\label{66221}
\frac{{\rm d}\psi_t(x)}{{\rm d}x}=-\frac{u-u_f}{v-v_f}, \quad \psi_t(0)=0,
\end{equation}
where $(u_f, v_f)(x, y)=(u_2, v_2)(x,y)$ for $(x, y)\in \Sigma_h\cup\Sigma_2$;

\noindent
on $y=w(x)$,
\begin{equation}\label{66231}
v(x,y)=u(x,y)w'(x).
\end{equation}

We still use $\phi$ to denote the inclination angle of the tail shock.
As in (\ref{52501}), we have that along the tail shock,
\begin{equation}\label{73003}
\kappa t_{+}\sin(\beta-\phi)\bar{\partial}_{+}c+
\kappa t_{-}\sin(\phi-\alpha)\bar{\partial}_{-}c+(N_f-N)\bar{\partial}_{_\Gamma}\phi=\mathcal{L}_{11}\bar{\partial}_{_\Gamma}r_{-}^{f}+
\mathcal{L}_{21}\bar{\partial}_{_\Gamma}r_{+}^{f},
\end{equation}
where $$\mathcal{L}_{11}=-\frac{q_f\sin(\alpha_f-\phi)}{2\cos A_f},\quad \mathcal{L}_{21}=-\frac{q_f\sin(\beta_f-\phi)}{2\cos A_f},$$
and $t_{\pm}$ are the same as that defined in (\ref{80601}).

As in (\ref{52505}), we have
\begin{equation}\label{73005}
\kappa t_{+}\sin(\beta-\chi)\bar{\partial}_{+}c+
\kappa t_{-}\sin(\chi-\alpha)\bar{\partial}_{-}c=
\mathcal{L}_{12}\bar{\partial}_{_\Gamma}r_{-}^{f}+\mathcal{L}_{22}\bar{\partial}_{_\Gamma}r_{+}^{f},
\end{equation}
where
$$
\mathcal{L}_{12}=\frac{1}{\sqrt{G_{u}^2+G_{v}^2}}\left[\frac{q_f(\rho-\rho_f)L_f\sin(\phi-\alpha_f)}{2\cos A_f}+\rho_fq_f(N-N_f)\left(\frac{\sin(\phi-\sigma_f)}{\sin(2A_f)}-\frac{\cos(\alpha_f-\phi)}{2\cos A_f}\right)\right]
$$
and
$$
\mathcal{L}_{22}=\frac{1}{\sqrt{G_{u}^2+G_{v}^2}}\left[\frac{q_f(\rho-\rho_f)L_f\sin(\phi-\beta_f)}{2\cos A_f}+\rho_fq_f(N-N_f)\left(\frac{\sin(\sigma_f-\phi)}{\sin(2A_f)}-\frac{\cos(\beta_f-\phi)}{2\cos A_f}\right)\right].
$$

Then as in (\ref{6612}) we have
\begin{equation}\label{73006}
\mathcal{A}\bar{\partial}_{_\Gamma}\phi=\mathcal{B}(\bar{\partial}_{-}c-\bar{\partial}_{+}c)
+\mathcal{C}_1\bar{\partial}_{_\Gamma}r_{-}^{f}+\mathcal{C}_2
\bar{\partial}_{_\Gamma}r_{+}^{f},
\end{equation}
where $\mathcal{A}$ and $\mathcal{B}$ are the same as that defined in (\ref{902a}) and (\ref{902b}), respectively,
$$
\mathcal{C}_{i}=\mathcal{L}_{i1}[t_{+}\sin(\beta-\chi)+t_{-}\sin(\chi-\alpha)]-\mathcal{L}_{i2}[t_{+}\sin(\beta-\phi)+t_{-}\sin(\phi-\alpha)], \quad i=1, 2.
$$

In order to overcome the difficulty caused by the multi-value singularity of $(u_f, v_f)$ at the origin, as in Section 2.3.3 we take the coordinate transformation
$$
x=x, \quad y=\eta(x, \zeta)
$$
where $0<x<\delta$, $\tan\phi_{pr}\leq \zeta\leq \tan\phi_{po}$,  and the function $\eta(x, \zeta)$ is defined by
$$
\eta_x(x, \zeta)=\tan(\alpha_f(x, \zeta)), \quad \eta_x(0, \zeta)=\zeta, \quad \eta(0, \zeta)=0.
$$
Then, $u_f$, $v_f$, $\alpha_f$, and $\beta_f$ can be seen as functions of $x$ and $\zeta$.
At the point $(x, \zeta)=(0, \tan\phi_{pr})$, we have
$$
 \tau_f=\tau_w, \quad (u_f, v_f)=(\hat{u}(\tau_w), \hat{v}(\tau_w)), \quad \phi=\alpha_f=\phi_{pr},\quad \mbox{and}\quad \mathcal{C}_1=0.
$$

In terms of the $(x, \zeta)$-coordinates, we have
$$
\bar{\partial}_{_\Gamma}=\cos\phi\partial_x+\cos\phi\eta_{\zeta}^{-1}(\tan\phi-\tan\alpha_f)\partial_{\zeta},
$$
where $\eta_{\zeta}(x, \eta)=\displaystyle\int_{0}^{x}\sec^2(\alpha_f(r, \zeta))\partial_{\zeta}\alpha_f(r, \zeta){\rm d}r$.
Therefore, by (\ref{73006}) we have
\begin{equation}\label{81001}
\begin{aligned}
&\bar{\partial}_{_\Gamma}(\phi-\alpha_f)+
\cos\phi\eta_{\zeta}^{-1}\partial_{\zeta}\alpha_f(\tan\phi-\tan\alpha_f)\\
~=~&\mathcal{A}^{-1}\left[\mathcal{B}(\bar{\partial}_{-}c-\bar{\partial}_{+}c)
+\mathcal{C}_1\bar{\partial}_{_\Gamma}r_{-}^{f}+\mathcal{C}_2
\bar{\partial}_{_\Gamma}r_{+}^{f}-\mathcal{A}\cos\phi\partial_x\alpha_f\right].
\end{aligned}
\end{equation}

For the needs of the following discussion, we define the following constants:
$$
\boxed{
\begin{aligned}
&\mathcal{A}_{*}=\mathcal{A}\big({u}_b^t(\tau_w), {v}_b^t(\tau_w), \hat{u}(\tau_w), \hat{v}(\tau_w), \phi_{pr}\big), \quad  \mathcal{B}_{*}=\mathcal{B}\big({u}_b^t(\tau_w), {v}_b^t(\tau_w), \hat{u}(\tau_w), \hat{v}(\tau_w), \phi_{pr}\big), \\& \mathcal{C}_{*}=\mathcal{C}_{2}\big({u}_b^t(\tau_w), {v}_b^t(\tau_w), \hat{u}(\tau_w), \hat{v}(\tau_w), \phi_{pr}\big),
\quad \mathcal{L}_{21}^{*}=\mathcal{L}_{21}\big({u}_b^t(\tau_w), {v}_b^t(\tau_w), \hat{u}(\tau_w), \hat{v}(\tau_w), \phi_{pr}\big)
\\& \mathcal{L}_{22}^{*}=\mathcal{L}_{22}\big({u}_b^t(\tau_w), {v}_b^t(\tau_w), \hat{u}(\tau_w), \hat{v}(\tau_w), \phi_{pr}\big), \quad  \mathcal{\chi}_{*}=\mathcal{\chi}\big({u}_b^t(\tau_w), {v}_b^t(\tau_w), \hat{u}(\tau_w), \hat{v}(\tau_w)\big),\\&q_{*}=\sqrt{({u}_b^t(\tau_w))^2+({v}_b^t(\tau_w))^2},\quad N_*={u}_b^t(\tau_w)\sin\phi_{pr} +{v}_b^t(\tau_w)\cos\phi_{pr}, \quad \alpha_*=\alpha({u}_b^t(\tau_w), {v}_b^t(\tau_w)),\\& \beta_*=\beta({u}_b^t(\tau_w), {v}_b^t(\tau_w)),\quad r_{+}^{*}=r_{+}({u}_b^t(\tau_w), {v}_b^t(\tau_w)), \quad  r_{-}^{*}=r_{-}({u}_b^t(\tau_w), {v}_b^t(\tau_w)),
 \\&\kappa_*=\kappa(\tau_{pr}(\tau_w)), \quad c_*=c(\tau_{pr}(\tau_w)),\quad\varrho_*=\hat{\partial}_{+}x_2(r_{+}^{po}, r_{-}^{pr}), \quad \varsigma_{\pm}=\frac{\partial\alpha(r_{+}^{po}, r_{-}^{pr})}{\partial r_{\pm}}.
\end{aligned}}
$$
As in (\ref{6611}), we have $\mathcal{A}_*>0$.
We also have $\varsigma_{-}<0$ and $\varrho_*>0$.

We define
\begin{equation}\label{zetastar}
\zeta_*=\frac{(\mathcal{C}_{*}-\mathcal{A}_*\varsigma_+)(1+\tan^2\theta_w)^{\frac{3}{2}}\cos\phi_{pr}}{2\mathcal{B}_{*}\kappa_{*}q_{*}
\varrho_*}.
\end{equation}
Then when $w''(0)>\zeta_*$,
$$
\mathcal{B}(\bar{\partial}_{-}c-\bar{\partial}_{+}c)+\mathcal{C}_2
\cos\phi\partial_{x}r_{+}^f-\mathcal{A}\cos\phi\partial_x\alpha_f=
\frac{2\kappa_*q_*w''(0)\mathcal{B}_*}{(1+\tan^2\theta_w)^{\frac{3}{2}}}-\frac
{(\mathcal{C}_{*}-\mathcal{A}_*\varsigma_+)\cos\phi_{pr}}
{\varrho_*}>0
$$
for $(x,\zeta)=(0, \tan\phi_{pr})$.
This implies that the tail shock will enter into the the centered wave region $\Sigma_2$ from the origin.
Then the local existence can be established by the solvability condition {\bf (H3)}.
This actually completes the proof of the following lemma.
\begin{lem}
Assume  that assumptions {\bf (H1)--(H3)} hold.
Then when $w''(0)>\zeta_*$ the free boundary problem  (\ref{42501}), (\ref{66211})--(\ref{66231}) admits a local classical solution.
\end{lem}

For convenience we denote the local solution of the typical free boundary problem  (\ref{42501}), (\ref{66211})--(\ref{66231}) by $(u, v)=(u_d, v_d)(x,y)$.
Let
\begin{equation}\label{724081}
(u, v)=\left\{
         \begin{array}{ll}
          (u_1, v_1)(x,y), & \hbox{$\psi_h(x)<y<\psi_h(\delta)$, $0<x<\delta$;} \\[2pt]
           (u_2, v_2)(x,y), & \hbox{$\psi_t(x)<y<\psi_h(x)$, $0<x<\delta$;} \\[2pt]
           (u_d, v_d)(x, y), & \hbox{$w(x)<y<\psi_t(x)$, $0<x<\delta$,}
         \end{array}
       \right.
\end{equation}
where $\delta$ is a small positive constant.
Then, the function $(u, v)(x, y)$ defined in (\ref{724081}) is a piecewise smooth local solution to the IBVP (\ref{42501}), (\ref{RBDD1}) for $w''(0)>\zeta_*$.
This completes the proof of Theorem \ref{thm23} for $w''(0)>\zeta_*$.

\subsubsection{\bf Existence of a pre-sonic tail shock}
We now prove Theorem \ref{thm23} for $w''(0)<\zeta_*$.
Let $\overline{\mathrm{O_2E_2}}=\{(r_{+}, r_{-})\mid r_{-}=r_{-}^{pr}, ~r_{+}^{po}-\eta_0\leq r_{+}< r_{+}^{po}\}$  and $\wideparen{\mathrm{OE}}=\{(x, y)\mid (x, y)=(x_2, y_2)(r_{+}, r_{-}), (r_{+}, r_{-})\in \overline{\mathrm{O_2E_2}}\}$.
For convenience, we denote the $C_{+}$ characteristic curve $\wideparen{\mathrm{OE}}$ by $y=\varphi(x)$, $0\leq x\leq x_2(r_{+}^{po}-\eta_0, r_{-}^{pr})$.
So, we have $\tan\alpha_2(x, \varphi(x))=\varphi'(x)$ along $\wideparen{\mathrm{OE}}$.

We consider (\ref{42501}) on an angular domain $R(\delta)=\{(x, y)\mid  0\leq x\leq \delta, w(x)\leq y\leq \varphi(x)\}$, where $\delta$ is a small positive constant.
We prescribe the following boundary conditions:
\begin{equation}\label{51601}
\left\{
  \begin{array}{ll}
    (u, v)=(u_2, v_2)(x, y), & \hbox{$(x, y)\in \wideparen{\mathrm{OE}}$;} \\[3pt]
    v(x,y)=u(x,y)w'(x), & \hbox{$y=w(x)$, $x>0$.}
  \end{array}
\right.
\end{equation}
The main result of this part is stated as the following lemma.
\begin{lem}
Assume  that assumptions {\bf (H1)--(H3)} hold.
Assume furthermore $w''(0)<\zeta_*$. Then when $\delta>0$ is sufficiently small, the boundary value problem  (\ref{42501}), (\ref{51601}) admits a piecewise smooth solution on $R(\delta)$. Moreover, the solution contains a single shock issued from the origin and the shock is pre-sonic.
\end{lem}

 The problem (\ref{42501}), (\ref{51601}) is actually a free boundary problem. However, since the states ahead of the shock are also unknown, this problem is different from the free boundary problem discussed in Li and Yu (\cite{Li-Yu}, Chapter 3.2) in the context of supersonic plane flow past a curved wedge. Moreover, the flow upstream of the shock is not $C^1$ smooth up to the shock boundary. So, in order to use idea of Li et. al. \cite{Li-Yu} to fix the free boundary and transfer the free boundary problem to a typical boundary value problem in functional form, we
first need to find the flow upstream of the pre-sonic shock provided that the shock curve itself is given.

Let
\begin{equation}\label{K}
\mathcal{K}_*=\frac{\sec^3\phi_{pr}}{\mathcal{A}_*}\left(\frac{2\kappa_*q_*w''(0)\mathcal{B}_*}{(1+\tan^2\theta_w)^{\frac{3}{2}}}-\frac
{\mathcal{C}_{*}\cos\phi_{pr}}{\varrho_*}\right).
\end{equation}
Then,
when $w''(0)<\zeta_*$, there holds
\begin{equation}\label{6601a}
\varphi''(0)=-\frac{\varsigma_+}{\varrho_*\cos^2\phi_{pr}}>\mathcal{K}_*.
\end{equation}

\begin{figure}[htbp]
\begin{center}
 \includegraphics[scale=0.48]{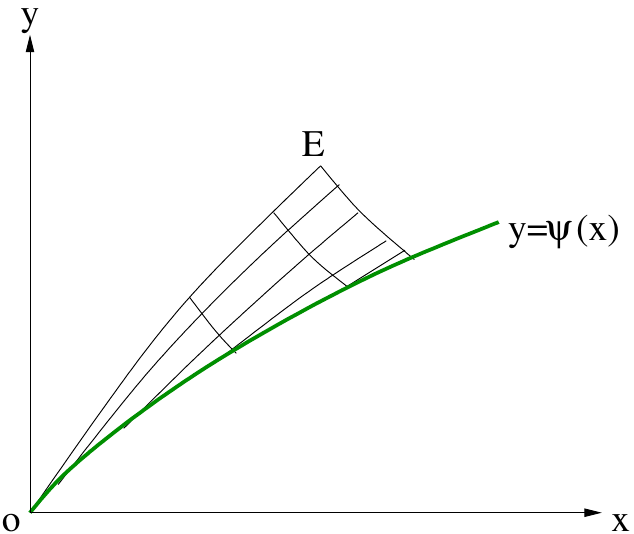}
\caption{\footnotesize Envelope problem (\ref{42501}), (\ref{51603}).}
\label{Figure15a}
\end{center}
\end{figure}

Let $y=\psi(x)$, $x\geq0$ be a given smooth function which satisfies the following assumptions:
\begin{equation}\label{51602}
\psi(0)=0; \quad
\psi'(0)=\tan\phi_{pr}; \quad \psi''(0)=\mathcal{K}_*; \quad  \mathcal{K}_*-\varepsilon_0<\psi''(x)<\mathcal{K}_*+\varepsilon_0\quad\mbox{for}\quad x\geq 0,
\end{equation}
where $\varepsilon_0$ is a small positive constant to be specified later.
We consider (\ref{42501}) with the following boundary conditions
\begin{equation}\label{51603}
\left\{
  \begin{array}{ll}
    (u, v)=(u_2, v_2)(x, y), & \hbox{$(x, y)\in \wideparen{\mathrm{OE}}$;} \\[3pt]
    \alpha(x, y)=\arctan \psi'(x), & \hbox{$y=\psi(x)$, $x>0$.}
  \end{array}
\right.
\end{equation}
\begin{lem}\label{lem516}
Assume $\eta_0$ is sufficiently small. Then
the boundary value problem  (\ref{42501}), (\ref{51603}) admits a classical solution on a curvilinear triangle domain bounded by $\wideparen{\mathrm{OE}}$, $y=\psi(x)$, and a forward $C_{-}$ characteristic curve issued from the point $\mathrm{E}$.
Moreover, the curve $y=\psi(x)$ is the envelope of all the $C_{+}$ characteristic curves; see Figure \ref{Figure15a}.
\end{lem}

\begin{proof}
We first define a function set
$$
\begin{aligned}
\Upsilon&:=\big\{g(r_{+})\in C^{1}[r_{+}^{po}-\eta_0, r_{+}^{po}]~\big|~ g(r_{+}^{po})=r_{-}^{pr},\quad  g'(r_{+}^{po})=\mathcal{G}_{*},\\& \qquad\qquad \qquad\qquad \frac{3\mathcal{G}_{*}}{2}\leq g'(r_{+})\leq\frac{\mathcal{G}_{*}}{2}\quad \mbox{for}\quad r_{+}^{po}-\eta_0\leq r_{+}\leq r_{+}^{po}\big\},
\end{aligned}
$$
where
$$
\mathcal{G}_{*}=-\frac{\mathcal{K}_*\varrho_*\cos^2\phi_{pr}}{\varsigma_-}-\frac{\varsigma_+}{\varsigma_-}.
$$
In view of (\ref{6601a}), we have $\mathcal{G}_{*}<0$.

\begin{figure}[htbp]
\begin{center}
\includegraphics[scale=0.48]{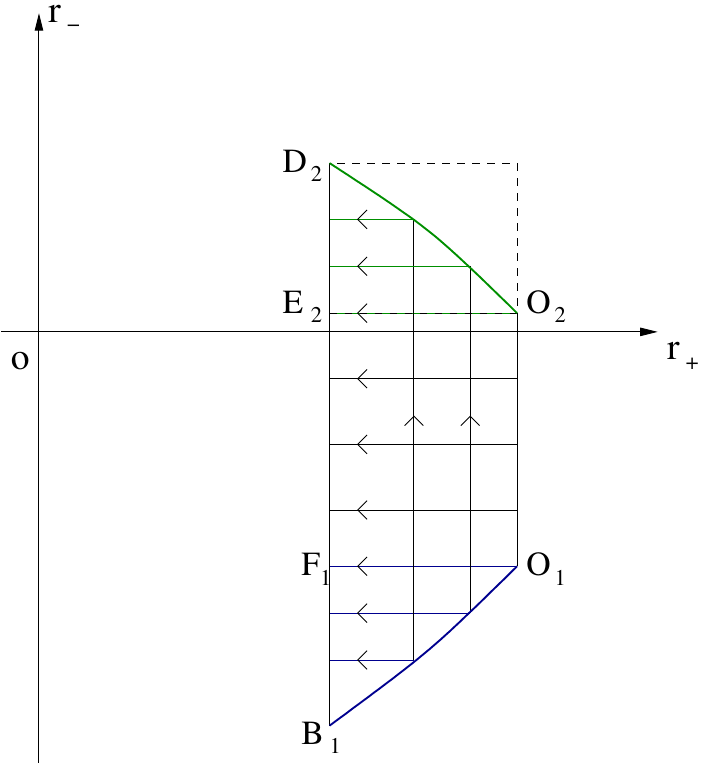}
\caption{\footnotesize Domains $\square(g)$ and $\Delta(g)$.}
\label{Figure15}
\end{center}
\end{figure}


For any given $g(r_{+})\in \Upsilon$,
we let
$\wideparen{\mathrm{O_2D_2}}=\{(r_{+}, r_{-})\mid r_{-}=g(r_{+}),~ r_{+}^{po}-\eta_0\leq r_{+}< r_{+}^{po}\}$; see Figure \ref{Figure15}.
On rectangular domain $\square(g)=\big\{(r_{+}, r_{-})\mid r_{-}^{pr}<r_{-}<g(r_{+}^{po}-\eta_0),~ r_{+}^{po}-\eta_0\leq r_{+}< r_{+}^{po}\big\}$ we consider (\ref{cdh}) with the boundary conditions:
\begin{equation}\label{72303}
(\hat{\partial}_{+}x)(r_{+}, r_{-})=(\hat{\partial}_{+}x_2)(r_{+}, r_{-}), \quad (r_{+}, r_{-})\in \overline{\mathrm{O_2E_2}};
\end{equation}
\begin{equation}\label{72304}
(\hat{\partial}_{-}x)(r_{+}, r_{-})=0, \quad (r_{+}, r_{-})\in\wideparen{\mathrm{O_2D_2}}.
\end{equation}

The linear  problem (\ref{cdh}), (\ref{72303})--(\ref{72304}) admits a unique $C^1$ solution  on $\overline{\square(g)}$. Moreover, for any small $\varepsilon>0$, when $\eta_0$ is sufficiently small, the solution satisfies
\begin{equation}\label{72305}
-\varepsilon<\hat{\partial}_{+}x-\varrho_*<\varepsilon\quad \mbox{and}\quad 0<\hat{\partial}_{-}x<\varepsilon\quad \mbox{on}\quad \overline{\Delta(g)}
\end{equation}
for any $g\in \Upsilon$, where $\Delta(g)=\big\{(r_{+}, r_{-})\mid r_{-}^{pr}<r_{-}<g(r_{+}),~ r_{+}^{po}-\eta_0\leq r_{+}< r_{+}^{po}\big\}$.

We denote by $(\hat{\partial}_{+}x, \hat{\partial}_{-}x)=(\xi_1(r_{+}, r_{-}), \xi_2(r_{+}, r_{-}))$ the solution of the problem (\ref{cdh}), (\ref{72303})--(\ref{72304}).
For any $(r_{+}, r_{-})\in \overline{\Delta(g)}$, we let
\begin{equation}\label{52201}
x(r_{+}, r_{-})=x_{2}(r_{+}, r_{-}^{pr})+\int_{r_{-}^{pr}}^{r_{-}} \xi_2(r_{+}, r){\rm d}r
\end{equation}
and
\begin{equation}\label{52202}
y(r_{+}, r_{-})=y_{2}(r_{+}, r_{-}^{pr})+\int_{r_{-}^{pr}}^{r_{-}} \lambda_{-}(r_{+}, r)\xi_2(r_{+}, r){\rm d}r.
\end{equation}
Then by $\hat{\partial}_{+}(\lambda_{-}\hat{\partial}_{-}x)=\hat{\partial}_{-}(\lambda_{+}\hat{\partial}_{+}x)$, the function $(x, y)=(x(r_{+}, r_{-}), y(r_{+}, r_{-}))$ satisfies (\ref{HT}) in $\overline{\Delta(g)}$.
By (\ref{72305}) we know that the mapping $(x, y)=(x(r_{+}, r_{-}), y(r_{+}, r_{-}))$,  $(r_{+}, r_{-})\in \overline{\Delta(g)}$ is injective.



Let the function $j(r_{+})\in C^1[r_{+}^{po}-\eta_0, r_{+}^{po}]$ be defined by
\begin{equation}\label{52203}
\alpha\big(r_{+}, j(r_{+})\big)=\arctan \psi'(x(r_{+}, g(r_{+}))), \quad j(r_{+}^{po})=r_{-}^{pr}.
\end{equation}
A direct computation yields
\begin{equation}\label{52204}
j'(r_{+})=-\frac{\psi''(x(r_{+}, g(r_{+})))\hat{\partial}_{+}x(r_{+}, g(r_{+}))}{(1+[\psi'(x(r_{+}, g(r_{+})))]^2)\alpha_{r_{-}}(r_{+}, j(r_{+}))}-\frac{\alpha_{r_{+}}(r_{+}, j(r_{+}))}{\alpha_{r_{-}}(r_{+}, j(r_{+}))}.
\end{equation}
So, when $\eta_0$ and $\varepsilon_1$ are sufficiently small, $j(r_{+})\in \Upsilon$ for any $g(r_{+})\in \Upsilon$.
We then define a mapping $T$ which maps $\Upsilon$ into $\Upsilon$.

For $\hat{g}$, $\check{g}\in \Upsilon$, we
let $\hat{j}=T\hat{g}$ and $\check{j}=T\check{g}$. As in (\ref{52201}) and (\ref{52202}), we use $(\hat{x}, \hat{y})(r_{+},r_{-})$ and $(\check{x}, \check{y})(r_{+}, r_{-})$ to denote the solution of the corresponding problem (\ref{HT}), (\ref{72303})--(\ref{72304}) on $\overline{\square(\hat{g})}$ and $\overline{\square(\check{g})}$.
Then we have
\begin{equation}\label{52207}
\alpha\big(r_{+}, \hat{j}(r_{+})\big)=\arctan \psi'(\hat{x}(r_{+}, \hat{g}(r_{+})))\quad \mbox{and}\quad
\alpha\big(r_{+}, \check{j}(r_{+})\big)=\arctan \psi'(\check{x}(r_{+}, \check{g}(r_{+})))
\end{equation}
for $r_{+}\in [r_{+}^{po}-\eta_0, r_{+}^{po}]$.
From (\ref{cdh}) and (\ref{72303})--(\ref{72304}) we have
\begin{equation}\label{52206}
\big\|\hat{\partial}_{-}(\hat{x}-\check{x})\big\|_{0; \overline{\square(\hat{g})}\cap\overline{\square(\check{g})}}<\mathcal{M}\big\|\hat{g}-\check{g}\big\|_{0; [r_{+}^{po}-\eta_0, r_{+}^{po}]}.
\end{equation}
where $\mathcal{M}$ is a positive constant depending only on $\mathcal{G}_{*}$, $\varrho_*$,  $r_{+}^{po}$, and $r_{-}^{pr}$.

Combining  (\ref{72305}), (\ref{52207}), and (\ref{52206}),
 there exist small $\eta_0>0$ and $0<\nu<1$  such that for any $\hat{g}$, $\check{g}\in \Upsilon$,
\begin{equation}\label{51705}
\big\|T\hat{g}-T\check{g}\big\|_{0; [r_{+}^{po}-\eta_0, r_{+}^{po}]}~\leq~\nu \big\|\hat{g}-\check{g}\big\|_{0; [r_{+}^{po}-\eta_0, r_{+}^{po}]},
\end{equation}





Now, we choose any $g_0\in \Upsilon$, we let $g_{i+1}=Tg_{i}$, $i=1, 2, 3, \cdot\cdot\cdot$.
Then there is a function $g\in C[r_{+}^{po}-\eta_0, r_{+}^{po}]$ such that
$$
\lim\limits_{i\rightarrow+\infty}g_i=g\quad \mbox{in}\quad C^0\big[r_{+}^{po}-\eta_0, r_{+}^{po}\big].
$$
Let $x=x(r_{+}, r_{-})$, $(r_{+}, r_{-})\in \overline{\Delta(g)}$ be defined the same as that in (\ref{52201}).
Since $g_i\in \Upsilon$, $i=1, 2, 3, \cdot\cdot\cdot$, we have $g\in C^{0,1}\big[r_{+}^{po}-\eta_0, r_{+}^{po}\big]$,
and hence $x(r_{+}, r_{-})\in C^1(\overline{\Delta(g)})$.
Therefore, by the definition of the mapping $T$ and (\ref{52204}), we have $g\in \Upsilon$ and $g=Tg$.

This completes the proof.
\end{proof}


For the needs of the following discussion, we use
$(u, v)=(u_\psi, v_\psi)(x, y)$ to denote the solution of the problem (\ref{42501}), (\ref{51603}).
We now  consider (\ref{42501}) on an angular domain $R_0(\delta)=\{(x, y)\mid  0\leq x\leq \delta, w(x)\leq y\leq \psi(x)\}$,
where the function $y=\psi(x)$ is an unknown function representing the pre-sonic shock and
is required to satisfy
$$
\psi(0)=0, \quad \psi'(0)=\tan\phi_{pr},  \quad \psi''(0)=\mathcal{K}_*,\quad \mbox{and}\quad \mathcal{K}_*-\varepsilon_0<\psi''(x)<\mathcal{K}_*+\varepsilon_0\quad\mbox{for}\quad  0\leq x\leq \delta.
$$
 We prescribe the following  boundary conditions:

\noindent
on $y=\psi(x)$,
\begin{equation}\label{6621A}
G(u, v, u_{\psi}, v_{\psi})=(\rho  u -\rho_{\psi}u_{\psi})(u -u_{\psi})+  (\rho  v -\rho_{\psi}v_{\psi})(v -v_{\psi})=0
\end{equation}
and
\begin{equation}\label{6622A}
\frac{{\rm d}\psi(x)}{{\rm d}x}=-\frac{u-u_{\psi}}{v-v_{\psi}};
\end{equation}

\noindent
on $y=w(x)$,
\begin{equation}\label{6623A}
v(x,y)=u(x,y)w'(x).
\end{equation}
The problem (\ref{42501}), (\ref{6621A})--(\ref{6623A}) is a free boundary problem. We can then adopt the framework proposed by Li and Yu (\cite{Li-Yu}, Chapter 3.2) to solve this problem.

\begin{lem}\label{lem60701}
Assume  that assumptions {\bf (H1)--(H3)} hold.
Assume furthermore that $\delta$ is sufficiently small. Then
the free boundary problem (\ref{42501}), (\ref{6621A})--(\ref{6623A}) admits a solution on $R_0(\delta)$, and the free boundary is a pre-sonic shock boundary.
\end{lem}
\begin{proof}
In order to fix the free boundary,
we introduce the  transformation
\begin{equation}\label{52209}
\xi=x, \quad \eta=\frac{y-w(x)}{X(x)},
\end{equation}
where $X(x)=\frac{\psi(x)-w(x)}{x}$ for $0<x<\delta$ and $X(0)=\psi'(0)-w'(0)$.
Then the slope boundary is changed into $\{(\xi, \eta)\mid\eta=0, \xi>0\}$ and the free boundary  is changed into $\{(\xi, \eta)\mid\eta=\xi, 0\leq \xi\leq\delta\}$.

From (\ref{52209}) we also have
\begin{equation}
\partial_y=X^{-1}(\xi)\partial_\eta\quad \mbox{and} \quad \partial_x=\partial_\xi-X^{-1}(\xi)\big(\eta X'(\xi)+w'(\xi)\big)\partial_\eta,
\end{equation}
where $$X'(\xi)=\frac{\psi'(\xi)-w'(\xi)}{\xi}-\frac{\psi(\xi)-w(\xi)}{\xi^2}.$$

By the transformation (\ref{52209}), the system (\ref{52801}) is changed into
\begin{equation}\label{52801B}
\left\{
  \begin{array}{ll}
    \partial_{\xi}r_{-}+X^{-1}(\xi)\big(\lambda_{+}-\eta X'(\xi)-w'(\xi)\big)\partial_{\eta}r_{-}=0,  \\[4pt]
\partial_{\xi}r_{+}+X^{-1}(\xi)\big(\lambda_{-}-\eta X'(\xi)-w'(\xi)\big)\partial_{\eta}r_{+}=0.
  \end{array}
\right.
\end{equation}
The boundary conditions  (\ref{6621A})--(\ref{6623A}) are respectively changed into:

\noindent
on $\eta=\xi$,
\begin{equation}\label{6621B}
G(u, v, u_{\psi}, v_{\psi})=0,
\end{equation}
\begin{equation}\label{6622B}
\frac{{\rm d}\psi(\xi)}{{\rm d}\xi}=-\frac{u-u_{\psi}}{v-v_{\psi}}, \quad \psi(0)=0;
\end{equation}
\noindent
on $\eta=0$,
\begin{equation}\label{6623B}
v= w'(\xi)u.
\end{equation}

We take the following iteration scheme to solve the fixed boundary value problem (\ref{52801B}), (\ref{6621B})--(\ref{6623B}) on $\Xi(\delta):=\{(\xi, \eta)\mid 0<\eta< \xi, 0<\xi<\delta\}$.

\vskip 2pt
 Let the constants $\ell_u$ and $\ell_v$ be defined by
\begin{equation}\label{6703a}
\left\{
  \begin{array}{ll}
    \ell_u\cos\chi_*+\ell_v\sin\chi_*+\mathcal{L}_{22}^{*}\varrho_{*}^{-1}=0,  \\[4pt]
   \ell_u\cos\phi_{pr}+\ell_v\sin\phi_{pr}+(c(\tau_w)-N_*)\mathcal{K}_*\cos^2\phi_{pr}
+\mathcal{L}_{21}^{*}\varrho_{*}^{-1}=0.
  \end{array}
\right.
\end{equation}
We define the following constants:
$$
\begin{aligned}
&\ell_3=c_*^{-1}\sec\phi_{pr}(\ell_u\cos\alpha_*+\ell_v\sin\alpha_*), \quad \ell_4=-c_*^{-1}\sec\phi_{pr}(\ell_u\cos\beta_*+\ell_v\sin\beta_*), \\&\qquad\qquad \ell_1=\frac{\tan\beta_*-w'(0)}{\tan\beta_*-\tan\phi_{pr}}\ell_3,\quad \mbox{and}\quad
\ell_2=\frac{\tan\alpha_*-w'(0)}{\tan\alpha_*-\tan\phi_{pr}}\ell_4.
\end{aligned}
$$
We then define
$$
\begin{aligned}
\Phi(\delta)&:=\big\{(r_{+}, r_{-})\in C^1(\overline{\Xi(\delta)})\mid (r_{+}, r_{-})(0, 0)=(r_+^*, r_-^*), \quad \partial_{\xi}r_{+}(0, 0)=\ell_{1},\quad \partial_{\xi}r_{-}(0, 0)=\ell_{2}, \\&  \partial_{d}r_{+}(0, 0)=\ell_{3}, \quad \partial_{d}r_{-}(0, 0)=\ell_{4}; \quad |r_{+}-r_+^*|<\varepsilon_{1}, \quad   |r_{-}-r_-^*|<\varepsilon_{2}, \quad |\partial_{\xi}r_{+}-\ell_{1}|<\varepsilon_{3},\\& \quad \quad |\partial_{\xi}r_{-}-\ell_{2}|<\varepsilon_{4}, \quad
|\partial_{d}r_{+}-\ell_{3}|<\varepsilon_{5}, \quad \mbox{and}\quad |\partial_{d}r_{-}-\ell_{4}|<\varepsilon_{6}
\quad \mbox{on}\quad  \overline{\Xi(\delta)}  \big\},
\end{aligned}
$$
where $\partial_{d}=\partial_{\xi}+\partial_{\eta}$, $\varepsilon_i$ ($i=1, 2, 3, 4, 5, 6$) are small positive constants to be specified later.

For any given $(\bar{r}_{+}, \bar{r}_{-})\in \Phi(\delta)$, we let $(\bar{u}_f, \bar{v}_f)$ be a continuous function defined on $\{(\xi, \eta)\mid\eta=\xi, 0\leq \xi\leq\delta\}$ such that $(\bar{u}_f, \bar{v}_f)(0, 0)= (\tilde{u}, \tilde{v})(\phi_{pr})$, and
\begin{equation}
\left\{
  \begin{array}{ll}
  G(\bar{u}, \bar{v}, \bar{u}_f, \bar{v}_f)=0,\\[4pt]
 (\bar{v} -\bar{v}_f)\sin\bar{\alpha}_f+(\bar{u} -\bar{u}_f)\cos\bar{\alpha}_f=0
  \end{array}
\right.
\end{equation}
on $\{(\xi, \eta)\mid\eta=\xi, 0\leq \xi\leq\delta\}$,
where $\bar{\alpha}_f=\alpha(\bar{u}_f, \bar{v}_f)$ and $(\bar{u}, \bar{v})=(u(\bar{r}_{+}, \bar{r}_{-}), v(\bar{r}_{+}, \bar{r}_{-}))$.

Let $\bar{\psi}(\xi)$ be defined by
$$
\bar{\psi}'(\xi)=\tan\bar{\alpha}_f(\xi,\xi), \quad  \bar{\psi}(0)=0.
$$
By the definition of $\bar{\alpha}_f$ we have $\bar{\psi}'(0)=\tan\phi_{pr}$.

As in (\ref{73003}) and (\ref{73005}), we have
\begin{equation}\label{73003A}
\cos\bar{\alpha}_f\partial_d\bar{u}+\sin\bar{\alpha}_f\partial_d\bar{v}
+(\bar{N}_f-\bar{N})\partial_d\bar{\alpha}_f=\bar{\mathcal{L}}_{21}\partial_d \bar{r}_{+}^f,
\end{equation}
and
\begin{equation}\label{73005A}
\cos\bar{\chi}\partial_d\bar{u}+\sin\bar{\chi}\partial_d\bar{v}=\bar{\mathcal{L}}_{22}\partial_d \bar{r}_{+}^f,
\end{equation}
where
$$
\begin{aligned}
&\bar{N}=\bar{u}\sin\bar{\alpha}_f +\bar{v}\cos\bar{\alpha}_f, \quad \bar{N}_f=\bar{u}_f\sin\bar{\alpha}_f +\bar{v}_f\cos\bar{\alpha}_f, \\&
\bar{\chi}=\chi(\bar{u}, \bar{v}, \bar{u}_f, \bar{v}_f), \quad \bar{\mathcal{L}}_{2i}={\mathcal{L}}_{2i}(\bar{u}, \bar{v}, \bar{u}_f, \bar{v}_f, \bar{\alpha}_f)\quad (i=1, 2).
\end{aligned}
$$
Then by the second equation of (\ref{6703a}) we have
$$
\partial_d\bar{\alpha}_f=\mathcal{K}_*\cos^2\phi_{pr}\quad \mbox{at}\quad (0, 0).
$$
Hence,
$$
\bar{\psi}''(0)=\mathcal{K}_*.
$$
So, when $\varepsilon_i$  ($i=1, 2, 3, 4, 5, 6$) are  sufficiently small, there holds
$$
 \mathcal{K}_*-\varepsilon_0<\bar{\psi}''(\xi)<\mathcal{K}_*+\varepsilon_0\quad\mbox{for}\quad  0\leq \xi\leq \delta.
$$

We then consider (\ref{42501}) with the boundary conditions
\begin{equation}\label{51603C}
\left\{
  \begin{array}{ll}
    (u, v)=(u_2, v_2)(x, y), & \hbox{$(x, y)\in \wideparen{\mathrm{OE}}$;} \\[3pt]
    \alpha(x, y)=\arctan \bar{\psi}'(x), & \hbox{$y=\bar{\psi}(x)$, $0\leq x\leq \delta$.}
  \end{array}
\right.
\end{equation}
By Lemma \ref{lem516}, the problem (\ref{42501}), (\ref{51603C}) admits a solution denoted by $({u}_{\bar{\psi}}, {v}_{\bar{\psi}})(x, y)$ for convenience.

By the implicit function theorem, near the point $\big(r_{+}, r_{-}, r_{+}^f, r_{-}^f\big)=(r_{+}^{*}, r_{-}^{*}, r_{+}^{po}, r_{-}^{pr})$ the relation (\ref{80901}) can be represented by
$$
r_{+}=\mathcal{F}\big(r_{-}, r_{+}^{f}, r_{-}^f\big),
$$
which satisfies $r_{+}^{*}=\mathcal{F}(r_{-}^{*}, r_{+}^{po}, r_{-}^{pr})$.
Moreover, by assumption {\bf (H2)} we have
\begin{equation}\label{80902}
\big|\partial_{r_{-}}\mathcal{F}(r_{-}^{*}, r_{+}^{po}, r_{-}^{pr})\big|<1.
\end{equation}
We then consider
 \begin{equation}\label{52401}
\left\{
  \begin{array}{ll}
    \partial_{\xi}r_{-}+\bar{X}^{-1}(\xi)\big(\lambda_{+}(\bar{r}_{+}, \bar{r}_{-})-\eta \bar{X}'(\xi)-w'(\xi)\big)\partial_{\eta}r_{-}=0,  \\[4pt]
\partial_{\xi}r_{+}+\bar{X}^{-1}(\xi)\big(\lambda_{-}(\bar{r}_{+}, \bar{r}_{-})-\eta \bar{X}'(\xi)-w'(\xi)\big)\partial_{\eta}r_{+}=0,
  \end{array}
\right.
\end{equation}
with the boundary conditions
\begin{equation}
r_{+}(\xi, \eta)=\mathcal{F}\big(\bar{r}_{-}(\xi,\eta), {r}_{+, \bar{\psi}}(\xi), {r}_{-, \bar{\psi}}(\xi)\big)\quad \mbox{for}\quad \eta=\xi,
\end{equation}
\begin{equation}\label{52403}
r_{-}=-\bar{r}_{+}+2\arctan w'(\xi) \quad \mbox{on}\quad \eta=0.
\end{equation}
Here, ${r}_{\pm, \bar{\psi}}(\xi)=r_{\pm}({u}_{\bar{\psi}}(\xi, \bar{\psi}(\xi)), {v}_{\bar{\psi}}(\xi, \bar{\psi}(\xi)))$,
$\bar{X}(\xi)=\frac{\bar{\psi}(\xi)-w(\xi)}{\xi}$, and $\bar{X}'(\xi)=\frac{\bar{\psi}'(\xi)-w'(\xi)}{\xi}-\frac{\bar{\psi}(\xi)-w(\xi)}{\xi^2}$.
Let $(r_{+}, r_{-})(\xi, \eta)$, $(\xi, \eta)\in \overline{\Xi(\delta)}$ be the solution of the linear problem (\ref{52401}).
We then define a mapping $\Psi$:
$$
(r_{+}, r_{-})=\Psi(\bar{r}_{+}, \bar{r}_{-}), \quad (\bar{r}_{+}, \bar{r}_{-})\in \Phi(\delta).
$$

In view of (\ref{80902}), there exists sufficiently small $\varepsilon_i$ ($i=1, 2, 3, 4, 5, 6$) such that when $\delta$ is sufficiently small the mapping $\Psi$ satisfies the following properties (Li and Yu \cite{Li-Yu}, Chapter 3):
\begin{itemize}
  \item $\Psi$ maps $\Phi(\delta)$ into $\Phi(\delta)$;
  \item there exists a $0<\nu<1$ such that for any $(\bar{r}_{+}', \bar{r}_{-}')\in \Phi(\delta)$ and $(\bar{r}_{+}'', \bar{r}_{-}'')\in \Phi(\delta)$,
$$
\big\|\Psi(\bar{r}_{+}', \bar{r}_{-}')-\Psi(\bar{r}_{+}'', \bar{r}_{-}'')\big\|_{0; \Xi(\delta)}~<~\nu
\big\|(\bar{r}_{+}', \bar{r}_{-}')-(\bar{r}_{+}'', \bar{r}_{-}'')\big\|_{0; \Xi(\delta)}.
$$
\end{itemize}
Since the original system is a diagonalized system,  there exists a $(r_{+}, r_{-})\in  \Phi(\delta)$ such that
$(r_{+}, r_{-})=\Psi({r}_{+}, {r}_{-})$.
Moreover, we have
$$
G(u, v, u_f, v_f)=0, \quad G(u, v, u_{\psi}, v_{\psi})=0, \quad \mbox{and}\quad \alpha(u_f, v_f)=\alpha(u_\psi, v_\psi)\quad\mbox{on}\quad y=\psi(x).
$$
So,
we have
$$
(u_f, v_f)=(u_{\psi}, v_{\psi})\quad \mbox{on}\quad y=\psi(x).
$$

This completes the proof of Lemma \ref{lem60701}.
\end{proof}

From Lemmas \ref{lem516} and \ref{lem60701}, we immediately prove Lemma \ref{lem516}.
For convenience we denote the local solution of the boundary problem  (\ref{42501}), (\ref{51601}) by $(u, v)=(u_d, v_d)(x,y)$.
Let
\begin{equation}\label{724081a}
(u, v)=\left\{
         \begin{array}{ll}
          (u_1, v_1)(x,y), & \hbox{$\psi_j(x)<y<\psi_h(\delta)$, $0<x<\delta$;} \\[2pt]
           (u_2, v_2)(x,y), & \hbox{$\psi_t(x)<y<\psi_h(x)$, $0<x<\delta$;} \\[2pt]
           (u_d, v_d)(x, y), & \hbox{$w(x)<y<\psi_t(x)$, $0<x<\delta$,}
         \end{array}
       \right.
\end{equation}
where $\delta$ is a small positive constant.
Then, the $(u, v)$ defined in (\ref{724081a}) is a piecewise smooth solution to the problem (\ref{42501}), (\ref{RBDD1}) for $w''(0)<\zeta_*$.
This completes the proof of Theorem \ref{thm23} for $w''(0)<\zeta_*$.

\section{Supersonic flows past a rarefactive ramp}

\subsection{2D steady full Euler equation for a van der Waals gas}
In this section we study the stability of fan-shock-fan composite waves for the 2D steady full Euler equations.
The 2D steady Euler equations takes the form
\begin{equation}\label{PSEU}
\left\{
  \begin{array}{ll}
    (\rho u)_x+(\rho v)_y=0, \\[4pt]
   (\rho u^2+p)_x+(\rho uv)_y=0, \\[4pt]
  (\rho uv)_x+(\rho v^2+p)_y=0, \\[4pt]
  (\rho uE)_x+(\rho vE)_y=0,
  \end{array}
\right.
\end{equation}
where $(u, v)$ is the velocity, $p$ is the pressure, $\rho$ is the density, $E=\frac{u^2+v^2}{2}+\epsilon(\tau, s)$ is the specific total energy, $\epsilon$ is the specific internal energy, $\tau$ is the specific volume, and $s$ is the specific entropy.
In the following discussion, we shall use $h=\epsilon+p\tau$ to denote the specific enthalpy and $c=\sqrt{-\tau^2p_{\tau}(\tau, s)}$ the speed of sound.

We take a polytropic van der Waals gas with the equation of state
\begin{equation}\label{van}
\epsilon=\epsilon(\tau, s)=\frac{1}{(\tau-1)^{\gamma-1}}\exp\left(\frac{(\gamma-1)s}{R}\right)-\frac{1}{\tau},
\end{equation}
where $R$ is the gas constant and $\gamma$ is a adiabatic constant between $1$ and $2$.

From (\ref{van}) and the relation $\mathrm{d}\epsilon=T{\rm d}s-p{\rm d}\tau$,  we have
\begin{equation}
p=-\epsilon_{\tau}(\tau, s)=\frac{\gamma -1}{(\tau-1)^{\gamma }}\exp\left(\frac{(\gamma -1)s}{R}\right)-\frac{1}{\tau^2}\label{p}
\end{equation}
and
\begin{equation}
T=\epsilon_{s}(\tau, s)=\frac{\gamma -1}{R(\tau-1)^{\gamma -1}}\exp\left(\frac{(\gamma -1)s}{R}\right),\label{T}
\end{equation}
 where $T$ is the temperature.
The internal energy can be written as a function of $p$ and $\tau$, i.e.,
\begin{equation}\label{61001}
\epsilon=\frac{1}{\gamma-1}\Big(p+\frac{1}{\tau^2}\Big)(\tau-1)-\frac{1}{\tau}.
\end{equation}

For convenience, we
set $S=(\gamma -1)\exp\big(\frac{(\gamma -1)s}{R}\big)$. Then
\begin{equation}\label{sig}
p=p(\tau, S)=\frac{ S}{(\tau-1)^{\gamma }}-\frac{1}{\tau^2}\quad\mbox{and}\quad S= S(p,\tau)=\Big(p+\frac{1}{\tau^2}\Big)(\tau-1)^{\gamma }.
\end{equation}
A direct computation yields
\begin{equation}
p_{\tau}(\tau, S)=-\frac{\gamma  S}{(\tau-1)^{\gamma+1}}+\frac{2}{\tau^3}\quad\mbox{and}\quad p_{\tau\tau}(\tau, S)=\frac{\gamma (\gamma+1) S}{(\tau-1)^{\gamma+2}}-\frac{6}{\tau^4}.\label{inflection}
\end{equation}
Thus the curve in the $(\tau, p)$-plane where the inflection points of the isentropes are located is given by
the function
\begin{equation}\label{9901}
p=i(\tau):=\frac{1}{\tau^2}\left[\frac{6}{\gamma (\gamma+1)}\Big(1-\frac{1}{\tau}\Big)^2-1\right];
\end{equation}
see the red curves in Figure \ref{Figure4}.

By (\ref{sig}) and (\ref{9901}) we have
\begin{equation}\label{122801}
S\big(i(\tau), \tau\big)=\frac{6(\tau-1)^{\gamma+2}}{\gamma (\gamma+1)\tau^4}.
\end{equation}
So, inflection points can exist only provided that $ S$ is less then a limiting value $S_*$.
A direct calculation shows
\begin{equation}\label{1811701}
 S_*= S\big(i(\tau_*), \tau_*\big)
=\frac{3(2-\gamma)^{2-\gamma}(\gamma+2)^{\gamma+2}}{256\gamma (\gamma+1)},\quad\mbox{where}\quad\tau_{*}=\frac{4}{2-\gamma}.
\end{equation}

\begin{rem}
There exists a $S_{cr}<S_*$ such that
if $S_{cr}<S< S_*$, then $p_{\tau}(\tau, S)<0$ for $\tau>1$ and the isentrope $p=\frac{S}{(\tau-1)^{\gamma}}-\frac{1}{\tau^2}$ has two and only two  inflection points denoted by $\tau_1^{i}( S)$ and $\tau_2^{i}( S)$; see appendix C of Fossati and  Quartapelle \cite{Fossati} and the references cited therein.
\end{rem}

\begin{figure}[htbp]
\begin{center}
\includegraphics[scale=0.38]{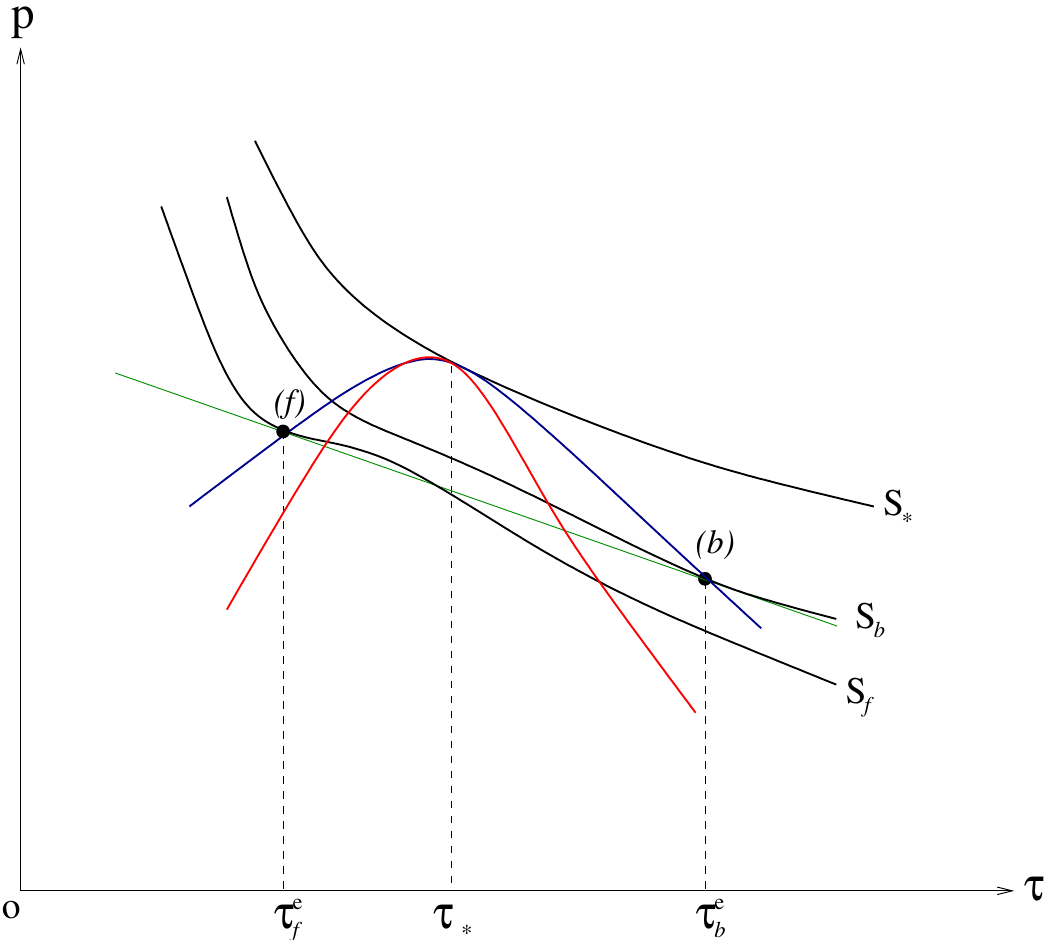}\qquad \qquad\includegraphics[scale=0.38]{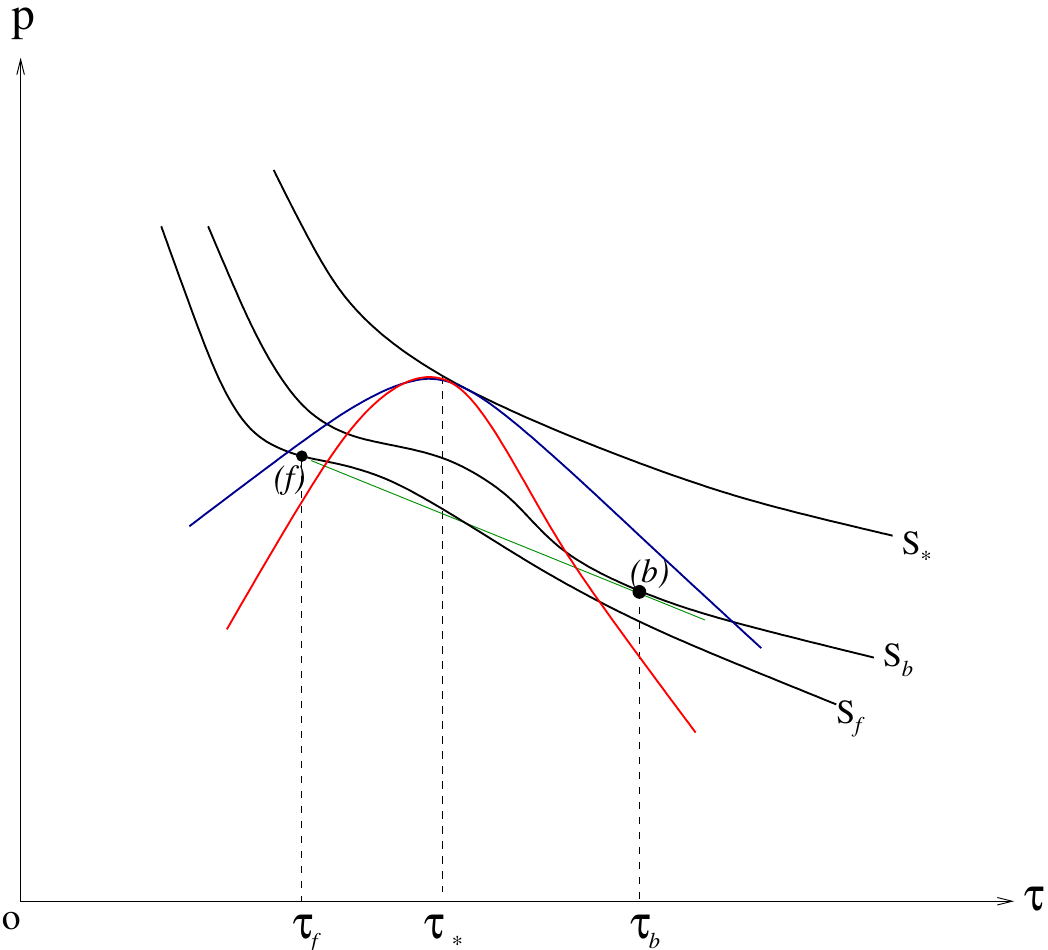}
\caption{\footnotesize Double-sonic shock and post-sonic shock.}
\label{Figure4}
\end{center}
\end{figure}

\subsection{Double-sonic shock and post-sonic shock}
We now discuss oblique shocks for (\ref{PSEU}) for the polytropic van der Waals gas.
Let $\phi$ be the inclination angle of the shock front. We still denote by $N$ and $L$
the components of $(u, v)$ normal and tangential to the shock direction, respectively; see (\ref{62604}).
We use subscripts `$f$' and `$b$' to denote the states on the front side and backside of the
shock front. Then
the Rankine-Hugoniot conditions can be written in the following form (see \cite{CaF}):
\begin{equation}\label{RHH}
\left\{
  \begin{array}{ll}
  \rho_f N_f=\rho_b N_b=m, \\[4pt]
\rho_f N_f^2+p_f=\rho_b N_b^2+p_b,\\[4pt]
    L_{f}=L_b,\\[4pt]
   N_f^2+2h_f= N_b^2+2h_b.
  \end{array}
\right.
\end{equation}
In addition to the jump condition (\ref{RHH}), admissible shock waves must satisfy the entropy inequality $s_b>s_f$.

From the first and second relations of (\ref{RHH}) we have
\begin{equation}\label{6810a}
m^2=-\frac{p_f-p_b}{\tau_f-\tau_b}.
\end{equation}
From  (\ref{RHH}) we have
\begin{equation}
\epsilon_f-\epsilon_b=-\frac{1}{2}(\tau_f-\tau_b)(p_f+p_b).\label{Hugoniot}
\end{equation}


Set
$$
\eta_f=\frac{\tau_b}{\tau_f}\quad\mbox{and}\quad \eta_b=\frac{\tau_f}{\tau_b}.
$$
Then by (\ref{61001}), (\ref{6810a}), and (\ref{Hugoniot}) we have
\begin{equation}\label{zeta}
(\eta_k-1)\left(m^2\tau_k\Big(\frac{\gamma+1}{2}\Big)\eta_k^3-\Big(m^2+\frac{\gamma-1}{2} \tau_km^2+\gamma p_k\Big)\eta_k^2-
\Big(\frac{1}{\tau_k^3}+\frac{\gamma-2}{\tau_k^2}\Big)\eta_k-\frac{1}{\tau_k^3}\right)=0,
\end{equation}
where $k=f,~ b$.

If we now set $|m|=\sqrt{-p_{\tau}(\tau_k, S_k)}$ ($k=f~\mbox{or}~b$), then by (\ref{sig}) and (\ref{inflection}) we have
\begin{equation}
m^2=\mathcal{H}(p_k, \tau_k):=\gamma \Big(\frac{p_k \tau_k^2+1}{\tau_k^2(\tau_k-1)}\Big)-\frac{2}{\tau_k^3}.\label{a}
\end{equation}
Combining this with (\ref{zeta}), we obtain
\begin{equation}\label{92902}
(\eta_k-1)\Bigg[\Big(\frac{\gamma+1}{2}\Big)\mathcal{H}(p_k, \tau_k)\tau_k \eta_k^2+
\Big(\frac{2}{\tau_k^3}+\frac{\gamma-2}{\tau_k^2}\Big)\eta_k+\frac{1}{\tau_k^3}\Bigg]=0.
\end{equation}

If $\rho_fc_f=\rho_bc_b=|m|$ of a shock, then the shock is called a double-sonic shock. See Figure \ref{Figure4} (left). The chord connecting the points $(\tau_f, p_f)$ and $(\tau_b, p_b)$ is tangent to the isentropes $S(p, \tau)=S_f$ and $S(p, \tau)=S_b$ at $(\tau_f, p_f)$ and $(\tau_b, p_b)$, respectively.

Double-sonic shocks occur when
\begin{equation}\label{100501}
\Big(\frac{2}{\tau_k^3}+\frac{\gamma-2}{\tau_k^2}\Big)^2-\frac{2(\gamma+1)f(p_k, \tau_k)}{\tau_k^2}=0.
\end{equation}
Inserting (\ref{a}) into (\ref{100501}) we have that double-sonic shocks occur when
\begin{equation}\label{double}
p_k=d(\tau_k):=\frac{(2-\gamma)^2\tau_k^3-(2-\gamma)(4-3\gamma)\tau_k^2-8(\gamma-1)\tau_k-4}{2\gamma (\gamma+1)\tau_k^4},\quad k=f,~b;
\end{equation}
see the blue curves in Figure \ref{Figure4}.
This relation was first derived by Cramer and Sen \cite{Cra}.

From (\ref{9901}) and (\ref{double}) we have
\begin{equation}\label{122803}
\begin{aligned}
&d(\tau)-i(\tau)\\=~&\displaystyle\frac{(2-\gamma)^2\tau^3-(2-\gamma)(4-3\gamma)\tau^2-8(\gamma-1)\tau-4}{2\gamma (\gamma+1)\tau^4}-
\frac{1}{\tau^2}\Big[\frac{6}{\gamma (\gamma+1)}\Big(1-\frac{1}{\tau}\Big)^2-1\Big]\\
=~&\displaystyle\frac{(\tau-1)[(\gamma-2)\tau+4]^2}{2\gamma (\gamma+1)\tau^4}~\geq~0.
\end{aligned}
\end{equation}

Let
\begin{equation}\label{122802}
\begin{array}{rcl}
\Hat{S}(\tau)=S\big(d(\tau), \tau\big)=\displaystyle\frac{\big[(2-\gamma)^2\tau^3-(\gamma^2-12\gamma+9)\tau^2-8(\gamma-1)\tau-4\big](\tau-1)^{\gamma }}
{2\gamma (\gamma+1)\tau^4}.
\end{array}
\end{equation}
From (\ref{122801}), (\ref{122803}), and (\ref{122802}) we know that
when  $S_*-S$ is positive and sufficiently small
there exist $\tau_f^e(S)$ and $\tau_b^e(S)$, where $\tau_f^{e}(S)<\tau_1^{i}(S)<\tau_{*}<\tau_2^{i}(S)<\tau_b^{e}(S)$, such that
$$p\big(\tau_k^e( S),  S\big)~=~d\big(\tau_k^e( S)\big), \quad k=f,~ b.$$

From (\ref{92902}) and (\ref{100501}) we also have that for double sonic shocks,
\begin{equation}\label{zetai}
\eta_k=\hat{\eta}(\tau_k):=-\frac{1}{(\gamma+1)\tau_k \mathcal{H}(p_k, \tau_k)}\Big(\frac{\gamma-2}{\tau_k^2}+\frac{2}{\tau_k^3}\Big)
=
\frac{2}{(2-\gamma)\tau_k-2}, \quad k=f,~ b.
\end{equation}
Thus, for double-sonic shocks,
 $\eta_f>1$ and $\eta_f\tau_f>\tau_{*}$, where $\tau_{*}$ is given in (\ref{1811701}); see \cite{Cra}.

If $\rho_bc_b=|m|>\rho_fc_f$ of a shock, then the shock is called a post-sonic shock; see Figure \ref{Figure4} (right).
The chord connecting the points $(\tau_f, p_f)$ and $(\tau_b, p_b)$ is tangent to the isentrope $S(p, \tau)=S_b$ at $(\tau_b, p_b)$.

The following proposition will be used for the stability of the fan-shock-fan composite wave.
\begin{prop}\label{81702}
(Post-sonic) For a given $S_f\in (S_{cr}, S_*)$. Assume that $S_*-S_f$ is sufficiently small. Then there exist functions $\tau_{po}(\tau_f)$ and $S_{po}(\tau_f)$ defined on  $\big[\tau_f^e( S_f), \tau_1^i(S_f)\big]$ such that
\begin{equation}
-p_{\tau}(\tau_f, S_f)<-\frac{p_f-p_b}{\tau_f-\tau_b}=-p_{\tau}(\tau_b, S_b)\quad \mbox{for}\quad \tau_f\in\big(\tau_f^e( S_f), \tau_1^i(S_f)\big),
\end{equation}
 where $p_f=p(\tau_f, S_f)$, $\tau_b=\tau_{po}(\tau_f)$, $S_b=S_{po}(\tau_f)$, and $p_b=p(\tau_{po}(\tau_f), S_{po}(\tau_f))$; see Figure \ref{Figure4} (right).
 Moreover, the solution satisfies
\begin{equation}\label{73002}
\begin{aligned}
&\tau_{po}(\tau_f^e(S_f))=\hat{\eta}(\tau_f^e(S_f))\tau_f^e(S_f),\quad S_{po}(\tau_f^e(S_f))=\hat{S}(\hat{\eta}(\tau_f^e(S_f))\tau_f^e(S_f)), \\[2pt]&\qquad\qquad\quad\frac{{\rm d}\tau_{po}}{{\rm d}\tau_f}\bigg|_{\tau_f=\tau_f^e( S_f)}=0,\quad\mbox{and}\quad \frac{{\rm d}S_{po}}{{\rm d}\tau_f}\bigg|_{\tau_f=\tau_f^e( S_f)}=0.
\end{aligned}
\end{equation}
\end{prop}
\begin{proof}
See Lai (\cite{Lai2}, Lemma 2.3).
\end{proof}


\subsection{Characteristics of 2D steady supersonic flow}
For smooth flow, system (\ref{PSEU}) can be written as
\begin{equation}
\left\{
  \begin{array}{ll}
    (\rho u)_x+(\rho v)_y=0, \\[4pt]
  uu_x+vu_y+\tau p_{x}=0,  \\[4pt]
    uv_x+vv_y+\tau p_{y}=0,\\[4pt]
    uS_{x}+vS_{y}=0.
  \end{array}
\right.\label{PsEuler}
\end{equation}

The eigenvalues of  (\ref{matrix1}) are determined
by
\begin{equation}
\Big(\lambda-\frac{v}{u}\Big)^2\big[(v-\lambda u)^{2}-c^{2}(1+\lambda^{2})\big]=0,
\end{equation}
which yields
\begin{equation}
\lambda=\lambda_{\pm}(u,v,c)=\frac{uv\pm
c\sqrt{u^{2}+v^{2}-c^{2}}}{u^{2}-c^{2}}\quad\mbox{and}\quad\lambda=\lambda_0=\frac{v}{u}.
\end{equation}
So, if and only if $u^{2}+v^{2}>c^{2}$,  system (\ref{matrix1}) is hyperbolic and has two families of wave characteristics $C_{\pm}$
defined as the integral curves of
$\frac{{\rm d}y}{{\rm d}x}=\lambda_{\pm}$. We also denote by $C_0$ the stream lines defined as the integral curves of $\frac{{\rm d}y}{{\rm d}x}=\lambda_{0}$.
In the flowing discussion we  will only consider supersonic flows.

See Figure \ref{Characteristicd}.
The direction of the wave
characteristic is defined as the tangent direction that forms an
acute angle $A$ with the flow velocity
$(u, v)$. By simple computation, we see that the $C_+$
characteristic direction forms with the flow velocity
$(u, v)$ the angle $A$ from $C_{+}$ to  $(u, v)$ in
the clockwise direction, and the $C_-$ characteristic direction forms with the
flow direction the angle $A$ from $C_{-}$ to $(u, v)$
in the counterclockwise direction.
The $C_{+}$ ($C_{-}$) characteristic angle is defined as the counterclockwise angle from the positive $x$-axis to the $C_{+}$ ($C_{-}$) characteristic direction.
We denote by $\alpha$ and
 $\beta$ the $C_{+}$ and $C_{-}$ characteristic angles,  respectively,
 where $0<\alpha-\beta<\pi$. Let $\sigma$ be the counterclockwise angle from the positive $x$-axis to the direction of the flow velocity.
Then, for the 2D steady full Euler equations (\ref{PSEU}), we still have the relations (\ref{210cqo})--(\ref{4402}).

Let $$B=\frac{q^{2}}{2}+h.$$
Then by the last three equations of (\ref{PsEuler}) we have
\begin{equation}\label{TB}
uB_{x}+vB_{y}=0.
\end{equation}
In what follows we assume
\begin{equation}\label{BERC}
B=\mbox{Constant}
\end{equation}
in the flow field. This assumption is satisfied if $B$ is constant in the incoming supersonic flow.

From (\ref{inflection}) and (\ref{BERC}), we also have
\begin{equation}\label{72301}
\rho_{x}=\frac{1}{c^2\tau}\left(-uu_{x}-vv_{x}-h_{_S} S_{x}\right),\quad
\rho_{y}=\frac{1}{c^2\tau}\left(-uu_{y}-vv_{y}-h_{_S}S_{y}\right).
\end{equation}

The first equation of (\ref{PsEuler}) can by (\ref{72301}) be written in the form
\begin{equation}
  (c^{2}-u^{2})u_{x}-uv(u_{y}+v_{x})+(c^{2}-v^{2})v_{y}=0.\label{mass}
\end{equation}
Multiplying
$$
\begin{array}{rcl}
\left(
 \begin{array}{cc}
c^{2}-u^{2} & -uv \\
  0 & -1\\
  \end{array}
  \right)\left(
           \begin{array}{c}
             u \\
             v \\
           \end{array}
         \right)_{x}+\left(
                         \begin{array}{ccccc}
                          -uv &  c^{2}-v^{2}\\
                           1 & 0 \\
                         \end{array}
                       \right)\left(
                                \begin{array}{c}
                                  u \\
                                  v \\
                                \end{array}
                              \right)_{y}=\left(
                                               \begin{array}{c}
                                                 0 \\
                                                 \omega \\
                                               \end{array}
                                             \right)
                                             \label{matrix1}
                                             \end{array}
$$
on the left
by
$(1,\mp c\sqrt{u^{2}+v^{2}-c^{2}})$, we get
\begin{equation}
\left\{
  \begin{array}{ll}
  \displaystyle \bar{\partial}_{+}u+\lambda_{-}\bar{\partial}_{+}v
  =\frac{\omega\sin A\cos A}{\cos\beta},  \\[10pt]
    \displaystyle  \bar{\partial}_{-}u+\lambda_{+}\bar{\partial}_{-}v=-\frac{\omega\sin A\cos A}{\cos\alpha},
  \end{array}
\right.\label{form}
\end{equation}
where $\omega=u_{y}-v_{x}$ is the vorticity.

By the fourth equation of (\ref{PsEuler}) we get
\begin{equation}\label{101102}
\bar{\partial}_{+}S=-\bar{\partial}_{-}S.
\end{equation}
From the second and third equations of (\ref{PsEuler}) we have
\begin{equation}
B_x=-v\omega+\frac{ S_{x}}{(\gamma-1)(\tau-1)^{\gamma-1}}\quad \mbox{and} \quad B_y=u\omega+\frac{ S_{y}}{(\gamma-1)(\tau-1)^{\gamma-1}}.\label{72601}
\end{equation}
Then by (\ref{BERC}) we have
\begin{equation}\label{61402}
\omega=\mp\frac{\bar{\partial}_{\pm}S}{c(\gamma-1)(\tau-1)^{\gamma-1}}.
\end{equation}
By the last three equations of (\ref{PsEuler}) we also have
\begin{equation}\label{72803}
u\omega_{x}+v\omega_{y}+(u_{x}+v_{y})\omega=\frac{\tau_{x} S_{y}-\tau_{y} S_{x}}{(\tau-1)^{\gamma }}.
\end{equation}

\subsection{Characteristic decompositions for the 2D steady Euler equations}
From the Bernoulli's law (\ref{BERC}) we have
\begin{equation}\label{726021}
u\bar{\partial}_{+}u+v\bar{\partial}_{+}v+\kappa c\bar{\partial}_{+}c=-\left(h_{_S}+\frac{\kappa \tau^2p_{_{\tau S}}}{2}\right)\bar{\partial}_{+}S
\end{equation}
and
\begin{equation}\label{41102}
u\bar{\partial}_{-}u+v\bar{\partial}_{-}v+\kappa c\bar{\partial}_{-}c=-\left(h_{_S}+\frac{\kappa \tau^2p_{_{\tau S}}}{2}\right)\bar{\partial}_{-}S,
\end{equation}
where
$$
\kappa=-\frac{2p_{\tau}}{2p_{\tau}+\tau p_{\tau\tau}}.
$$

From (\ref{U}) we have
\begin{equation}\label{72701}
\bar{\partial}_{\pm}u=\frac{\cos\sigma}{\sin A}\bar{\partial}_{\pm}c+\frac{c\cos\alpha\bar{\partial}_{\pm}\beta
-c\cos\beta\bar{\partial}_{\pm}\alpha}{2\sin^{2}A},
\end{equation}
\begin{equation}\label{72702}
\bar{\partial}_{\pm}v=\frac{\sin\sigma}{\sin A}\bar{\partial}_{\pm}c
+\frac{c\sin\alpha\bar{\partial}_{\pm}\beta
-c\sin\beta\bar{\partial}_{\pm}\alpha}{2\sin^{2}A}.
\end{equation}
Inserting (\ref{72701}) and (\ref{72702}) into (\ref{form}), we obtain
\begin{equation}\label{3}
\bar{\partial}_{+}c=\frac{c}{\sin2A}
(\bar{\partial}_{+}\alpha-\cos2A\bar{\partial}_{+}\beta)+\omega\sin^{2}A,
\end{equation}
\begin{equation}
\bar{\partial}_{-}c=\frac{c}{\sin2A}
(\cos2A\bar{\partial}_{-}\alpha-\bar{\partial}_{-}\beta)-\omega\sin^{2}A.\label{4}
\end{equation}

Inserting (\ref{72701})--(\ref{72702}) into (\ref{726021}), we get
\begin{equation}
\left(\frac{1}{\sin^{2}A}+\kappa\right)\bar{\partial}_{+}c=
\frac{c\cos A}{2\sin^{3}A}(\bar{\partial}_{+}\alpha-\bar{\partial}_{+}\beta)
+\omega+j\bar{\partial}_{+} S,\label{bB1}
\end{equation}
where
$$
j=-\frac{cp_{_{\tau S}}}{2 p_{\tau}+\tau p_{\tau\tau}}-\frac{\tau}{c(\tau-1)^{\gamma }}.
$$
Similarly, we have
\begin{equation}
\left(\frac{1}{\sin^{2}A}+\kappa \right)\bar{\partial}_{-}c=
\frac{c\cos A}{2\sin^{3}A}(\bar{\partial}_{-}\alpha-\bar{\partial}_{-}\beta)
-\omega+j\bar{\partial}_{-} S.\label{bB2}
\end{equation}

Inserting (\ref{bB1}) into (\ref{3}), we obtain
\begin{equation}\label{6}
c\bar{\partial}_{+}\alpha=\Omega\cos^{2}Ac\bar{\partial}_{+}\beta-\Big(\frac{\kappa\sin2A\sin^{2} A}{1+\kappa}\Big)\omega+\frac{j\sin2A}{1+\kappa}\bar{\partial}_{+} S,
\end{equation}
where
\begin{equation}\label{40}
\Omega=-\frac{4p_{\tau}+\tau p_{\tau\tau}}{\tau p_{\tau\tau}}-\tan^2 A.
\end{equation}
Similarly, inserting (\ref{bB2}) into (\ref{4}), we obtain
\begin{equation}\label{5}
c\bar{\partial}_{-}\beta=\Omega\cos^{2}Ac\bar{\partial}_{-}\alpha-\Big(\frac{\kappa\sin2A\sin^{2} A}{1+\kappa}\Big)\omega-\frac{j\sin2A}{1+\kappa}\bar{\partial}_{-} S.
\end{equation}

Combining (\ref{3}) with (\ref{6}), we have
\begin{equation}
c\bar{\partial}_{+}\beta=-(1+\kappa)\tan A\bar{\partial}_{+}c+\omega\sin^2A\tan A
+j\tan A \bar{\partial}_{+} S\label{9}
\end{equation}
and
\begin{equation}
c\bar{\partial}_{+}\alpha=-\left(\frac{1+\kappa}{2}\right)\Omega\sin2 A\bar{\partial}_{+}c-\omega\sin^2A\tan A+
j\tan A\cos 2A\bar{\partial}_{+} S.\label{10}
\end{equation}
Combining (\ref{4}) with (\ref{5}), we have
\begin{equation}
c\bar{\partial}_{-}\alpha=(1+\kappa)\tan A\bar{\partial}_{-}c+\omega\sin^2A\tan A
-j\tan A \bar{\partial}_{-} S
\end{equation}
and
\begin{equation}
c\bar{\partial}_{-}\beta=\left(\frac{1+\kappa}{2}\right) \Omega\sin2 A\bar{\partial}_{-}c-\omega\sin^2A\tan A-j\tan A\cos2A\bar{\partial}_{-} S.\label{8}
\end{equation}

Equations (\ref{9})--(\ref{8}) can by $c=\sqrt{-\tau^2p_{\tau}(\tau, S)}$ be written as
\begin{equation}\label{1}
c\bar{\partial}_{+}\alpha=-\frac{p_{\tau\tau}}{4c\rho^4}\Omega\sin2 A\bar{\partial}_{+}\rho-\omega\sin^2A\tan A+j_1\bar{\partial}_{+} S,
\end{equation}
\begin{equation}\label{8301}
c\bar{\partial}_{-}\beta=\frac{p_{\tau\tau}}{4c\rho^4}\Omega\sin2 A\bar{\partial}_{-}\rho-\omega\sin^2A\tan A-j_1\bar{\partial}_{-} S,
\end{equation}
\begin{equation}\label{8302}
c\bar{\partial}_{-}\alpha=\frac{p_{\tau\tau}}{2c\rho^4}\tan A\bar{\partial}_{-}\rho+\omega\sin^2A\tan A-j_2\bar{\partial}_{-} S,
\end{equation}
\begin{equation}\label{2}
c\bar{\partial}_{+}\beta=-\frac{p_{\tau\tau}}{2c\rho^4}\tan A\bar{\partial}_{+}\rho+\omega\sin^2A\tan A+j_2\bar{\partial}_{+} S,
\end{equation}
where
$$
j_1=\frac{\tau^2p_{_{\tau S}}}{2c}\left(\frac{1+\kappa}{2}\right)\Omega\sin2 A+j\tan A\cos 2A\quad \mbox{and}\quad
j_2=(1+\kappa)\tan A\frac{\tau^2 p_{_{\tau S}}}{2c}+j\tan A.
$$

Inserting (\ref{1})--(\ref{2}) into (\ref{72701}) and (\ref{72702}),
we get
\begin{equation}\label{11}
\bar{\partial}_{+}u=c\tau\sin\beta\bar{\partial}_{+}\rho+\omega\cos \sigma\sin A+j_3\bar{\partial}_{+} S,
\end{equation}
\begin{equation}\label{73001}
\bar{\partial}_{-}u=-c\tau\sin\alpha\bar{\partial}_{-}\rho-\omega\cos \sigma\sin A+j_4\bar{\partial}_{-} S,
\end{equation}
\begin{equation}\label{72804}
\bar{\partial}_{+}v=-c\tau\cos\beta\bar{\partial}_{+}\rho+\omega\sin \sigma\sin A+j_5\bar{\partial}_{+} S,
\end{equation}
\begin{equation}\label{81005}
\bar{\partial}_{-}v=c\tau\cos\alpha\bar{\partial}_{-}\rho-\omega\sin \sigma\sin A+j_6\bar{\partial}_{-} S,
\end{equation}
where
$$
\begin{aligned}
&j_3=-\kappa\sin\beta\frac{\tau^2p_{_{\tau S}}}{2c}-j\sin\beta,
\quad
j_4=\kappa\sin\alpha\frac{\tau^2p_{_{\tau S}}}{2c}+j\sin\alpha,\\&
j_5=\kappa\cos\beta\frac{\tau^2p_{_{\tau S}}}{2c}+j\cos\beta,
\quad \mbox{and}\quad
j_6=-\kappa\cos\alpha\frac{\tau^2p_{_{\tau S}}}{2c}-j\cos\alpha.
\end{aligned}
$$

\begin{prop}\label{10603}
(Commutator relations) We have the following commutator relations:
\begin{equation}
\begin{array}{rcl}
\bar{\partial}_{-} \bar{\partial}_{+}- \bar{\partial}_{+} \bar{\partial}_{-}=
\displaystyle\frac{1}{\sin2 A}\Big[\big(\cos2 A \bar{\partial}_{+}\beta- \bar{\partial}_{-}\alpha\big) \bar{\partial}_{-}-
\big( \bar{\partial}_{+}\beta-\cos2 A \bar{\partial}_{-}\alpha\big) \bar{\partial}_{+}\Big],
\end{array}
\label{comm}
\end{equation}
\begin{equation}
\begin{array}{rcl}
\bar{\partial}_{0} \bar{\partial}_{+}- \bar{\partial}_{+} \bar{\partial}_{0}=
\displaystyle\frac{1}{\sin A}\Big[\big(\cos A \bar{\partial}_{+}\sigma- \bar{\partial}_{0}\alpha\big) \bar{\partial}_{0}-
\big( \bar{\partial}_{+}\sigma-\cos A \bar{\partial}_{0}\alpha\big) \bar{\partial}_{+}\Big].
\end{array}
\label{comm1}
\end{equation}
\end{prop}
\begin{proof}
The commutator relation (\ref{comm}) was first derived by Li, Zhang $\&$ Zheng \cite{Li-Zhang-Zheng}.
The commutation relation (\ref{comm1}) was first derived by Lai \cite{Lai2}.
\end{proof}

From (\ref{32403}), (\ref{8302}), (\ref{2}),  and (\ref{comm1})  we have
\begin{equation}\label{81301}
\begin{aligned}
\bar{\partial}_{0}\bar{\partial}_{+} S&=\frac{1}{\sin A}\big(\cos A\bar{\partial}_{0}\alpha-
\bar{\partial}_{+}\sigma\big)\bar{\partial}_{+} S
=\frac{1}{\sin A}\Big[\frac{1}{2}(\bar{\partial}_{+}\alpha+\bar{\partial}_{-}\alpha)
-\bar{\partial}_{+}\sigma\Big]\bar{\partial}_{+} S\\&=\frac{1}{2\sin A}\Big[\bar{\partial}_{-}\alpha
-\bar{\partial}_{+}\beta\Big]\bar{\partial}_{+} S=\frac{p_{\tau\tau}}{2c^2\rho^4\cos A}\Big[\bar{\partial}_{-}\rho
+\bar{\partial}_{+}\rho\Big]\bar{\partial}_{+} S=-\frac{\tau^2p_{\tau\tau}}{c^2}\bar{\partial}_{0}\tau
\bar{\partial}_{+}S.
\end{aligned}
\end{equation}
Let
$
\mathcal{E}(\tau,S)=e^{\int^{\tau}-\frac{p_{\tau\tau}(\tau, S)}{p_{\tau}(\tau, S)}{\rm d}\tau}.
$
Then the equation (\ref{81301}) can be written as
\begin{equation}\label{7403}
\bar{\partial}_{0}\big(\mathcal{E}(\tau,S)\bar{\partial}_{+} S\big)
~=~\mathcal{E}_{\tau}(\tau,S)\bar{\partial}_{0}\tau
\bar{\partial}_{+} S-\mathcal{E}(\tau,S)\frac{\tau^2p_{\tau\tau}}{c^2}\bar{\partial}_{0}\tau
\bar{\partial}_{+} S~=~0.
\end{equation}
This implies that $\mathcal{E}(\tau,S)\bar{\partial}_{+} S$ is invariant along each stream line.

From (\ref{tau})  and (\ref{72804})--(\ref{73001}) we have
\begin{equation}\label{41105}
\begin{aligned}
u_x+v_y&=\frac{1}{\sin 2A}\Big(\sin\alpha\bar{\partial}_{-}u-\sin\beta\bar{\partial}_{+}u+
\cos\beta\bar{\partial}_{+}v-\cos\alpha\bar{\partial}_{-}v\Big)=-
\frac{c\tau}{\sin A}\bar{\partial}_{0}\rho
\end{aligned}
\end{equation}
and
\begin{equation}\label{41106}
\begin{aligned}
&\tau_y S_x-\tau_x S_y\\[4pt]
=&\Bigg[\Big(\frac{\cos\alpha+\cos\beta}{\sin2A}\Big)\frac{\sin\beta\bar{\partial}_{+}\tau
-\sin\alpha\bar{\partial}_{-}\tau}{\sin2A}-
\Big(\frac{\sin\alpha+\sin\beta}{\sin2A}\Big)\frac{\cos\beta\bar{\partial}_{+}\tau
-\cos\alpha\bar{\partial}_{-}\tau}{\sin2A}\Bigg\}\bar{\partial}_{+}S
\\=&-\frac{\bar{\partial}_{+}S}{\sin 2A}(\bar{\partial}_{+}\tau+\bar{\partial}_{-}\tau)
=
\frac{\bar{\partial}_{+}S}{\rho^2\sin A}\bar{\partial}_{0}\rho.
\end{aligned}
\end{equation}
Inserting (\ref{41105}) and (\ref{41106}) into (\ref{72803}), we get
\begin{equation}\label{4801}
\begin{aligned}
\bar{\partial}_{0}\omega&~=~\frac{\bar{\partial}_{0}\rho}{\rho}\omega+
\frac{\bar{\partial}_{+}S}{\rho^2(\tau-1)^{\gamma }\sin A}\bar{\partial}_{0}\rho.
\end{aligned}
\end{equation}

\begin{prop}\label{10604}
For the variable $\rho$, we have the characteristic decompositions
\begin{equation}\label{cd}
\left\{
  \begin{array}{ll}
\begin{aligned}
   \bar{\partial}_{+}\bar{\partial}_{-}\rho~=~&\displaystyle
    \frac{\tau^4p_{\tau\tau}}{4c^2\cos^{2}A }\Big[(\bar{\partial}_{-}\rho)^2+(\varphi-1)\bar{\partial}_{-}\rho\bar{\partial}_{+}\rho\Big]
+g_1\bar{\partial}_{-}\rho \bar{\partial}_{+}S+g_2\bar{\partial}_{+}\rho\bar{\partial}_{+} S\\[4pt]&+g_3(\bar{\partial}_{+} S)^2+\bar{\partial}_{+}(g_4\bar{\partial}_{+} S),
\end{aligned}\\[18pt]
\begin{aligned}
  \bar{\partial}_{-}\bar{\partial}_{+}\rho~=~&\displaystyle
    \frac{\tau^4p_{\tau\tau}}{4c^2\cos^{2}A }\Big[(\bar{\partial}_{+}\rho)^2+(\varphi-1)\bar{\partial}_{-}\rho\bar{\partial}_{+}\rho\Big]
+g_5\bar{\partial}_{-}\rho \bar{\partial}_{+}S+g_6\bar{\partial}_{+}\rho\bar{\partial}_{+} S\\[4pt]&+g_7(\bar{\partial}_{+} S)^2+\bar{\partial}_{+}(g_8\bar{\partial}_{+} S),
\end{aligned}
  \end{array}
\right.
\end{equation}
where
$$
\varphi=2\sin^2A-\frac{8p_{\tau}\cos^4A}{\tau p_{\tau\tau}},
$$
and $g_{i}$ $(i=1, \cdot\cdot\cdot, 8)$ are elementary functions of the variables $\alpha$, $\beta$, $\tau$ and $S$.
\end{prop}

\begin{proof}
From (\ref{11}), (\ref{73001}), and (\ref{comm}) we have
$$
\begin{array}{rcl}
&&\bar{\partial}_{+}\big[c\tau\sin\alpha\bar{\partial}_{-}\rho+\omega\cos \sigma\sin A\big]+
\bar{\partial}_{-}\big[c\tau\sin\beta\bar{\partial}_{+}\rho+\omega\cos \sigma\sin A\big]
\\[6pt]&=&\displaystyle\frac{1}{\sin2 A}( \bar{\partial}_{-}\alpha- \cos2 A\bar{\partial}_{+}\beta)(c\tau\sin\alpha\bar{\partial}_{-}\rho+\omega\cos \sigma\sin A-j_4\bar{\partial}_{-} S)+\bar{\partial}_{+}(j_4\bar{\partial}_{-} S)\\[8pt]&&\displaystyle-\frac{1}{\sin2 A}(\bar{\partial}_{+}\beta-\cos2 A \bar{\partial}_{-}\alpha)(c\tau\sin\beta\bar{\partial}_{+}\rho+\omega\cos \sigma\sin A+j_3\bar{\partial}_{+} S)
-\bar{\partial}_{-}(j_3\bar{\partial}_{+} S).
\end{array}
$$
Hence
\begin{equation}
\begin{array}{rcl}
&&\displaystyle(\sin\alpha+\sin\beta)\frac{\partial(c\tau)}{\partial\rho}\frac{1}{c\tau }\bar{\partial}_{+}\rho\bar{\partial}_{-}\rho+\sin\alpha\bar{\partial}_{+}\bar{\partial}_{-}\rho+
\sin\beta\bar{\partial}_{-}\bar{\partial}_{+}\rho\\[8pt]&=&\displaystyle\frac{1}{\sin2 A}
\Big[( \sin\alpha\bar{\partial}_{-}\alpha- \sin\alpha\cos2 A\bar{\partial}_{+}\beta-\cos\alpha\sin2 A\bar{\partial}_{+}\alpha) \bar{\partial}_{-}\rho\\[4pt]&&\displaystyle\qquad\quad-(\sin\beta \bar{\partial}_{+}\beta-\sin\beta \cos2 A \bar{\partial}_{-}\alpha+\cos\beta\sin2 A\bar{\partial}_{-}\beta) \bar{\partial}_{+}\rho\Big]\\[6pt]&&\displaystyle+
\frac{1}{c\tau\sin2 A}\Big[\big(\cos2 A \bar{\partial}_{+}\beta- \bar{\partial}_{-}\alpha\big) j_4\bar{\partial}_{-}S-
\big( \bar{\partial}_{+}\beta-\cos2 A \bar{\partial}_{-}\alpha\big)j_3\bar{\partial}_{+}S\Big]\\[6pt]&&-\displaystyle
\frac{1}{c\tau\sin2 A}\big(\cos2 A \bar{\partial}_{+}\beta- \bar{\partial}_{-}\alpha+
 \bar{\partial}_{+}\beta-\cos2 A \bar{\partial}_{-}\alpha\big)\omega\cos\sigma\sin A
\\[8pt]&&\displaystyle-\frac{1}{c\tau}\Big[\bar{\partial}_{+}(\omega\cos \sigma\sin A)+\bar{\partial}_{-}(\omega\cos \sigma\sin A)\Big]
+\frac{\bar{\partial}_{+}(j_4\bar{\partial}_{-} S)}{c\tau}-\frac{\bar{\partial}_{-}
(j_3\bar{\partial}_{+} S)}{c\tau}\\[8pt]&&\displaystyle-\sin\alpha\frac{\partial(c\tau)}{\partial S}\frac{1}{c\tau }\bar{\partial}_{+}S\bar{\partial}_{-}\rho-\sin\beta\frac{\partial(c\tau)}{\partial S}\frac{1}{c\tau }\bar{\partial}_{-}S\bar{\partial}_{+}\rho.
\end{array}\label{12}
\end{equation}
Applying the commutator relation (\ref{comm}) for $\rho$, we obtain
\begin{equation}\label{61401}
\bar{\partial}_{-} \bar{\partial}_{+}\rho- \bar{\partial}_{+} \bar{\partial}_{-}\rho=
\frac{1}{\sin2 A}\big[(\cos2 A \bar{\partial}_{+}\beta- \bar{\partial}_{-}\alpha) \bar{\partial}_{-}\rho-
(\bar{\partial}_{+}\beta-\cos2 A \bar{\partial}_{-}\alpha) \bar{\partial}_{+}\rho\big].
\end{equation}
Combining this with (\ref{12})
and using (\ref{101102}), (\ref{61402}), (\ref{1})--(\ref{2}), (\ref{81301}), and (\ref{4801})  we obtain
(\ref{cd}).
This completes the proof.
\end{proof}

We define
$$
\mathcal{W}_{+}:=\bar{\partial}_{+}\rho-g_4\bar{\partial}_{+}S\quad \mbox{and} \quad \mathcal{W}_{-}:=\bar{\partial}_{-}\rho-g_8\bar{\partial}_{+}S.
$$
Then by (\ref{cd}) we have
\begin{equation}\label{cd1}
\left\{
  \begin{array}{ll}
   \bar{\partial}_{+}\mathcal{W}_{-}=
   \displaystyle \frac{\tau^4p_{\tau\tau}}{4c^2\cos^{2}A }\Big[\mathcal{W}_{-}^2+(\varphi-1)\mathcal{W}_{-}\mathcal{W}_{+}\Big]
+\hat{g}_1\mathcal{W}_{-}\bar{\partial}_{+}S+\hat{g}_2\mathcal{W}_{+}\bar{\partial}_{+} S+\hat{g}_3(\bar{\partial}_{+} S)^2,\\[12pt]
  \bar{\partial}_{-}\mathcal{W}_{+}=
   \displaystyle  \frac{\tau^4p_{\tau\tau}}{4c^2\cos^{2}A }\Big[\mathcal{W}_{+}^2+(\varphi-1)\mathcal{W}_{-}\mathcal{W}_{+}\Big]
+\hat{g}_4\mathcal{W}_{-} \bar{\partial}_{+}S+\hat{g}_5\mathcal{W}_{+}\bar{\partial}_{+} S+\hat{g}_6(\bar{\partial}_{+} S)^2,
  \end{array}
\right.
\end{equation}
where
$\hat{g}_{i}$ $(i=1, \cdot\cdot\cdot, 6)$ are  elementary functions of $\alpha$, $\beta$, $\tau$ and $S$.
We also use the first equation of (\ref{cd1}) to derive
\begin{equation}\label{cd4}
\bar{\partial}_{+}\left(\frac{1}{\mathcal{W}_{-}}\right)=
-\frac{\tau^4p_{\tau\tau}}{4c^2\cos^{2}A }-\left(\frac{(\varphi-1)\tau^4p_{\tau\tau}\mathcal{W}_{+}}{4c^2\cos^{2}A }
+\hat{g}_1\bar{\partial}_{+}S\right)\frac{1}{\mathcal{W}_{-}}-\frac{\hat{g}_2\mathcal{W}_{+}\bar{\partial}_{+} S+\hat{g}_3(\bar{\partial}_{+} S)^2}{\mathcal{W}_{-}^2}.
\end{equation}

\vskip 4pt
Let $n_1$ be a positive constant.
We define
\begin{equation}\label{81401}
\mathcal{Z}_{+}:=\tau^{n_1}\mathcal{W}_{+}\quad \mbox{and}\quad \mathcal{Z}_{-}:=\tau^{n_1}\mathcal{W}_{-}.
\end{equation}
Then by (\ref{cd1}) we have
\begin{equation}\label{cd2}
\left\{
  \begin{array}{ll}
   \tau^{n_1}\bar{\partial}_{+}\mathcal{Z}_{-}=
   \displaystyle \frac{\tau^4p_{\tau\tau}}{4c^2\cos^{2}A }\Big[\mathcal{Z}_{-}^2+(\varphi-1)\mathcal{Z}_{-}\mathcal{Z}_{+}\Big]-n_1\tau\mathcal{Z}_{-}\mathcal{Z}_{+}
+\Big(\tilde{g}_1\mathcal{Z}_{-}+\tilde{g}_2\mathcal{Z}_{+}\Big)\bar{\partial}_{+} S+\tilde{g}_3(\bar{\partial}_{+} S)^2,\\[12pt]
  \tau^{n_1}\bar{\partial}_{-}\mathcal{Z}_{+}=
   \displaystyle  \frac{\tau^4p_{\tau\tau}}{4c^2\cos^{2}A }\Big[\mathcal{Z}_{+}^2+(\varphi-1)\mathcal{Z}_{-}\mathcal{Z}_{+}\Big]-n_1\tau\mathcal{Z}_{-}\mathcal{Z}_{+}
+\Big(\tilde{g}_4\mathcal{Z}_{-} +\tilde{g}_5\mathcal{Z}_{+}\Big)\bar{\partial}_{+} S+\tilde{g}_6(\bar{\partial}_{+} S)^2,
  \end{array}
\right.
\end{equation}
where
$\tilde{g}_{i}$ $(i=1, \cdot\cdot\cdot, 6)$ are  elementary functions of $\alpha$, $\beta$, $\tau$ and $S$.

\vskip 4pt
Let $n_2$ be a positive constant.
We define
\begin{equation}\label{81402}
\mathcal{R}_{+}:=\rho^{n_2}\mathcal{W}_{+}\quad \mbox{and}\quad \mathcal{R}_{-}:=\rho^{n_2}\mathcal{W}_{-}.
\end{equation}
Then by (\ref{cd1}) we have
\begin{equation}\label{cd3}
\left\{
  \begin{array}{ll}
   \tau^{n_2}\bar{\partial}_{+}\mathcal{R}_{-}=
   \displaystyle \frac{\tau^4p_{\tau\tau}}{4c^2\cos^{2}A }\Big[\mathcal{Z}_{-}^2+(\varphi-1)\mathcal{Z}_{-}\mathcal{Z}_{+}\Big]+n_2\tau\mathcal{Z}_{-}\mathcal{Z}_{+}
+\Big(\breve{g}_1\mathcal{Z}_{-}+\breve{g}_2\mathcal{Z}_{+}\Big)\bar{\partial}_{+} S+\breve{g}_3(\bar{\partial}_{+} S)^2,\\[12pt]
  \tau^{n_2}\bar{\partial}_{-}\mathcal{R}_{+}=
   \displaystyle  \frac{\tau^4p_{\tau\tau}}{4c^2\cos^{2}A }\Big[\mathcal{Z}_{+}^2+(\varphi-1)\mathcal{Z}_{-}\mathcal{Z}_{+}\Big]+n_2\tau\mathcal{Z}_{-}\mathcal{Z}_{+}
+\Big(\breve{g}_4\mathcal{Z}_{-}+\breve{g}_5\mathcal{Z}_{+}\Big)\bar{\partial}_{+} S+\breve{g}_6(\bar{\partial}_{+} S)^2,
  \end{array}
\right.
\end{equation}
where
$\breve{g}_{i}$ $(i=1, \cdot\cdot\cdot, 6)$ are elementary functions of $\alpha$, $\beta$, $\tau$ and $S$.

\vskip 4pt


\subsection{Self-similar fan-shock-fan composite wave}
We consider (\ref{PSEU}) with the initial-boundary conditions
\begin{equation}\label{RBD}
\left\{
  \begin{array}{ll}
    (u, v, \tau, S)(0, y)=(u_0, 0, \tau_0, S_0), & \hbox{$y>0$;} \\[4pt]
 v=u\tan\theta_{w}, & \hbox{$(x, y)=r(\cos\theta_w, \sin\theta_w)$, $r>0$,}
  \end{array}
\right.
\end{equation}
where $-\pi<\theta_w<0$ and $(u_0, 0, \tau_0, S_0)$ is a constant state.
We assume that the incoming flow satisfies the following assumptions:
\begin{equation}\label{61801}
 S_{cr}<S_0<S_*; \quad  1<\tau_0<\tau_{f}^e(S_0); \quad u_0>c_0:=\sqrt{-\tau_0p_{\tau}(\tau_0, S_0)}.
\end{equation}
In what follows, we are going to construct a self-similar fan-shock-fan composite wave solution to the problem (\ref{PSEU}), (\ref{RBD}).

\begin{lem}
Assume that $(\bar{q}(\theta), \bar{\tau}(\theta), \bar{\sigma}(\theta), \bar{S}(\theta))$ satisfies the equations
\begin{equation}\label{102302}
\left\{
  \begin{array}{ll}
    \bar{q}(\theta)'\cos(\bar{A}(\theta))+\bar{q}(\theta)\bar{\sigma}'(\theta)\sin(\bar{A}(\theta))=0, \\[4pt]
    \bar{q}(\theta)\bar{q}'(\theta)+\bar{\tau}(\theta)p_{\tau}(\bar{\tau}(\theta), \bar{S}(\theta))\bar{\tau}'(\theta)=0,  \\[4pt]
    \bar{\sigma}(\theta)+\bar{A}(\theta)=\theta,\\[4pt]
\bar{S}'(\theta)=0
  \end{array}
\right.
\end{equation}
and $\bar{q}(\theta)>\bar{c}(\theta)$
 for $\theta\in (\theta_1, \theta_2)$,
where
\begin{equation}\label{102303}
\bar{A}(\theta)=\arcsin\Big(\frac{\bar{c}(\theta)}{\bar{q}(\theta)}\Big)\quad \mbox{and}\quad \bar{c}(\theta)=\bar{\tau}(\theta)\sqrt{-p_{\tau}(\bar{\tau}(\theta), \bar{S}(\theta))}.
\end{equation}
And let
\begin{equation}\label{102306}
\bar{u}(\theta)=\bar{q}(\theta)\cos(\bar{\sigma}(\theta))\quad \mbox{and}\quad \bar{v}(\theta)=\bar{q}(\theta)\sin(\bar{\sigma}(\theta)).
\end{equation}
Then
\begin{equation}
 (u, v, \tau, S)=(\bar{u}, \bar{v}, \bar{\tau}, \bar{S})\big(\arctan(y/x)\big)
\end{equation}
is a centered simple wave (Prandtl-Meyer fan) solution of (\ref{PSEU}) on the triangle region
$$
\begin{aligned}
\Theta&=\Big\{(x, y)~\big|~(x, y)=(r\cos\theta, r\sin\theta),~ r>0,
~\theta\in(\theta_1, \theta_2)
\Big\}.
\end{aligned}
$$
\end{lem}
\begin{proof}
Firstly, by the second and the fourth equations (\ref{102302}) we know that the Bernoulli's law (\ref{BERC}) holds.
By the definition of $(u, v)$ we have $\omega=0$ in $\Theta$.
From the third equation of (\ref{102302}) we have
\begin{equation}\label{82104}
\big(\bar{u}(\theta), \bar{v}(\theta)\big)\cdot(\sin\theta, -\cos\theta)=\bar{c}(\theta).
\end{equation}
So, for any fixed $\theta\in (\theta_1, \theta_2)$, the ray $C_{+}^{\theta}:$ $x=r\cos\theta$, $y=r\sin\theta$ ($0\leq r>0$)
is tangent to the sonic circle $(x-\bar{u}(\theta))^2+(y-\bar{v}(\theta))^2=\bar{c}^2(\theta)$.
Therefore,  for any $\theta\in (\theta_1, \theta_2)$, $C_{+}^{\theta}$
 is
a $C_{+}$ characteristic line. Hence, we have $$\alpha=\theta, \quad \bar{\partial}_{+}u=0,\quad \bar{\partial}_{+}v=0,\quad\mbox{and} \quad\bar{\partial}_{+}\tau=0\quad \mbox{in}\quad \Theta.$$ Consequently,
\begin{equation}\label{82101}
\bar{\partial}_{+}u+\lambda_{-}\bar{\partial}_{+}v=0\quad\mbox{in}\quad \Theta.
\end{equation}

By the first equation of (\ref{102302}) and $\alpha=\theta$, we have
\begin{equation}\label{82102}
\begin{aligned}
\bar{\partial}_{-}u+\lambda_{+}\bar{\partial}_{-}v
&=\big( \bar{u}'(\theta)+\tan\theta  \bar{v}'(\theta)\big)\bar{\partial}_{-}\theta\\&=\big(\bar{q}(\theta)'\cos(\bar{A}(\theta))
+\bar{q}(\theta)\bar{\sigma}'(\theta)\sin(\bar{A}(\theta))\big)\frac{\bar{\partial}_{-}\theta}{\cos\theta}=0 \quad\mbox{in}\quad \Theta.
\end{aligned}
\end{equation}
This completes the proof of the Lemma.
\end{proof}

System (\ref{102302}) can by (\ref{102303}) be written as
\begin{equation}\label{102304}
\left\{
  \begin{array}{ll}
    \displaystyle\bar{q}'=\frac{2\bar{c}p_{\tau}(\bar{\tau}, \bar{S})\cos \bar{A}}{\bar{\tau}p_{\tau\tau}(\bar{\tau}, \bar{S})},\\[10pt]
     \displaystyle\bar{\tau}'=-\frac{2\bar{q}\bar{c}\cos \bar{A}}{\bar{\tau}^2p_{\tau\tau}(\bar{\tau}, \bar{S})},\\[10pt]
 \displaystyle\bar{\sigma}'=-\frac{2 p_{\tau}(\bar{\tau}, \bar{S})\cos^2 \bar{A}}{\bar{\tau}p_{\tau\tau}(\bar{\tau}, \bar{S})},\\[10pt]
 \displaystyle\bar{S}'=0.
  \end{array}
\right.
\end{equation}

We first consider (\ref{102304}) with data
\begin{equation}\label{102305a}
(\bar{q}, \bar{\tau}, \bar{\sigma}, \bar{S})(\alpha_0)=(u_0, \tau_0, 0, S_0),
\end{equation}
where $\alpha_0=\arcsin(\frac{c_0}{u_0})$.
Since $\tau_0<\tau_{f}^{e}(S_0)<\tau_1^i(S_0)$, there exits a $\phi_d<\alpha_0$
such that the initial value problem (\ref{102304}), (\ref{102305a}) admits a solution for $\theta\in [\phi_{d}, \alpha_0)$, and the solution satisfies
\begin{equation}\label{N81601}
\bar{q}'(\theta)<0, \quad  \bar{\sigma}'(\theta)>0, \quad \bar{\tau}'(\theta)<0,  \quad \mbox{and}\quad 0<\bar{A}(\theta)<\frac{\pi}{2}.
\end{equation}
for $\theta\in [\phi_{d}, \alpha_0)$.
Moreover, $\bar{\tau}(\phi_{d})=\tau_f^e(S_0)$.
For convenience, we denote the solution of the problem (\ref{102304}), (\ref{102305a})
by $(\bar{q}_l, \bar{\tau}_l, \bar{\sigma}_l, \bar{S}_l)(\theta)$, $\theta\in [\phi_{d}, \alpha_0]$.

Set
$\bar{u}_{l}(\theta)=\bar{q}_l(\theta)\cos(\bar{\sigma}_l(\theta))$ and $\bar{v}_{l}(\theta)=\bar{q}_l(\theta)\sin(\bar{\sigma}_l(\theta))$.
Then
$$(u, v, \tau, S)=(\bar{u}_{l}, \bar{v}_{l}, \bar{\tau}_{l}, \bar{S}_{l})(\arctan(y/x))$$ is a centered simple wave  solution of (\ref{PSEU}) on $\big\{(x, y)~\big|~(x, y)=(r\cos\theta, r\sin\theta),~ r>0,
~\theta\in(\phi_d, \alpha_0)
\big\}$.
Since $\bar{\tau}_{l}(\phi_d)=\tau_f^e(S_0)$ and $\bar{S}_{l}(\phi_d)=S_0$, the fan wave can be  followed by a double-sonic shock located at $\theta=\phi_d$, and the front side state of the  shock is $(\bar{u}_l(\phi_d), \bar{v}_l(\phi_d), \bar{\tau}_l(\phi_d), S_0)$. We denote by $(u_d, v_d, \tau_d, S_d)$ the back side state of the double-sonic shock. By the previous discussion we have
\begin{equation}\label{81405}
\tau_d=\hat{\eta}(\tau_f^e(S_0))\tau_f^e(S_0), \quad S_d=\hat{S}(\tau_d).
\end{equation}
The velocity $(u_d, v_d)$ can then be determined by (\ref{52304}), (\ref{6810a}) and the first and the third relations of (\ref{RHH}).
Let
$$
\sigma_d=\arctan\Big(\frac{v_d}{u_d}\Big)\quad \mbox{and}\quad A_d=\arcsin\left(\frac{c_d}{q_d}\right),
$$
where $q_d=\sqrt{u_d^2+v_d^2}$ and $c_d=\sqrt{-\tau_d^2p_{\tau}(\tau_d, S_d)}$.
By $\tau_d>\bar{\tau}_l(\phi_d)$ we have $\sigma_d<\bar{\sigma}_{l}(\phi_{d})$.

We next consider (\ref{102304}) with data
\begin{equation}\label{102305}
(\bar{q}, \bar{\tau}, \bar{\sigma}, \bar{S})(\phi_{d})=(q_d, \tau_d, \sigma_d, S_d).
\end{equation}

Let
$$
\hat{q}(\tau)=\sqrt{q_d^2+2h(\tau_d, S_d)-h(\tau, S_d)}, \quad \tau\geq\tau_d.
$$
Since $\tau_d=\tau_{b}^e(S_d)>\tau_2^i(S_d)$,
 the function $q=\hat{q}(\tau)$ is strictly monotonically increasing for $\tau>\tau_d$ and has an inverse function denoted by $\tau=\hat{\tau}(q)$ for $q>q_d$.
Let
\begin{equation}\label{3506}
\alpha_v=-\int_{q_d}^{q_{_{lim}}}\frac{\sqrt{q^2-\hat{c}^2(q)}}{q\hat{c}(q)}{\rm d}q,
\end{equation}
where
$$
q_{_{lim}}=\sqrt{q_d^2+2h(\tau_d, S_d)-h_{\infty}}, \quad h_{\infty}=\lim\limits_{\tau\rightarrow +\infty}h(\tau, S_d), \quad \mbox{and}\quad  \hat{c}(q)=\hat{\tau}(q)\sqrt{-p_{\tau}(\hat{\tau}(q), S_d)}.
$$
Then the initial value problem (\ref{102304}), (\ref{102305}) admits a solution for $\theta\in (\alpha_v, \phi_{d})$. Moreover, the solution satisfies (\ref{N81601}) for $\theta\in (\alpha_v, \phi_{d})$. Moreover,
$$
 \lim\limits_{\theta\rightarrow \alpha_v}\bar{\tau}(\theta)=+\infty \quad \mbox{and}\quad \lim\limits_{\theta\rightarrow \alpha_v}\bar{q}(\theta)=q_{_{lim}}.
$$
For convenience, we denote the solution of the problem (\ref{102304}), (\ref{102305})
by $(\bar{q}_r, \bar{\tau}_r, \bar{\sigma}_r, \bar{S}_r)(\theta)$, $\theta\in (\alpha_v, \phi_d]$.
Let
$\bar{u}_{r}(\theta)=\bar{q}_r(\theta)\cos(\bar{\sigma}_r(\theta))$ and $\bar{v}_{r}(\theta)=\bar{q}_r(\theta)\sin(\bar{\sigma}_r(\theta))$.
Then
$$(u, v, \tau, S)=(\bar{u}_{r}, \bar{v}_{r}, \bar{\tau}_{r}, \bar{S}_{r})(\arctan(y/x)), \quad \alpha_v<\arctan(y/x)\leq \phi_d$$ is a centered simple  wave  solution of (\ref{PSEU}) on
$\big\{(x, y)~\big|~(x, y)=(r\cos\theta, r\sin\theta),~ r>0,
~\theta\in(\alpha_v, \phi_d)
\big\}$.

When
 $\alpha_v<\theta_{w}<\sigma_d$, there exists a unique $\alpha_w\in (\alpha_v, \phi_d)$ such that
$\bar{v}_r(\alpha_w)=\bar{u}_r(\alpha_w)\tan\theta_{w}$. In this case, the problem (\ref{PSEU}), (\ref{RBD}) has a self-similar fan-shock-fan composite wave solution which can be represented by
\begin{equation}\label{72801a}
(u, v, \tau, S)=\left\{
               \begin{array}{ll}
                 (u_0, 0, \tau_0, S_0), & \hbox{$\alpha_0\leq\theta\leq \frac{\pi}{2}$;} \\[2pt]
                 (\bar{u}_l, \bar{v}_l, \bar{\tau}_l, \bar{S}_l)(\theta), & \hbox{$\phi_d<\theta<\alpha_0$;} \\[2pt]
 (\bar{u}_r, \bar{v}_r, \bar{\tau}_r, \bar{S}_r)(\theta), & \hbox{${\alpha}_w<\theta<\phi_d$;} \\[2pt]
                 (u_w, v_w, \tau_w, S_d), & \hbox{$\theta_w<\theta<{\alpha}_w$,}
               \end{array}
             \right.
\end{equation}
where  $(u_w, v_w, \tau_w)=(\bar{u}_r, \bar{v}_r, \bar{\tau}_r)(\alpha_w)$ and $(x,y)=(r\cos\theta, r\sin\theta)$; see Figure \ref{Figure5}.


\begin{figure}[htbp]
\begin{center}
\includegraphics[scale=0.45]{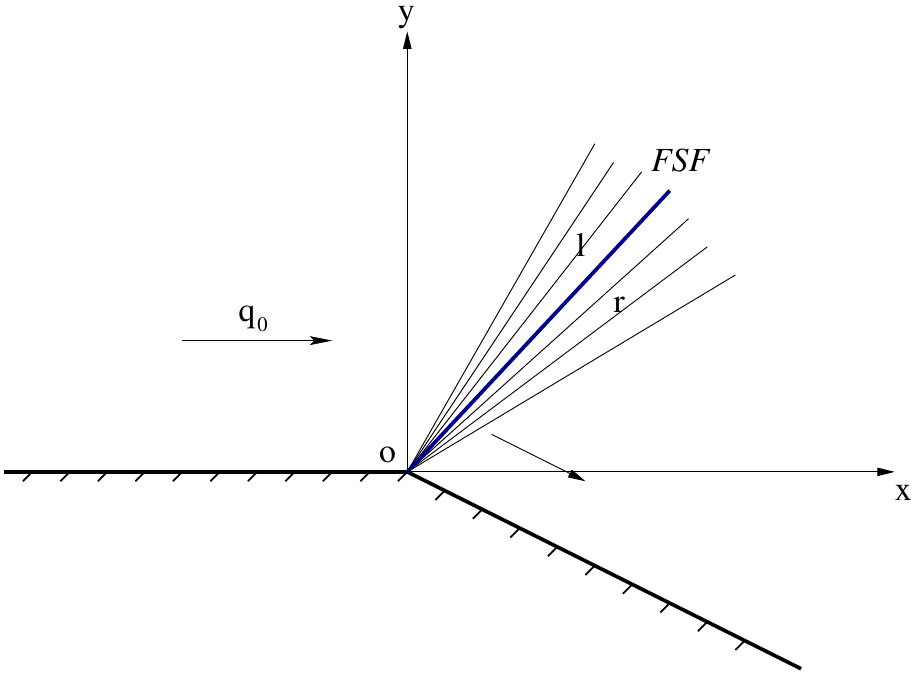}
\caption{\footnotesize An oblique fan-shock-fan composite wave.}
\label{Figure5}
\end{center}
\end{figure}

When
 $\theta_{w}\leq \alpha_v$, the problem (\ref{PSEU}), (\ref{RBD}) has a self-similar fan-shock-fan composite wave solution which can be represented by
$$
(u, v, \tau, S)=\left\{
               \begin{array}{ll}
                 (u_0, 0, \tau_0, S_0), & \hbox{$\alpha_0\leq\theta\leq \frac{\pi}{2}$;} \\[2pt]
                 (\bar{u}_l, \bar{v}_l, \bar{\tau}_l, \bar{S}_l)(\theta), & \hbox{$\phi_d<\theta<\alpha_0$;} \\[2pt]
 (\bar{u}_r, \bar{v}_r, \bar{\tau}_r, \bar{S}_r)(\theta), & \hbox{${\alpha}_v<\theta<\phi_d$;} \\[2pt]
                \mbox{vacuum}, & \hbox{$\theta_w<\theta<{\alpha}_v$.}
               \end{array}
             \right.
$$

\subsection{Stability of the fan-shock-fan composite wave}

\subsubsection{\bf Problem and the main result}
We now study the stability of the oblique fan-shock-fan composite wave (\ref{72801a}).
We consider (\ref{PSEU}) with the boundary conditions
\begin{equation}\label{RBD1}
\left\{
  \begin{array}{ll}
    (u, v, \tau, S)(0, y)=(u_i, v_{i}, \tau_i, S_i)(y), & \hbox{$y\geq 0$;} \\[4pt]
 v(x,y)=u(x,y)w'(x), & \hbox{$y=w(x)$, $x>0$.}
  \end{array}
\right.
\end{equation}
In (\ref{RBD1}),
 $(u_i, v_{i}, \tau_i, S_i)(y)\in C^1[0, +\infty)$ and $w(x)\in C^2[0,+\infty)$ are given functions which satisfy
\begin{equation}\label{8280a}
q_{i}(y)>c_{i}(y)\quad \mbox{and} \quad -\frac{\pi}{2}<\beta_{i}(y)<\alpha_i(y)<\frac{\pi}{2}\quad  \mbox{for}\quad y>0,
\end{equation}
where $q_{i}=\sqrt{u_i^2+v_i^2}$, $c_i=\tau_i\sqrt{-p_{\tau}(\tau_i, S_i)}$, $\beta_i=\arctan(\frac{v_i}{u_i})-\arcsin(\frac{c_i}{q_i})$, and $\alpha_i=\arctan(\frac{v_i}{u_i})+\arcsin(\frac{c_i}{q_i})$.
So, the problem (\ref{PSEU}), (\ref{RBD1}) is a initial-bundary value problem.

Set $(u_0, v_0, \tau_0, S_0)=(u_i, v_{i}, \tau_i, S_i)(0)$,  $c_0=\sqrt{-\tau_0^2p_{\tau}(\tau_0, S_0)}$ and $\theta_w=\arctan w'(0)$. We assume
\begin{description}
\item[(A1)] $u_0>c_0$, $v_0=0$, $S_{cr}<S_0<S_*$, $1<\tau_0<\tau_f^e(S_0)$, and $\alpha_v<\theta_w<\sigma_d$, where the constants $\alpha_v$ and $\sigma_d$ are the same as that defined in Section 3.5.
\end{description}
By the result of Section 3.5 we know that if $(u_i, v_i, \tau_i, S_i)(y)\equiv(u_0, 0, \tau_0, S_0)$ for $y>0$ and $w''(x)\equiv 0$ for $x>0$, then the problem (\ref{PSEU}), (\ref{RBD1}) admits a self-similar fan-shock-fan composite wave solution (\ref{72801a}).

On the angular $\{(x,y)\mid w(x)<y<x\tan\alpha_w, x>0\}$, we consider (\ref{PSEU}) with the boundary conditions
\begin{equation}\label{72805}
\left\{
  \begin{array}{ll}
    (u, v, \tau, S)= (\bar{u}_r, \bar{v}_r, \bar{\tau}_r, \bar{S}_r)(\alpha_w), & \hbox{$y=x\tan\alpha_w$, $x>0$;} \\[4pt]
    v(x,y)=u(x,y)w'(x), & \hbox{$y=w(x)$, $x>0$,}
  \end{array}
\right.
\end{equation}
where the function $(\bar{u}_r, \bar{v}_r, \bar{\tau}_r, \bar{S}_r)(\theta)$ and the constant $\alpha_w$ are given in the last subsection.
By \cite{CaF} and \cite{Li-Yu} we know that the problem (\ref{PSEU}), (\ref{72805}) admits a local simple wave solution, even if $w''(0)\neq0$.  This implies that if $(u_i, v_i, \tau_i, S_i)(y)\equiv(u_0, 0, \tau_0, S_0)$ for $y>0$, then near the corner the self-similar fan-shock-fan composite wave will not be perturbed by the curved ramp.
So, we  assume that the ramp is straight, i.e., $w(x)=x\tan\theta_w$, $x>0$.

\begin{figure}[htbp]
\begin{center}
\includegraphics[scale=0.5]{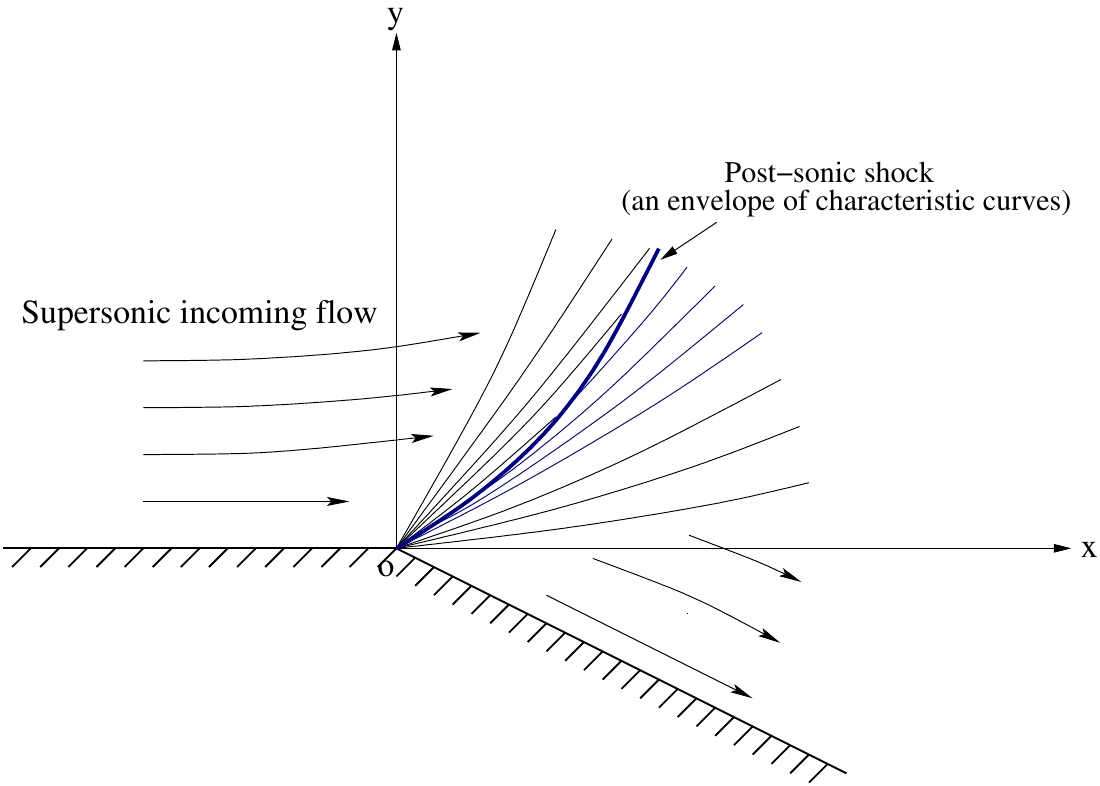}
\caption{\footnotesize A 2D steady supersonic flow with a post-sonic shock and two centered waves past a rarefactive ramp.}
\label{Figure6}
\end{center}
\end{figure}

By the structure of the self-similar solution (\ref{72801a}), we can deduce that there is a shock issued from the origin and the shock is double-sonic at the origin. However, the type of the shock away from the origin is a priori unknown.
Actually, the type will be determined by the incoming flow.
We purpose to construct a supersonic solution with a post-sonic shock near the corner,
so we also make the following assumptions about the incoming flow:
\begin{description}
  \item[(A2)] $B(u_i, v_{i}, \tau_i, S_i)=\mbox{Const.}$ and $S_i\equiv S_0$;
  \item[(A3)] $u_{i}'(0)+v_{i}'(0)\tan\alpha_0>0$, where $\alpha_0=\arcsin(\frac{c_0}{u_0})$.
\end{description}
The main result is stated as follows.
\begin{thm}\label{main}
Assume that assumptions {\bf(A1)--(A3)} hold. Then there exist small $\delta>0$ and $\delta'>0$ such that the problem (\ref{PSEU}), (\ref{RBD1}) admits a local piecewise Lipschitz continuous solution on the domain $\{(x, y)\mid x\tan\theta_{w}\leq y\leq\delta', 0\leq x\leq\delta\}$. The solution contains a single shock issued from the origin. Moreover, the shock is post-sonic and is an envelope of $C_{+}$ characteristic curves of the flow downstream of it; see Figure \ref{Figure6}.
\end{thm}

In the following discussion, we will construct the solution piece by piece. We first show that under assumption {\bf(A3)} the shock away from the origin is post-sonic, and then use the characteristic decomposition method to overcome the difficulty caused by the singularity on the back side of the  post-sonic shock.

\subsubsection{\bf Cauchy problem}
We first consider (\ref{PSEU}) with the data
\begin{equation}\label{61901}
(u, v, \tau, S)(0, y)=(u_i, v_{i}, \tau_i, S_i)(y), \quad y\geq0.
\end{equation}
By (\ref{8280a}) the problem (\ref{PSEU}), (\ref{61901}) is a Cauchy problem.
\begin{lem}
Under assumptions {\bf (A1)--(A3)}, there exists a small $\delta>0$ such that the Cauchy problem (\ref{PSEU}), (\ref{61901}) admits a classical solution on the domain  $\{(x, t)\mid y_2(x)\leq y\leq y_2(\delta), 0\leq x\leq\delta\}$, where $y=y_2(x)$ is a $C_{+}$ characteristic curve issued from the origin, i.e.,
$$
\frac{{\rm d}y_2(x)}{{\rm d}x}=\tan\alpha(x, y_2(x)), \quad y_2(0)=0;
$$
see Figure \ref{Figure7} (left).
Moreover, the solution satisfies
\begin{equation}\label{61902}
-3\mathcal{L}_0<\bar{\partial}_{+}\rho<-\mathcal{L}_0\quad \mbox{on}\quad y=y_2(x),\quad 0\leq x\leq\delta,
\end{equation}
where $$\mathcal{L}_0:=\frac{u_{i}'(0)+v_{i}'(0)\tan\alpha_0}{2c_0\tau_0}.$$
\end{lem}
\begin{proof}
The local existence of a classical solution follows routinely from the idea of Li and Yu \cite{{Li-Yu}} (Chapter 2).
Moreover, by assumption {\bf(A3)} and (\ref{61402}) we know that the solution is isentropic and irrotational.
Furthermore, by (\ref{form}) we have that the solution satisfies
$$
\bar{\partial}_{+}u=-\tan\beta\bar{\partial}_{+}v=-\cos\alpha\tan\beta(u_y+\tan\alpha v_y)\quad \mbox{for}\quad x=0,\quad y\geq 0.
$$
Hence, by (\ref{11}) and $\alpha(0, 0)=-\beta(0,0)=\alpha_0$ we have that the solution satisfies
$$
(\bar{\partial}_{+}\rho)(0, 0)=-\frac{u_{i}'(0)+ v_{i}'(0)\tan\alpha_0}{c_0\tau_0}.
$$
Thus, when $\delta$ is sufficiently small there holds the estimate (\ref{61902}).
This completes the proof.
\end{proof}

For convenience, we denote the local solution of the Cauchy problem (\ref{PSEU}), (\ref{61901})  by $(u, v, \tau, S)=(u_1(x,y), v_1(x,y), \tau_1(x,y), S_0)$.

\begin{figure}[htbp]
\begin{center}
\includegraphics[scale=0.48]{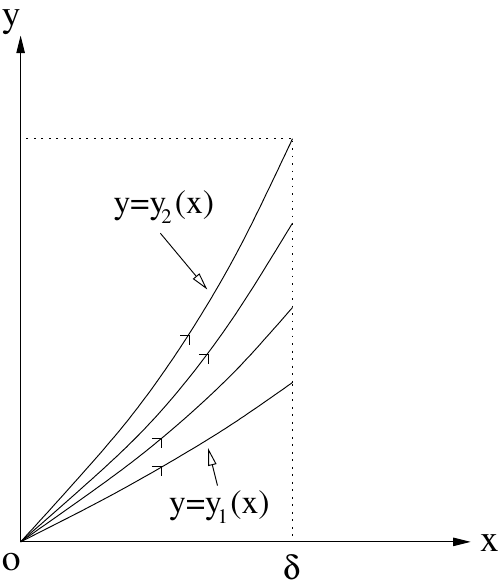}\qquad\quad \includegraphics[scale=0.48]{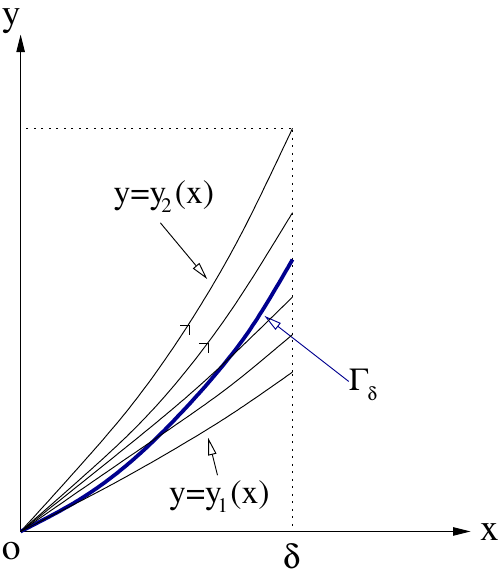}\quad \qquad \includegraphics[scale=0.48]{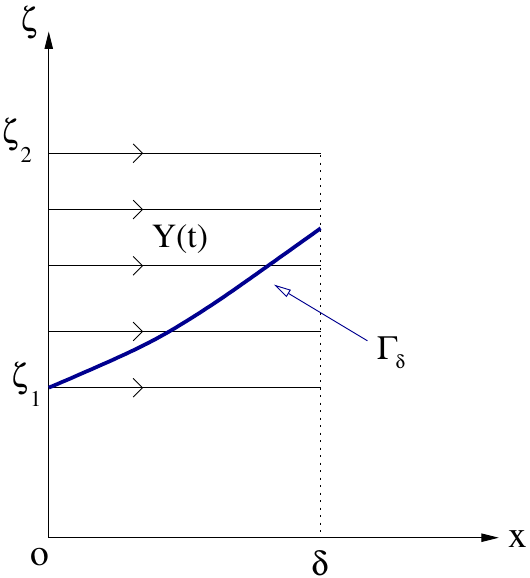}
\caption{\footnotesize Centered wave and post-sonic shock.}
\label{Figure7}
\end{center}
\end{figure}

\subsubsection{\bf Centered wave problem}
\begin{defn}\label{defn1}
(Li and Yu \cite{Li-Yu}) Let $Y(\delta)$ be an angular domain with curved boundaries:
\begin{equation}
Y(\delta)~:=~\{(x,y)\mid 0< x\leq\delta,~ y_{1}(x)\leq y\leq y_{2}(x)\},
\end{equation}
where
$y_{1}(0)=y_{2}(0)=0$.
A function $(u, v, \tau, S)(x,y)$ is called a $C_{+}$ type centered wave of (\ref{PsEuler}) with $(0,0)$ as the center point if the following properties are satisfied:
\begin{enumerate}
  \item $(u, v, \tau, S)$ can be implicitly determined by the functions
$y=\eta(x,\zeta)$ and
$(u, v, \tau, S)=(\tilde{u}, \tilde{v}, \tilde{\tau}, \tilde{S})(x,\zeta)$
defined on a rectangular domain
$T(\delta)~:=~\{(x,\zeta)\mid 0\leq x\leq\delta, ~~\zeta_1\leq \zeta\leq \zeta_2\}$, where $\zeta_1=y_{1}'(0)<y_{2}'(0)=\zeta_2$.
Moreover, $\eta$ and $(\tilde{u}, \tilde{v}, \tilde{\tau}, \tilde{S})$ belong to $C^{1}\big(T(\delta)\big)$, and for any $(x,\zeta)\in T(\delta)\setminus\{x=0\}$ there holds
$\eta_{_{\zeta}}(x,\zeta)>0$.
\item The function $(u, v, \tau, S)(x,y)$ defined above satisfies (\ref{PsEuler})  on $Y(\delta)$;
\item For any $\zeta\in[\zeta_1,\zeta_2]$, $y=\eta(x,\zeta)$ gives the $C_{+}$ characteristic line issued from $(0,0)$ with the slope $\zeta$ at $(0,0)$, i.e.,
\begin{equation}\label{pgz}
    \eta_{x}(x,\zeta)=\lambda_{+}(\tilde{u}, \tilde{v}, \tilde{c})=\tan\tilde{\alpha},\quad
    \eta(0,\zeta)=0,\quad
\eta_{x}(0,\zeta)=\zeta;
\end{equation}
\item $\zeta=\zeta_1$ and $\zeta=\zeta_2$ correspond to $y=y_1(x)$ and $y=y_2(x)$, respectively.
\end{enumerate}
We call $(\tilde{u}, \tilde{v}, \tilde{\tau}, \tilde{S})(0,\zeta)$ the principal part of this $C_{+}$ type centered wave and $\zeta$ the characteristic parameter.
\end{defn}

We take the coordinate transformation
\begin{equation}
x=x,\quad y=\eta(x, \zeta).
\end{equation}
By computation we have
\begin{equation}\label{transd}
\partial_{x}= \partial_{x}-\eta_x \eta_\zeta^{-1}\partial_\zeta, \quad
\partial_{y}= \eta_\zeta^{-1}\partial_\zeta.
\end{equation}

Let
\begin{equation}\label{61905a}
\partial_{+}=\partial_x+\tan\alpha\partial_y, \quad \partial_{-}=\partial_x+\tan\beta\partial_y.
\end{equation}
Then, in terms of the $(x, \zeta)$ coordinates,
\begin{equation}\label{61905}
\partial_{+}=\partial_x, \quad \partial_{-}=\partial_x+(\tan\tilde{\beta}-\tan\tilde{\alpha})\eta_{\zeta}^{-1}\partial_{\zeta}.
\end{equation}

From (\ref{pgz}) we have
\begin{equation}\label{21901}
\frac{\partial \eta}{\partial \zeta}(x,\zeta)=\int_{0}^{x}\sec^2(\tilde{\alpha}(r, \zeta))\frac{\partial\tilde{\alpha}(r, \zeta)}{\partial \zeta}{\rm d}r.
\end{equation}
Moreover, by (\ref{pgz}) we have
\begin{equation}\label{61908}
\sec^2(\tilde{\alpha}(0, \zeta))\frac{\partial\tilde{\alpha}(0, \zeta)}{\partial \zeta}=1.
\end{equation}

Let $(\tilde{u}_{+}, \tilde{v}_{+}, \tilde{\tau}_{+}, \tilde{S}_{+})(\zeta)=(\tilde{u}, \tilde{v}, \tilde{\tau}, \tilde{S})(0,\zeta)$.
Then by  (\ref{pgz}) one has $\zeta=\lambda_{+}(\tilde{u}_{+}, \tilde{v}_{+}, \tilde{c}_{+})$.
Inserting (\ref{transd}) into (\ref{PsEuler}) and letting $x\rightarrow 0$, we have that $(\tilde{u}_{+}, \tilde{v}_{+}, \tilde{\tau}_{+}, \tilde{S}_{+})(\zeta)$ satisfies
\begin{equation}\label{principal}
\left\{
  \begin{array}{ll}
   \displaystyle -\zeta\tilde{\rho}_{+}\frac{{\rm d} \tilde{u}_{+}}{{\rm d} \zeta}-
\zeta\tilde{u}_{+}\frac{{\rm d} \tilde{\rho}_{+}}{{\rm d} \zeta} +\tilde{\rho}_{+}\frac{{\rm d} \tilde{v}_{+}}{{\rm d} \zeta}+
\tilde{v}_{+}\frac{{\rm d} \tilde{\rho}_{+}}{{\rm d} \zeta}=0,\\[8pt]
 \displaystyle -\zeta\tilde{u}_{+}\frac{{\rm d}  \tilde{u}_{+}}{{\rm d} \zeta}+\tilde{v}_{+}\frac{{\rm d}  \tilde{u}_{+}}{{\rm d} \zeta}-
\zeta\tilde{\tau}_{+}\frac{{\rm d} \tilde{p}_{+}}{{\rm d} \zeta}=0,\\[8pt]
\displaystyle -\zeta\tilde{u}_{+}\frac{{\rm d}  \tilde{v}_{+}}{{\rm d} \zeta}+\tilde{v}_{+}\frac{{\rm d}  \tilde{v}_{+}}{{\rm d} \zeta}+
\tilde{\tau}_{+}\frac{{\rm d} \tilde{p}_{+}}{{\rm d} \zeta}=0,\\[8pt]
\displaystyle (\tilde{v}_{+}-\tilde{u}_{+}\zeta)\frac{{\rm d} \tilde{S}_{+}}{{\rm d} \zeta}=0.
  \end{array}
\right.
\end{equation}
This implies
$$
\frac{{\rm }d\tilde{S}_{+}}{{\rm d}\zeta}=0,  \quad \frac{{\rm d} \tilde{u}_{+}}{{\rm d} \zeta}+\zeta\frac{{\rm d} \tilde{v}_{+}}{{\rm d} \zeta}=0,\quad \mbox{and}\quad
\frac{{\rm }d}{{\rm d}\zeta}B(\tilde{u}_{+}, \tilde{v}_{+}, \tilde{\tau}_{+}, \tilde{S}_{+})=0.
$$


By the previous result we know that if we let
$$(\tilde{u}_{+}, \tilde{v}_{+}, \tilde{\tau}_{+}, \tilde{S}_{+})(\zeta)=(\bar{u}_l, \bar{v}_l, \bar{\tau}_l, \bar{S}_l)(\arctan\zeta), \quad\zeta_1\leq \zeta\leq\zeta_2,$$
where $\zeta_1=\tan\phi_d$ and $\zeta_2=\tan\alpha_0$.
Then $(\tilde{u}_{+}, \tilde{v}_{+}, \tilde{\tau}_{+}, \tilde{S}_{+})(\zeta)$ satisfies (\ref{principal}) and
\begin{equation}\label{61903}
(\tilde{u}_{+}, \tilde{v}_{+}, \tilde{\tau}_{+}, \tilde{S}_{+})(\zeta_2)=(u_0, 0, \tau_0, S_0).
\end{equation}

In order to construct the flow upstream of the shock propagated from the origin, we consider the following {\bf Centered wave problem}. Construct a $C_{+}$ type centered wave
such that it
 satisfies the boundary condition
\begin{equation}
(u, v, \tau, S)=(u_1(x,y), v_1(x,y), \tau_1(x,y), S_0)\quad \mbox{on} \quad y=y_2(x),\quad  0\leq x\leq\delta,
\end{equation}
and has the principle part
\begin{equation}\label{61906}
(\tilde{u}, \tilde{v}, \tilde{\tau}, \tilde{S})(0,\zeta)=(\bar{u}_l, \bar{v}_l, \bar{\tau}_l, \bar{S}_l)(\arctan\zeta), \quad \zeta_1\leq \zeta\leq \zeta_2.
\end{equation}

\begin{lem}
Assume that $\delta$ is sufficiently small. Then the centered wave problem admits a classical solution on $Y(\delta)$. Moreover, in the domain $Y(\delta)$ the solution satisfies
\begin{equation}\label{61902a}
\rho_c\leq \rho\leq \rho_0, \quad
\bar{\partial}_{-}\rho<0,\quad \mbox{and}\quad
-3e^{2(\rho_0-\rho_c)\mathcal{J}_1}\mathcal{L}_0<\bar{\partial}_{+}\rho<-e^{4\mathcal{J}_1\mathcal{J}_2(\rho_c-\rho_0)}\mathcal{L}_0,
\end{equation}
where
$$
\begin{aligned}
&\rho_c=\tau_c^{-1}, \quad \tau_c=\frac{1}{2}\big[\tau_1^i(S_0)+\tau_f^e(S_0)\big],\quad
\mathcal{J}_1=\max\limits_{\theta\in [\phi_d, \alpha_0]}\left\{\frac{\bar{\tau}_{l}^4p_{\tau\tau}(\bar{\tau}_l, S_0)}{4\bar{c}_{l}^2\cos\bar{A}_{l}}\right\}\\&\quad \mbox{and}\quad \mathcal{J}_2=\max\limits_{\theta\in [\phi_d, \alpha_0]}\left\{2\sin^2\bar{A}_{l}-\frac{8p_{\tau}(\bar{\tau}_{l}, S_0)\cos^4\bar{A}_{l}}{\bar{\tau}_{l}p_{\tau\tau}(\bar{\tau}_{l}, S_0)}\right\}.
\end{aligned}
$$
\end{lem}
\begin{proof}
By (\ref{61402}), (\ref{81301}), and $\bar{\partial}_{+}S=0$ on $y=y_2(x)$ we know that
the centered wave solution is isentropic and irrotational. Then the local existence can be obtained by a hodograph transformation method; see Li and Yu  \cite{{Li-Yu}} (Chapter 7.7).
The reader can also refer to Section 2.4.3 in the present paper.
It only remains to prove (\ref{61902a}).

By (\ref{61905}) we have
\begin{equation}\label{61907}
\bar{\partial}_{-}\rho=\cos\tilde{\beta}\left(\tilde{\rho}_x
+(\tan\tilde{\beta}-\tan\tilde{\alpha})\eta_{\zeta}^{-1}\tilde{\rho}_{\zeta}\right)=\cos\tilde{\beta}\tilde{\rho}_x-\frac{\sin (2\tilde{A})}{\cos\alpha}\eta_{\zeta}^{-1}\tilde{\rho}_{\zeta}.
\end{equation}
From (\ref{61906}) and (\ref{102304}) we have
\begin{equation}\label{8290a}
\tilde{\rho}_{\zeta}(0, \zeta)=\frac{\bar{\rho}_{l}'(\arctan\zeta)}{1+\zeta^2}>0\quad \mbox{for}\quad \zeta_1\leq \zeta\leq \zeta_2.
\end{equation}
Thus, by (\ref{21901}), (\ref{61908}), (\ref{61907}) we have that when $\delta$ is sufficiently small,
$\bar{\partial}_{-}\rho<0$ in $Y(\delta)$.

For any point in $Y(\delta)$ the backward $C_{-}$ characteristic curve issued from this point stays in
 $Y(\delta)$ until it intersects $y=y_2(x)$ ($0<x<\delta$) at some point.
So, by  (\ref{61902}) and the second equation of (\ref{cd}) we know that the solution satisfies $\bar{\partial}_{+}\rho<0$.
Furthermore, by continuity we have that when
 $\delta$ is sufficiently small,
\begin{equation}\label{61910}
\rho_c\leq \rho\leq \rho_0, \quad \varphi<2\mathcal{J}_2, \quad \mbox{and}\quad  \frac{\tau^4p_{\tau\tau}}{4c^2\cos^{2}A }<2\mathcal{J}_1\quad\mbox{in}\quad  Y(\delta).
\end{equation}

From the second equation of (\ref{cd}) we have
$$
 \bar{\partial}_{-}\bar{\partial}_{+}\rho~=~
    \frac{\tau^4p_{\tau\tau}}{4c^2\cos^{2}A }\Big[(\bar{\partial}_{+}\rho)^2+(\varphi-1)\bar{\partial}_{-}
\rho\bar{\partial}_{+}\rho\Big]>-\frac{\tau^4p_{\tau\tau}}{4c^2\cos^{2}A }\bar{\partial}_{-}\rho\bar{\partial}_{+}\rho\quad\mbox{in}\quad  Y(\delta).
$$
Integrating this along $C_{-}$ characteristic curves issued from points on $y=y_2(x)$, $0<x<\delta$ and recalling (\ref{61910}) and  (\ref{61902}) we have $\bar{\partial}_{+}\rho>-3e^{2(\rho_0-\rho_c)\mathcal{J}_1}\mathcal{L}_0$ in $Y(\delta)$.

From (\ref{21901}) and (\ref{61907}) we know that when $\delta$ is sufficiently small,
\begin{equation}\label{7301}
\bar{\partial}_{-}\rho<-3e^{2(\rho_0-\rho_c)\mathcal{J}_1}\mathcal{L}_0  \quad\mbox{in}\quad  Y(\delta).
\end{equation}

From the second equation of (\ref{cd}) and (\ref{7301}) we have that when $\delta$ is sufficiently small, the solution satisfies
$$
 \bar{\partial}_{-}\bar{\partial}_{+}\rho~=~
    \frac{\tau^4p_{\tau\tau}}{4c^2\cos^{2}A }\Big[(\bar{\partial}_{+}\rho)^2+(\varphi-1)\bar{\partial}_{-}\rho\bar{\partial}_{+}\rho\Big]
<\frac{\tau^4p_{\tau\tau}}{4c^2\cos^{2}A }\varphi\bar{\partial}_{-}\rho\bar{\partial}_{+}\rho\quad\mbox{in}\quad  Y(\delta).
$$
Integrating this along $C_{-}$ characteristic curves issued from $y=y_2(x)$ and recalling (\ref{61910}) and  (\ref{61902}) we have $\bar{\partial}_{+}\rho<-e^{4\mathcal{J}_1\mathcal{J}_2(\rho_c-\rho_0)}\mathcal{L}_0$ in $Y(\delta)$.

This completes the proof.
\end{proof}

For convenience, we still denote by $(u, v, \tau, S)=(u_1(x,y), v_1(x,y), \tau_1(x,y), S_0)$, $(x, y)\in Y(\delta)$ the centered wave solution.

\subsubsection{\bf Existence of a post-sonic shock}
We denote by $\Gamma$ the shock issued front the origin. From the self-similar solution (\ref{72801a}) we know that the inclination angle of the shock at the origin is $\phi_d$.
In the following discussion we still use $\bar{\partial}_{_\Gamma}$ to denote
 the directional derivative along the shock.

We compute
$$
N=u\sin\phi-v\cos\phi=q\cos\sigma\sin\phi-q\sin\sigma\cos\phi=q\sin(\phi-\sigma).
$$
Thus, we have
\begin{equation}\label{61112}
\phi=\sigma+\arcsin\Big(\frac{N}{q}\Big)\quad \mbox{on}\quad\Gamma.
\end{equation}




From the first relation of (\ref{RHH})  we have
$$
m=\rho  N =\rho  q  (\cos\sigma \sin\phi-\sin\sigma \cos\phi).
$$
Thus, we have
\begin{equation}\label{4250103}
\bar{\partial}_{_\Gamma}m=\frac{N}{q}\bar{\partial}_{_\Gamma}(\rho q )-\rho L \bar{\partial}_{_\Gamma}\sigma +\rho L \bar{\partial}_{_\Gamma}\phi\quad\mbox{along}\quad\Gamma.
\end{equation}

From the Bernoulli's law (\ref{BERC}) and the last two relations of (\ref{RHH}) we have
\begin{equation}\label{62010}
q \bar{\partial}_{_\Gamma} q +\tau  p_{\tau}\bar{\partial}_{_\Gamma} \tau+ h_{_S}\bar{\partial}_{_\Gamma} S  =0\quad \mbox{along}\quad \Gamma.
\end{equation}
This yields
$$
\bar{\partial}_{_\Gamma} (\rho q )=\frac{1}{q }(q ^2+\tau ^2 p_{\tau})\bar{\partial}_{_\Gamma}\rho -\frac{\rho h_{_S}\bar{\partial}_{_\Gamma} S }{q }\quad \mbox{along}\quad \Gamma.
$$
Inserting this into (\ref{4250103}), we get
\begin{equation}\label{250102}
\bar{\partial}_{_\Gamma} \phi=\frac{\bar{\partial}_{_\Gamma} m}{\rho L }-\frac{N }{\rho L q ^2}(q ^2+\tau ^2 p_{\tau})\bar{\partial}_{_\Gamma}\rho +\frac{N h_{_S}\bar{\partial}_{_\Gamma} S  }{Lq ^2}+\bar{\partial}_{_\Gamma}\sigma \quad \mbox{along}\quad \Gamma.
\end{equation}

From (\ref{250102}) we have
\begin{equation}\label{42804a}
\bar{\partial}_{_\Gamma} \sigma =\bar{\partial}_{_\Gamma} \phi-\frac{\bar{\partial}_{_\Gamma} m}{\rho L }+\frac{N }{\rho L q ^2}(q ^2+\tau ^2 p_{\tau})\bar{\partial}_{_\Gamma}\rho -\frac{N h_{_S}\bar{\partial}_{_\Gamma} S  }{Lq ^2}\quad \mbox{along}\quad \Gamma.
\end{equation}
We also have
\begin{equation}\label{50701}
\bar{\partial}_{_\Gamma}u=\bar{\partial}_{_\Gamma}(q\cos\sigma)=\cos\sigma\bar{\partial}_{_\Gamma}q
-q\sin\sigma\bar{\partial}_{_\Gamma}\sigma,
\end{equation}
\begin{equation}\label{50702}
\bar{\partial}_{_\Gamma}v=\bar{\partial}_{_\Gamma}(q\sin\sigma)=\sin\sigma\bar{\partial}_{_\Gamma}q
+q\cos\sigma\bar{\partial}_{_\Gamma}\sigma.
\end{equation}

We use subscripts `$f$' and `$b$' to denote the front side and back side states of the shock  $\Gamma$, respectively.

We first make an a priori assumption that away from the origin, the shock is post-sonic and stays in the interior of $Y(\delta)$.
Then $(u_f, v_f, \tau_f)=(u_1, v_1, \tau_1)(x, y)$ and $S_f\equiv S_0$ on $\Gamma$.
We also assume a priori that
\begin{equation}\label{72202}
\tau_f^e(S_0)<\tau_1(x,y)<\tau_1^i(S_0)\quad \mbox{on}\quad \Gamma\setminus \mathrm{O}.
\end{equation}
Then by Proposition \ref{81702} we have
\begin{equation}\label{62603}
\tau_b=\tau_{po}(\tau_1), \quad S_b=S_{po}(\tau_1),\quad \mbox{and}\quad
m^2=-p_{\tau}(\tau_{b},S_b)>-p_{\tau}(\tau_1, S_0)\quad \mbox{on}\quad \Gamma\setminus \mathrm{O}.
\end{equation}
This implies
\begin{equation}\label{6810}
N_b=c_b:=\sqrt{-\tau_b^2p_{\tau}(\tau_b, S_b)}\quad \mbox{and}\quad N_f>c_f:=\sqrt{-\tau_1^2p_{\tau}(\tau_1, S_0)}\quad \mbox{on}\quad \Gamma\setminus \mathrm{O}.
\end{equation}
In the following discussion we shall use subscript `$1$' to denote various quantities corresponding to the centered wave solution in Section 3.6.3.

By (\ref{6810}) and (\ref{61112}) we have
\begin{equation}\label{61113}
\phi=\sigma_b+A_b=\alpha_b\quad\mbox{and}\quad \phi>\sigma_1+A_1=\alpha_1  \quad \mbox{on}\quad \Gamma\setminus \mathrm{O}.
\end{equation}
We differentiate $m^2=-p_{\tau}(\tau_b, S_b)$ along $\Gamma$ to obtain
\begin{equation}\label{61105}
2m\bar{\partial}_{_\Gamma} m=-\Big(p_{\tau}(\tau_b, S_b)\tau_{po}'(\tau_1)+
p_{_S}(\tau_b, S_b)S_{po}'(\tau_1)\Big)\bar{\partial}_{_\Gamma} \tau_1\quad \mbox{along}\quad \Gamma.
\end{equation}
From (\ref{250102}) and $S_1\equiv S_0$ on $\Gamma$, we have
\begin{equation}\label{61115}
\bar{\partial}_{_\Gamma} \phi=\frac{\bar{\partial}_{_\Gamma} m}{\rho_1L_1}-\frac{N_1}{\rho_1L_1q_1^2}\Big(q_1^2+\tau_1^2 p_{\tau}(\tau_1, S_0)\Big)\bar{\partial}_{_\Gamma}\rho_1+\bar{\partial}_{_\Gamma}\sigma_1\quad \mbox{along}\quad \Gamma.
\end{equation}

Let
$\psi(x)$ be determined by
\begin{equation}\label{61102}
\left\{
  \begin{array}{ll}
    \psi'(x)=\tan(\phi(\psi(x), x)), & \hbox{$x>0$;} \\[2pt]
    \psi(0)=0.
  \end{array}
\right.
\end{equation}
Then the shock front can be represented by the function $y=\psi(x)$.
Actually, in order to obtain $\psi(x)$, we only need to consider (\ref{61115}) with data
\begin{equation}\label{61114}
\phi(0, 0)=\phi_d.
\end{equation}

Since the the initial value problem (\ref{61115}), (\ref{61114}) uses the flow states in $Y(\delta)$, in order to prove the existence of the post-sonic front we need to prove (\ref{72202}) and
\begin{equation}\label{61101}
\psi(x)>y_{1}(x), \quad x> 0
\end{equation}
where $y=y_1(x)$, $0<x<\delta$ represents the lower $C_{+}$ characteristic boundary of the centered wave; see Figure \ref{Figure7}.

Let $\partial_{_\Gamma}=\partial_x+\tan\phi\partial_{y}$.
Then by (\ref{61105}) and (\ref{61115}) we have
\begin{equation}\label{61601}
{\partial}_{_\Gamma} \phi=\left(\chi_{11}-\frac{N_1\cos^2 A_1}{\rho_1L_1}\right){\partial}_{_\Gamma}\rho_1 +{\partial}_{_\Gamma}\left(\frac{\alpha_1+\beta_1}{2}\right),
\end{equation}
where
$$
\chi_{11}=\frac{p_{\tau}(\tau_b, S_b)\tau_{po}'(\tau_1)+
p_{_S}(\tau_b, S_b)S_{po}'(\tau_1)}{2\rho_1^4N_1L_1}.
$$

By (\ref{61905a}) we have
$$
\partial_{_\Gamma}=(1-\chi_{12})\partial_{+}+\chi_{12}\partial_{-},
$$
where $$\chi_{12}=\frac{\tan\phi-\tan\alpha_1}{\tan\beta_1-\tan\alpha_1}.$$
Hence, by (\ref{1})--(\ref{2}) and noticing that the flow upstream of the shock is isentropic and irrotational,  we have
\begin{equation}\label{61602}
\begin{aligned}
\partial_{_\Gamma}\alpha_1=&-\frac{(1-\chi_{12})p_{\tau\tau}(\tau_1, S_0)}{4c_1^2\rho_1^4}\left(-\frac{4p_{\tau}(\tau_1,S_0)}{\tau_1 p_{\tau\tau}(\tau_1, S_0)}-1-\tan^2A_1\right)\sin (2A_1)\partial_+\rho_1\\&+\frac{\chi_{12} p_{\tau\tau}(\tau_1, S_0)\tan A_1}{2c_1^2\rho_1^4}\partial_{-}\rho_1
\end{aligned}
\end{equation}
and
\begin{equation}\label{61603}
\begin{aligned}
\partial_{_\Gamma}\beta_1=&\frac{\chi_{12}p_{\tau\tau}(\tau_1, S_0)}{4c_1^2\rho_1^4}\left(-\frac{4p_{\tau}(\tau_1,S_0)}{\tau_1 p_{\tau\tau}(\tau_1, S_0)}-1-\tan^2A_1\right)\sin (2A_1)\partial_-\rho_1\\&-\frac{(1-\chi_{12}) p_{\tau\tau}(\tau_1, S_0)\tan A_1}{2c_1^2\rho_1^4}\partial_{+}\rho_1.
\end{aligned}
\end{equation}

Since $(u_1, v_1, \rho_1)(x, y)$ has a multi-valued singularity at the origin, in the following discussion we shall use the $(x,\zeta)$ coordinates defined in Section 3.6.3 to compute $\partial_{_\Gamma}\phi$ and $\partial_{_\Gamma}\alpha_1$.

In terms of the $(x, \zeta)$ coordinates,
\begin{equation}\label{62002}
\partial_{_\Gamma}=\partial_x+\eta_{\zeta}^{-1}(\tan\phi-\tan\alpha_1)\partial_{\zeta}.
\end{equation}
Then by (\ref{61601})--(\ref{61603}) and (\ref{61905}) we get
\begin{equation}\label{62001}
\partial_{_\Gamma}(\phi-\alpha_1)=\chi_{13}\partial_x\rho_1+\chi_{13}\eta_{\zeta}^{-1}(\tan\phi-\tan\alpha_1)\partial_{\zeta}\rho_1,
\end{equation}
where
$$
\chi_{13}=\chi_{11}+\frac{\sin (2A_1)}{2\rho_1}-\frac{N_1\cos^2 A_1}{\rho_1L_1}-\frac{p_{\tau\tau}(\tau_1, S_0)\tan A_1}{2c_1^2\rho_1^4}.
$$
Since $\tau_1(0, \zeta_1)=\tau_f^e(S_0)$, by (\ref{73002}) we have $\chi_{11}(0, \zeta_1)=0$.
Moreover, at the point $(x, \zeta)=(0, \zeta_1)$, $N_1=c_1=q_1\sin A_1$ and $L_1=q_1\cos A_1$. Hence, we have
$$
\chi_{13}(0, \zeta_1)=\left(-\frac{p_{\tau\tau}(\tau_1, S_0)\tan A_1}{2c_1^2\rho_1^4}\right)(0, \zeta_1)<0.
$$

By the previous results about the centered wave, see (\ref{61905}), (\ref{61902a}), and (\ref{8290a}), we have
$$
\partial_x\rho_1(0, \zeta_1)<0\quad \mbox{and}\quad \partial_\zeta\rho_1(0, \zeta_1)>0.
$$
Therefore, by solving (\ref{62001}) with data $(\phi-\alpha_1)(0, \zeta_1)=0$ and recalling (\ref{21901}) and (\ref{61908}), we have
\begin{equation}\label{62003}
\big(\partial_{_\Gamma}(\phi-\alpha_1)\big)(0, \zeta_1)=\frac{\chi_{13}(0, \zeta_1)\partial_x\rho_1(0, \zeta_1)}{1-\chi_{13}(0, \zeta_1)\partial_\zeta\rho_1(0, \zeta_1)\sec^2\phi_d}>0.
\end{equation}
This implies
$$
\psi'(0)-y_1'(0)=\frac{\chi_{13}(0, \zeta_1)\partial_x\rho_1(0, \zeta_1)\sec^2\phi_d}{1-\chi_{13}(0, \zeta_1)\partial_\zeta\rho_1(0, \zeta_1)\sec^2\phi_d}>0.
$$

From (\ref{62002}) and (\ref{62003}) we have
\begin{equation}\label{62012}
\partial_{_\Gamma}\rho_1(0, \zeta_1)=\frac{\partial_x\rho_1(0, \zeta_1)}{1-\chi_{13}(0, \zeta_1)\partial_\zeta\rho_1(0, \zeta_1)\sec^2\phi_d}<0.
\end{equation}

Combining (\ref{62001})--(\ref{62012}) we know that
when $\delta$ is sufficiently small, the function $\psi(x)$, $0\leq x\leq \delta$ exists. Moreover,
\begin{equation}\label{9401}
\lim\limits_{x\rightarrow 0+}\tau_1(x, \psi(x))=\tau_f^e(S_0),
\end{equation}
and
$$
\psi(x)>y_{1}(x)\quad \mbox{and}\quad \tau_f^e(S_0)<\tau_1(x, \psi(x))<\tau_1^i(S_0) \quad \mbox{for}\quad 0<x\leq\delta.
$$
Therefore, we obtain the post-sonic shock front and the states $(u_b, v_b, \tau_b, S_b)(x,y)$ on its back side.
For convenience, we let
$$\Gamma_{\delta}=\big\{(x, y)\mid y=\psi(x), 0\leq x\leq\delta\big\}.$$

\vskip 4pt
The following estimates (monotonicity conditions) are crucial for the existence of a downstream flow of the post-sonic shock $\Gamma_{\delta}$.
\begin{lem}\label{61302}
(Structural conditions)
Assume $\delta$ to be sufficiently small. Then there hold the following estimates:
\begin{equation}\label{61111}
\cos\alpha_b \bar{\partial}_{_\Gamma} u_b+\sin\alpha_b \bar{\partial}_{_\Gamma} v_b>0\quad \mbox{along}\quad \Gamma_{\delta};
\end{equation}
\begin{equation}\label{61305}
\cos\beta_b \bar{\partial}_{_\Gamma} u_b+\sin\beta_b \bar{\partial}_{_\Gamma} v_b<0\quad \mbox{along}\quad \Gamma_{\delta}.
\end{equation}
\end{lem}
\begin{proof}
{\bf 1.}
Using  (\ref{61602}) and (\ref{61603}) and recalling (\ref{61905}), we have
$$
\partial_{_\Gamma}(\alpha_1+\beta_1)=
\frac{\sin 2A_1}{\rho_1}\Big((2\chi_{12}-1)\partial_{x}\rho_1+(\tan\phi-\tan\alpha_1)\eta_{\zeta}^{-1}\partial_{\zeta}\rho_1\Big).
$$
Thus, by (\ref{61601}) and (\ref{62003}) we have
\begin{equation}\label{62011}
(\partial_{_\Gamma}\phi)(0, \zeta_1)=\left[\frac{\sin 2A_1}{\rho_1}\bigg(
\frac{-\partial_x\rho_1+\chi_{13}\partial_x\rho_1\partial_\zeta\rho_1\sec^2\phi_d}{1-\chi_{13}\partial_\zeta\rho_1\sec^2\phi_d}\bigg)\right](0, \zeta_1)>0.
\end{equation}

By (\ref{50701})--(\ref{50702}) we have
\begin{equation}\label{42801}
\begin{aligned}
&\cos\alpha_b \bar{\partial}_{_\Gamma} u_b+\sin\alpha_b \bar{\partial}_{_\Gamma} v_b~=~
\cos A_b \bar{\partial}_{_\Gamma} q_b+q_b\sin A_b \bar{\partial}_{_\Gamma} \sigma_b\quad \mbox{along}\quad \Gamma_{\delta}.
\end{aligned}
\end{equation}

From (\ref{62010}), we have
\begin{equation}\label{42803}
\bar{\partial}_{_\Gamma} q_b=-\tau q_b \sin^2 A_b \bar{\partial}_{_\Gamma} \rho_b -\frac{h_{_S}(\tau_b, S_b)S_{po}'(\tau_1)\bar{\partial}_{_\Gamma} \tau_1 }{q_b }\quad \mbox{along}\quad \Gamma_{\delta}.
\end{equation}

From (\ref{42804a}) one has
\begin{equation}\label{42804}
\bar{\partial}_{_\Gamma} \sigma_b=\bar{\partial}_{_\Gamma} \phi-\frac{\bar{\partial}_{_\Gamma} m}{\rho_bL_b}+\frac{N_b}{\rho_bL_bq_b^2}(q_b^2+\tau_b^2 p_{\tau}(\tau_b, S_b))\bar{\partial}_{_\Gamma}\rho_b-\frac{N_b h_{_S}(\tau_b, S_b)S_{po}'(\tau_1)\bar{\partial}_{_\Gamma} \tau_1 }{L_bq_b ^2}\quad \mbox{along}\quad \Gamma_{\delta}.
\end{equation}

Inserting (\ref{42803}) and  (\ref{42804}) into  (\ref{42801}) and recalling the first equality in (\ref{6810}), we get
\begin{equation}\label{6201d}
\begin{aligned}
&\cos\alpha_b \bar{\partial}_{_\Gamma} u_b+\sin\alpha_b \bar{\partial}_{_\Gamma} v_b\\=~&
q_b\sin A_b \left(\bar{\partial}_{_\Gamma} \phi-\frac{\bar{\partial}_{_\Gamma} m }{\rho_bL_b}\right)-\frac{ h_{_S}(\tau_b, S_b)S_{po}'(\tau_1)\bar{\partial}_{_\Gamma} \tau_1 }{L_b} \quad \mbox{along}\quad \Gamma_{\delta}.
\end{aligned}
\end{equation}
Therefore, by (\ref{73002}), (\ref{61105}), (\ref{62011}), and (\ref{9401}) we know that when $\delta$ is sufficiently small the estimate (\ref{61111}) holds.

{\bf 2.} By (\ref{50701})--(\ref{50702}) we have
\begin{equation}\label{82006}
\cos\beta_b \bar{\partial}_{_\Gamma} u_b+\sin\beta_b \bar{\partial}_{_\Gamma} v_b~=~
\cos A_b \bar{\partial}_{_\Gamma} q_b-q_b\sin A_b \bar{\partial}_{_\Gamma} \sigma_b\quad \mbox{along}\quad \Gamma_{\delta}.
\end{equation}
Inserting (\ref{42803}) and (\ref{42804}) into (\ref{82006}), we get
\begin{equation}\label{7302}
\begin{aligned}
&\cos\beta_b \bar{\partial}_{_\Gamma} u_b+\sin\beta_b \bar{\partial}_{_\Gamma} v_b\\=~&
-q_b\sin A_b \left(\bar{\partial}_{_\Gamma} \phi-\frac{\bar{\partial}_{_\Gamma} m }{\rho_bL_b}\right)+
2\rho_bq_b\cos A_b\sin^2 A_b\tau_{po}'(\tau_1)\bar{\partial}_{_\Gamma} \tau_1\\&\qquad -\frac{\cos (2A_b) h_{_S}(\tau_b, S_b)S_{po}'(\tau_1)\bar{\partial}_{_\Gamma} \tau_1 }{L_b}\quad \mbox{along}\quad \Gamma_{\delta}.
\end{aligned}
\end{equation}
Therefore, by (\ref{73002}), (\ref{61105}), (\ref{62011}), and (\ref{9401})  we know that when $\delta$ is sufficiently small the estimate (\ref{61305}) holds.
This completes the proof.
\end{proof}


\subsubsection{\bf Singular boundary value problem}
In order to obtain the flow downstream of the post-sonic shock. We consider (\ref{PSEU}) with the boundary conditions
\begin{equation}\label{BDD}
\left\{
  \begin{array}{ll}
    (u, v, \rho, S)(x, y)=(u_b, v_b, \rho_b, S_b)(x, y), & \hbox{$y=\psi(x)$, $0<x<\delta$;} \\[4pt]
    v=u\tan\theta_w, & \hbox{$y=x\tan\theta_w$, $x>0$.}
  \end{array}
\right.
\end{equation}

\begin{figure}[htbp]
\begin{center}
\includegraphics[scale=0.55]{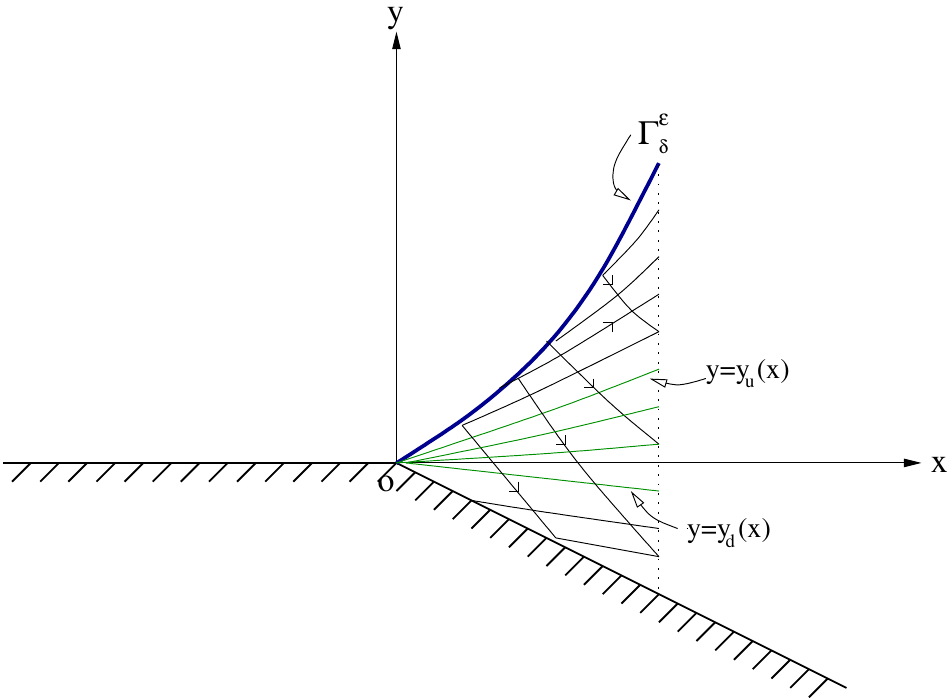}
\caption{\footnotesize A regularized problem.}
\label{Figure8}
\end{center}
\end{figure}

From (\ref{61113}) we know that the direction of the post-sonic shock $\Gamma_{\delta}$ is the same as the $C_{+}$ characteristic direction of the downstream flow field. However, by  (\ref{61402}), (\ref{62011}), (\ref{7302}), and $\bar{\partial}_{_\Gamma}S_b=S_{po}'(\tau_1)\bar{\partial}_{_\Gamma}\tau_1$ on $\Gamma_{\delta}$ we know that when $\delta$ is sufficiently small the first characteristic equation of (\ref{form}) does not hold along $\Gamma_{\delta}$. This implies that $\Gamma_{\delta}$ is not a characteristic and would result in a fact that the solution of the boundary value problem (\ref{PSEU}), (\ref{BDD}) is not regular on the boundary $\Gamma_{\delta}$.

In order to overcome the difficulty caused by the singularity, we consider a regularized problem.
For any small $\varepsilon>0$, we let the function $\psi_{\varepsilon}(x)$ be defined by
\begin{equation}
\left\{
  \begin{array}{ll}
    \psi_{\varepsilon}'(x)=\psi'(x)+\varepsilon, & \hbox{$0\leq x\leq\delta$;} \\[2pt]
    \psi_{\varepsilon}(0)=0.
  \end{array}
\right.
\end{equation}
We consider (\ref{PSEU}) with the boundary conditions
\begin{equation}\label{PBD}
\left\{
  \begin{array}{ll}
    (u, v, \rho, S)(x, y)=(u_b, v_b, \rho_b, S_b)(x, \psi(x)), & \hbox{$y=\psi_{\varepsilon}(x)$, $0<x<\delta$;} \\[4pt]
    v=u\tan\theta_w, & \hbox{$y=x\tan\theta_w$, $x>0$.}
  \end{array}
\right.
\end{equation}

For convenience, we denote by $\Gamma_{\delta}^{\varepsilon}$ the curve $y=\psi_{\varepsilon}(x)$, $0\leq x<\delta$.
Set $$\phi_{\varepsilon}(x, \psi_{\varepsilon}(x))=\arctan \psi_{\varepsilon}'(x).$$ We define the directional derivative along $\Gamma_{\delta}^{\varepsilon}$,
$$
\bar{\partial}_{\varepsilon}:=\cos\phi_{\varepsilon}\partial_x+\sin\phi_{\varepsilon}\partial_y.
$$
Then we have that for $0<x<\delta$,
\begin{equation}
\big(\bar{\partial}_{\varepsilon}u, \bar{\partial}_{\varepsilon}v, \bar{\partial}_{\varepsilon}\rho, \bar{\partial}_{\varepsilon}S\big)(x, \psi_{\varepsilon}(x))=\frac{\cos\phi_\varepsilon(x, \psi_{\varepsilon}(x))}{\cos\phi(x, \psi(x))}\big(\bar{\partial}_{_\Gamma}u_b, \bar{\partial}_{_\Gamma}v_b, \bar{\partial}_{_\Gamma}\rho_b, \bar{\partial}_{_\Gamma}S_b\big)(x, \psi(x)).
\end{equation}

By (\ref{72802}) we have
$$
\bar{\partial}_{\varepsilon}=\frac{\sin(\phi_{\varepsilon}-\beta)}{\sin 2A}\bar{\partial}_{+}+
\frac{\sin(\alpha-\phi_{\varepsilon})}{\sin 2A}\bar{\partial}_{-}\quad \mbox{on}\quad \Gamma_{\delta}^{\varepsilon}.
$$
Thus
\begin{equation}\label{61501}
\left\{
  \begin{array}{ll}
   \displaystyle\frac{\sin(\phi_{\varepsilon}-\beta)}{\sin 2A}\bar{\partial}_{+}u+
\frac{\sin(\alpha-\phi_{\varepsilon})}{\sin 2A}\bar{\partial}_{-}u=\bar{\partial}_{\varepsilon}u,\\[12pt]
 \displaystyle\frac{\sin(\phi_{\varepsilon}-\beta)}{\sin 2A}\bar{\partial}_{+}v+
\frac{\sin(\alpha-\phi_{\varepsilon})}{\sin 2A}\bar{\partial}_{-}v=\bar{\partial}_{\varepsilon}v, 
  \end{array}
\right.
\end{equation}
and
\begin{equation}\label{61502}
\frac{\sin(\phi_{\varepsilon}-\beta)}{\sin 2A}\bar{\partial}_{+}S+
\frac{\sin(\alpha-\phi_{\varepsilon})}{\sin 2A}\bar{\partial}_{-}S=\bar{\partial}_{\varepsilon}S
\end{equation}
on $\Gamma_{\delta}^{\varepsilon}$.

From (\ref{101102}) and (\ref{61502}), we have
\begin{equation}
\bar{\partial}_{+}S=-\bar{\partial}_{-}S=\frac{\sin (2A)\bar{\partial}_{\varepsilon}S}{\sin(\phi_{\varepsilon}-\beta)-\sin(\alpha-\phi_{\varepsilon})} \quad \mbox{on}\quad \Gamma_{\delta}^{\varepsilon}.
\end{equation}
By (\ref{73002}) and (\ref{9401}) we have
\begin{equation}\label{82900g}
\lim\limits_{x\rightarrow 0+}\bar{\partial}_{+}S(x, \psi_{\varepsilon}(x))= 0.
\end{equation}

From (\ref{form}) and (\ref{61501}) we have that on $\Gamma_{\delta}^{\varepsilon}$,
$$
\bar{\partial}_{+}v=\frac{\cos\beta(\cos\alpha\bar{\partial}_{\varepsilon}u+\sin\alpha\bar{\partial}_{\varepsilon}v)}
{\sin(\phi_{\varepsilon}-\beta)}
-\frac{\omega\sin 2A}{\sin(\phi_{\varepsilon}-\beta)(\tan\alpha-\tan\beta)}\left(\frac{\sin (\phi_{\varepsilon}-\beta)}{2\cos\beta}-
\frac{\sin (\alpha-\phi_{\varepsilon})}{2\cos\alpha}\right)
$$
and
$$
\bar{\partial}_{-}v=\frac{\cos\alpha(\cos\beta\bar{\partial}_{\varepsilon}u+\sin\beta\bar{\partial}_{\varepsilon}v)}
{\sin(\phi_{\varepsilon}-\alpha)}
-\frac{\omega\sin 2A}{\sin(\phi_{\varepsilon}-\alpha)(\tan\alpha-\tan\beta)}\left(\frac{\sin (\phi_{\varepsilon}-\beta)}{2\cos\beta}-
\frac{\sin (\alpha-\phi_{\varepsilon})}{2\cos\alpha}\right).
$$
Combining these with (\ref{72804}) and (\ref{81005}), we have that on $\Gamma_{\delta}^{\varepsilon}$,
\begin{equation}\label{51503}
\begin{aligned}
\bar{\partial}_{+}\rho~=~&-\frac{\cos\alpha\bar{\partial}_{\varepsilon}u+\sin\alpha\bar{\partial}_{\varepsilon}v}
{c\tau\sin(\phi_{\varepsilon}-\beta)}+\frac{\omega\sin \sigma\sin A+j_5\bar{\partial}_{+} S}{c\tau\cos\beta}
\\&\quad~+~\frac{\omega\sin 2A}{c\tau\cos\beta\sin(\phi_{\varepsilon}-\beta)(\tan\alpha-\tan\beta)}\left(\frac{\sin (\phi_{\varepsilon}-\beta)}{2\cos\beta}-
\frac{\sin (\alpha-\phi_{\varepsilon})}{2\cos\alpha}\right)
\end{aligned}
\end{equation}
and
\begin{equation}\label{51504}
\begin{aligned}
\bar{\partial}_{-}\rho~=~&\frac{\cos\beta\bar{\partial}_{\varepsilon}u+\sin\beta\bar{\partial}_{\varepsilon}v}
{c\tau\sin(\phi_{\varepsilon}-\alpha)}+\frac{\omega\sin \sigma\sin A+j_6\bar{\partial}_{+} S}{c\tau\cos\alpha}
\\&\quad~-~\frac{\omega\sin 2A}{c\tau\cos\alpha\sin(\phi_{\varepsilon}-\alpha)(\tan\alpha-\tan\beta)}\left(\frac{\sin (\phi_{\varepsilon}-\beta)}{2\cos\beta}-
\frac{\sin (\alpha-\phi_{\varepsilon})}{2\cos\alpha}\right).
\end{aligned}
\end{equation}

We define
$$
\Xi:=\Big\{(q, \tau, \sigma, S)~ \big|~ (q, \tau, \sigma, S)=(\bar{q}_{r}, \bar{\tau}_{r}, \bar{\sigma}_{r}, \bar{S}_{r})(\theta), ~\alpha_w\leq \theta\leq \phi_d\Big\}.
$$
For small $\nu>0$, we define
$$
\Xi_{\nu}:=\Big\{(q, \tau, \sigma, S)~\big|~ |(q-q', \tau-\tau', \sigma-\sigma', S-S')|\leq \nu, ~(q',\tau', \sigma',S')\in\Xi\Big\}.
$$
So, when $\nu$ is sufficiently small,
\begin{equation}\label{82900c}
p_{\tau\tau}(\tau, S)>0\quad \mbox{and} \quad 0<\frac{c(\tau, S)}{q}<1\quad  \mbox{in} \quad \Xi_{\nu},
\end{equation}
where $c(\tau, S)=\sqrt{-\tau^2p_{\tau}(\tau, S)}$.

We define the following constants:
$$
\mathcal{N}_1:=\max\limits_{(q, \tau, \sigma, S)\in\Xi_{\nu}}\frac{\tau^3\varphi p_{\tau\tau}}{4c^2\cos^2 A}, \quad \mathcal{N}_2:=\max\limits_{(q, \tau, \sigma, S)\in\Xi_{\nu}}\frac{\tau^3 p_{\tau\tau}}{4c^2\cos^2 A},
 $$
$$
\mathcal{N}_3:=\max\limits_{(q, \tau, \sigma, S)\in\Xi_{\nu}}\frac{4c^2\cos^2 A}{\tau^4 p_{\tau\tau}}, \quad \mbox{and}\quad
\mathcal{N}_4:=\max\limits_{(q, \tau, \sigma, S)\in\Xi_{\nu}}\frac{-p_{\tau\tau}}{p_{\tau}}.
$$
where $A=\arcsin(\frac{c}{q})$.

Let $n_1>\mathcal{N}_1+1$ and $n_2>\mathcal{N}_2+1$ be two large positive constants.
We then define the following constants:
$$
\begin{aligned}
&\qquad\mathcal{M}:=\max\limits_{(q, \tau, \sigma, S)\in\Xi_{\nu} }\Big|\big(\tau^{-1}\tilde{g}_{1},~\cdot\cdot\cdot, ~\tau^{-1}\tilde{g}_{6}, ~ \tau^{-1}\breve{g}_{1},~\cdot\cdot\cdot,~ \tau^{-1}\breve{g}_{6}\big)\Big|,\\
&\mathcal{C}_1:=\frac{\partial_{_\Gamma}\phi(0, \zeta_1)\cos\phi_d}{\tau_d^{1-n_1}\sin (2A_d)}, \quad \mathcal{C}_2:=\frac{\partial_{_\Gamma}\phi(0, \zeta_1)\cos\phi_d}{\tau_d^{n_2+1}\sin (2A_d)},\quad\mbox{and}   \quad \mathcal{C}_3:=\frac{\partial_{_\Gamma}\phi(0, \zeta_1)\cos\phi_d}{\tau_d^{n_2+1}\varepsilon},
\end{aligned}
$$
where $\alpha=\sigma+\arcsin(\frac{c}{q})$,  $\beta=\sigma-\arcsin(\frac{c}{q})$, and the constants $\phi_d$, $A_d$, and $\tau_d$ are defined in Section 3.5.

By (\ref{6201d}), (\ref{7302}),  (\ref{51503}), and (\ref{51504}) we know that
when $\delta$ and $\varepsilon$ are sufficiently small, where $\delta$ is independent of $\varepsilon$, the boundary data (\ref{PBD}) satisfy
\begin{equation}\label{62201}
(q, \tau, \sigma, S)\in \Xi_{\nu/2},\quad
\mathcal{Z}_{\pm}<-\frac{1}{2}\mathcal{C}_1, \quad \mathcal{R}_{+}>-2\mathcal{C}_2,\quad \mbox{and} \quad \mathcal{R}_{-}>-2\mathcal{C}_3\quad \mbox{on}\quad \Gamma_{\delta}^{\varepsilon}.
\end{equation}
where the variables $\mathcal{Z}_{\pm}$ and $\mathcal{R}_{\pm}$ are defined in (\ref{81401}) and (\ref{81402}), respectively.

Let $\mathcal{D}=\sup\limits_{\Gamma_{\delta}^{\varepsilon}}|\bar{\partial}_{+}S|$.
By (\ref{73002}) and (\ref{82900g}) we have
$$
\mathcal{D}\rightarrow 0  \quad \mbox{as}\quad \delta\rightarrow 0.
$$
Hence, when $\delta$ is sufficiently small,
\begin{equation}\label{82900e}
(\mathcal{M}+1)\mathcal{D}e^{\mathcal{N}_4(\tau_w-\tau_d+3\nu)}<\frac{1}{4}\mathcal{C}_1.
\end{equation}

\begin{lem}
There exists a small $\delta_0<\delta/2$ such that
the boundary value problem (\ref{PSEU}), (\ref{PBD}) admits a classical solution on $\Sigma^{\varepsilon}(2\delta_0):=\{(x,y)\mid x\tan\theta_w \leq y\leq\psi_{\varepsilon}(x), 0<x\leq 2\delta_0\}$.
\end{lem}
\begin{proof}
The boundary value problem (\ref{PSEU}), (\ref{PBD}) is a discontinuous boundary value problem, since the boundary data is discontinuous at the origin.
The local existence of a continuous and piecewise $C^1$ solution follows from  Li and Yu \cite{Li-Yu}.
According to the local structure of the fan-shock-fan composite wave, there is a $C_{+}$ centered wave issued from the origin with the principle part
\begin{equation}\label{82900b}
(\tilde{u}, \tilde{v}, \tilde{\tau}, \tilde{S})(0,\zeta)=(\bar{u}_r, \bar{v}_r, \bar{\tau}_r, \bar{S}_r)(\arctan\zeta), \quad \tan\alpha_w\leq \zeta\leq \tan\phi_d.
\end{equation}
The centered wave flow region is $\Sigma^{\varepsilon}_2(2\delta_0)=\{(x, y)\mid  y_{d}(x)\leq y\leq y_{u}(x), 0<x\leq 2\delta_0\}$,
where $y=y_{u}(x)$ ($0\leq x\leq 2\delta_0$) and $y_{d}(x)$ ($0\leq x\leq 2\delta_0$) are two $C_{+}$ characteristic boundaries of this centered wave, i.e., $y_u'(0)=\tan\phi_d$ and $y_d'(0)=\tan\alpha_w$.
The two characteristic curves are also the weakly discontinuity curves of the solution.
\end{proof}

\begin{lem}\label{7401}
Assume that the boundary value problem (\ref{PSEU}), (\ref{PBD}) admits a classical solution on $\Sigma^{\varepsilon}(\delta_1)$ for some $\delta_1\in (\delta_0, \delta)$. Then the solution satisfies
\begin{equation}\label{61201}
(q, \tau, \sigma, S)\in \Xi_{\nu},\quad \big|\bar{\partial}_{+}S\big|< \mathcal{D}e^{\mathcal{N}_4(\tau_w-\tau_d+3\nu)},\quad
\mathcal{Z}_{\pm}<-\frac{1}{2}\mathcal{C}_1, \quad \mbox{and}\quad \mathcal{R}_{+}>-2\mathcal{C}_2
\end{equation}
in $\Sigma^{\varepsilon}(\delta_1)$.
\end{lem}
\begin{proof}

The proof proceeds in three steps.

\noindent
{\it Step 1.}
Let
$
\Sigma_1^{\varepsilon}(\delta_1)=\{(x, y)\mid y_{u}(x)\leq y< \psi_{\varepsilon}(x), 0<x\leq \delta_1\}
$.
For any point $\mathrm{Y}\in \Sigma^{\varepsilon}_1(\delta_1)$, the backward $C_{\pm}$ characteristic curves issued from $\mathrm{Y}$ stay in $\Sigma^{\varepsilon}_{1}(\delta_1)$ until they intersect $\Gamma_{\delta}^{\varepsilon}$ at some points $\mathrm{Y}_{\pm}$.
Let $\Sigma_{\mathrm{Y}}$ be a closed domain bounded by $\wideparen{\mathrm{Y_{+}Y}}$,  $\wideparen{\mathrm{Y_{-}Y}}$, and $\wideparen{\mathrm{Y_{+}Y_{-}}}$, where $\wideparen{\mathrm{Y_{\pm}Y}}$ are the $C_{\pm}$ characteristic curves connecting the points $\mathrm{Y_{\pm}}$ and $\mathrm{Y}$, and $\wideparen{\mathrm{Y_{+}Y_{-}}}$ is a portion of $\Gamma_{\delta}^{\varepsilon}$; see Figure \ref{Fig9} (left).  The backward $C_0$ characteristic curve issued from $\mathrm{Y}$ stays in $\Sigma_{\mathrm{Y}}$ until it intersects $\wideparen{\mathrm{Y_{+}Y_{-}}}$  at a point $\mathrm{Y}_0$. We assert that if the inequalities in (\ref{61201}) hold for every point in $\Sigma_{\mathrm{Y}}\setminus \{\mathrm{Y}\}$ then they also hold at the point $\mathrm{Y}$.

\begin{figure}[htbp]
\begin{center}
\includegraphics[scale=0.65]{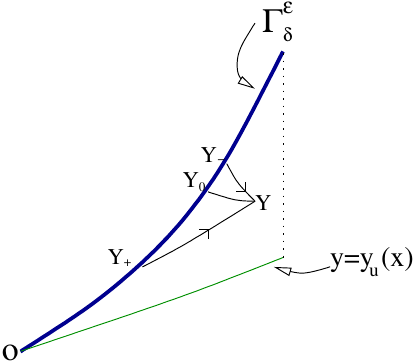} \qquad \includegraphics[scale=0.73]{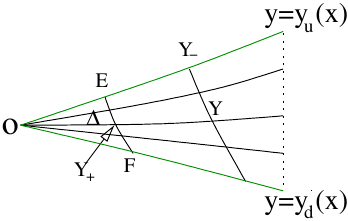} \qquad \includegraphics[scale=0.73]{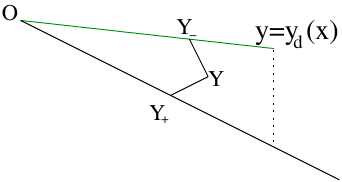}
\caption{\footnotesize Domains $\Sigma^{\varepsilon}_1$, $\Sigma^{\varepsilon}_2$, and $\Sigma^{\varepsilon}_3$.}
\label{Fig9}
\end{center}
\end{figure}

 Integrating (\ref{7403}) along the straight line $\wideparen{\mathrm{Y_0Y}}$ from $\mathrm{Y}_0$ to $\mathrm{Y}$ we have
\begin{equation}\label{82900f}
\big|(\bar{\partial}_{+}S)(\mathrm{Y})\big|< \mathcal{D}e^{\mathcal{N}_4(\tau_w-\tau_d+3\nu)}.
\end{equation}

Using the characteristic equations (\ref{1})--(\ref{81005}) and recalling (\ref{62201}) and the assumption of the assertion, we have $(q, \tau, \sigma, S)(\mathrm{Y})\in \Xi_{2\nu/3}$ as $\delta$ is sufficiently small, where $\delta$ is independent of $\nu$ and $\mathrm{Y}$.


Suppose $\mathcal{Z}_{+}=-\frac{1}{2}\mathcal{C}_1$ and $\mathcal{Z}_{-}\leq-\frac{1}{2}\mathcal{C}_1$ at $\mathrm{Y}$.
Then by the assumption of the assertion we have $\bar{\partial}_{-}\mathcal{Z}_{+}\geq 0$.
While, by the second equation of (\ref{cd2}), (\ref{82900e}), and (\ref{82900f}) we have
\begin{equation}\label{8300a}
\begin{aligned}
&\tau^{n_1-1}\bar{\partial}_{-}\mathcal{Z}_{+}\\~=~&
   \displaystyle  \frac{\tau^3p_{\tau\tau}}{4c^2\cos^{2}A }\Big[\mathcal{Z}_{+}^2+(\varphi-1)\mathcal{Z}_{-}\mathcal{Z}_{+}\Big]-n_1\mathcal{Z}_{-}\mathcal{Z}_{+}
+\tau^{-1}\Big(\tilde{g}_4\mathcal{Z}_{-} +\tilde{g}_5\mathcal{Z}_{+}\Big)\bar{\partial}_{+} S+\tau^{-1}\tilde{g}_6(\bar{\partial}_{+}
S)^2\\~<~&\displaystyle  \frac{\tau^3\varphi p_{\tau\tau}}{4c^2\cos^{2}A }\mathcal{Z}_{-}\mathcal{Z}_{+}-n_1\mathcal{Z}_{-}\mathcal{Z}_{+}
+\tau^{-1}\Big(\tilde{g}_4\mathcal{Z}_{-} +\tilde{g}_5\mathcal{Z}_{+}\Big)\bar{\partial}_{+} S+\tau^{-1}\tilde{g}_6(\bar{\partial}_{+} S)^2
\\~<~&-\mathcal{Z}_{-}\mathcal{Z}_{+}-\mathcal{M}\Big(\mathcal{Z}_{-} +\mathcal{Z}_{+}\Big)|\bar{\partial}_{+} S|+
\mathcal{M}|\bar{\partial}_{+} S|^2\\~=~&-\frac{\mathcal{Z}_{-}}{3}\Big(\mathcal{Z}_{+}+\mathcal{M}|\bar{\partial}_{+} S|\Big)-\frac{\mathcal{Z}_{+}}{3}\Big(\mathcal{Z}_{-}+\mathcal{M}|\bar{\partial}_{+} S|\Big)+\Big(\mathcal{M}|\bar{\partial}_{+} S|^2-\frac{\mathcal{Z}_{+}\mathcal{Z}_{-}}{3}\Big)<0\quad \mbox{at}\quad \mathrm{Y}.
\end{aligned}
\end{equation}
 This leads to a contradiction. So, we get $\mathcal{Z}_{+}(\mathrm{Y})<-\frac{1}{2}\mathcal{C}_1$. Similarly, we have $\mathcal{Z}_{-}(\mathrm{Y})<-\frac{1}{2}\mathcal{C}_1$.

Suppose $\mathcal{R}_{+}(\mathrm{Y})=-2\mathcal{C}_2$. Then by the assumption of the assertion we have $\bar{\partial}_{-}\mathcal{R}_{+}\leq 0$ at $\mathrm{Y}$.
While, by the second equation of (\ref{cd3}) we have
\begin{equation}\label{8300b}
\begin{aligned}
&\tau^{n_2-1}\bar{\partial}_{-}\mathcal{R}_{+}\\~=~&
   \displaystyle  \frac{\tau^3p_{\tau\tau}}{4c^2\cos^{2}A }\Big[\mathcal{Z}_{+}^2+(\varphi-1)\mathcal{Z}_{-}\mathcal{Z}_{+}\Big]+n_2\mathcal{Z}_{-}\mathcal{Z}_{+}
+\tau^{-1}\Big(\breve{g}_4\mathcal{Z}_{-} +\breve{g}_5\mathcal{Z}_{+}\Big)\bar{\partial}_{+} S+\tau^{-1}\breve{g}_6(\bar{\partial}_{+}
S)^2\\~>~&-\displaystyle  \frac{\tau^3p_{\tau\tau}}{4c^2\cos^{2}A }\mathcal{Z}_{-}\mathcal{Z}_{+}+n_2\mathcal{Z}_{-}\mathcal{Z}_{+}
+\tau^{-1}\Big(\breve{g}_4\mathcal{Z}_{-} +\breve{g}_5\mathcal{Z}_{+}\Big)\bar{\partial}_{+} S+\tau^{-1}\breve{g}_6(\bar{\partial}_{+} S)^2
\\~>~&\mathcal{Z}_{-}\mathcal{Z}_{+}+\mathcal{M}\Big(\mathcal{Z}_{-} +\mathcal{Z}_{+}\Big)|\bar{\partial}_{+} S|-
\mathcal{M}|\bar{\partial}_{+} S|^2\\~=~&\frac{\mathcal{Z}_{-}}{3}\Big(\mathcal{Z}_{+}+\mathcal{M}|\bar{\partial}_{+} S|\Big)+\frac{\mathcal{Z}_{+}}{3}\Big(\mathcal{Z}_{-}+\mathcal{M}|\bar{\partial}_{+} S|\Big)-\Big(\mathcal{M}|\bar{\partial}_{+} S|^2-\frac{\mathcal{Z}_{+}\mathcal{Z}_{-}}{3}\Big)>0\quad \mbox{at}\quad \mathrm{Y}.
\end{aligned}
\end{equation}
This leads to a contraction. So, we have  $\mathcal{R}_{+}(\mathrm{Y})>-2\mathcal{C}_2$. This completes the proof of the assertion.
Therefore, by an argument of continuity we obtain that the inequalities in (\ref{61201}) hold for every point in $\Sigma^{\varepsilon}_1(\delta_1)$.

\vskip 4pt
\noindent
{\it Step 2.}
Let $\mathrm{E}$ be a point on the $C_{+}$ characteristic curve $y=y_u(x)$. When $\mathrm{E}$ is close to the origin, the forward $C_{-}$ characteristic curve issued from $\mathrm{E}$ intersects the $C_{+}$ characteristic curve $y=y_d(x)$ at a point $\mathrm{F}$. We denote by $\Delta$ the triangle domain bounded by the characteristic curves $\wideparen{\mathrm{OE}}$, $\wideparen{\mathrm{OF}}$, and $\wideparen{\mathrm{EF}}$; see Figure \ref{Fig9} (mid). By (\ref{61905}) we know that when $\mathrm{E}$ is close to the origin
$(q, \tau, \sigma, S)\in \Xi_{2\nu/3}$ in $\Delta$. As in the last step we have
$|\bar{\partial}_{+}S|\leq \mathcal{D}e^{\mathcal{N}_4(\tau_w-\tau_d+3\nu)}$ in $\Delta$.

From (\ref{61906}) and (\ref{102304}) we have
\begin{equation}
\tilde{\rho}_{\zeta}(0, \zeta)=\frac{\bar{\rho}_{r}'(\arctan\zeta)}{1+\zeta^2}>0\quad \mbox{for}\quad \tan\alpha_w\leq \zeta\leq \tan\phi_d.
\end{equation}
Then by (\ref{82900b}) and (\ref{61905}) and the above estimates in $\Delta$ we have that when $\mathrm{E}$ is close to the origin, $\mathcal{Z}_{-}<-\mathcal{C}_1$ in $\Delta$.

In the last step we have obtained $\mathcal{Z}_{+}<-\frac{1}{2}\mathcal{C}_1$ and $\mathcal{R}_{+}>-2\mathcal{C}_2$ on $y=y_{u}(x)$ ($0<x\leq \delta_1$).
 As in (\ref{8300a}) and (\ref{8300b}) we have  $\mathcal{Z}_{+}<-\frac{1}{2}\mathcal{C}_1$ and $\mathcal{R}_{+}>-2\mathcal{C}_2$ in $\Delta$.

For any point $\mathrm{Y}\in \Sigma^{\varepsilon}_2(\delta_1)\setminus\Delta$,
the backward $C_{+}$ and $C_{-}$ characteristic curves issued from $\mathrm{Y}$ intersect the $C_{+}$ characteristic curve $y=y_u(x)$ and the $C_{-}$ characteristic curve $\wideparen{\mathrm{EF}}$ at some points $\mathrm{Y}_{-}$ and $\mathrm{Y}_{+}$. Let $\Sigma_{\mathrm{Y}}$ be a closed domain bounded by $\wideparen{\mathrm{EY}_{-}}$, $\wideparen{\mathrm{EY}_{+}}$, $\wideparen{\mathrm{Y_{-}Y}}$, and $\wideparen{\mathrm{Y_{+}Y}}$. As in the last step, we know that if the inequalities in (\ref{61201}) hold for every point in $\Sigma_{\mathrm{Y}}\setminus \{\mathrm{Y}\}$ then they also hold at the point $\mathrm{Y}$.
Therefore, by an argument of continuity we know that (\ref{61201}) hold in $\Sigma^{\varepsilon}_2(\delta_1)\setminus\Delta$.


\vskip 4pt
\noindent
{\it Step 3.}
Let $
\Sigma^{\varepsilon}_3(\delta_1)=\{(x, y)\mid x\tan\theta_w\leq x\leq y_{d}(x), 0\leq x\leq \delta_1\}
$.
We are going to prove that the solution satisfies
\begin{equation}\label{7402}
(q, \tau, \sigma, S)\in \Xi_{\nu},\quad
\mathcal{Z}_{\pm}<-\frac{1}{2}\mathcal{C}_1, \quad \mbox{and}\quad \mathcal{R}_{\pm}>-2\mathcal{C}_2
\end{equation}
in $\Sigma^{\varepsilon}_3(\delta_1)$.

 For any point $\mathrm{Y}\in \Sigma^{\varepsilon}_3(\delta_1)$, the backward $C_{+}$ characteristic curve issued from $\mathrm{Y}$ stays in $\Sigma^{\varepsilon}_{3}(\delta_1)$ until
it intersects the ramp at a point $\mathrm{Y}_{+}$, and
 the backward $C_{-}$ characteristic curve issued from $\mathrm{Y}$ stays in $\Sigma^{\varepsilon}_{3}(\delta_1)$ until
it intersects $y=y_{d}(x)$ ($0<x\leq \delta_1$)  at a point $Y_{-}$; see Figure \ref{Fig9} (right).
We denote by $\Sigma_{\mathrm{Y}}$ the closed domain bounded by $\wideparen{\mathrm{Y_{+}Y}}$, $\wideparen{\mathrm{Y_{-}Y}}$, $\wideparen{\mathrm{OY}_{-}}$, and $\overline{\mathrm{OY}_{+}}$.

(i) If $\mathrm{Y}$ does not lie on the ramp, then $\wideparen{\mathrm{Y_{\pm}Y}}$ exist. As in the first step, if the inequalities in (\ref{7402}) hold for every point in $\Sigma_{\mathrm{Y}}\setminus\{\mathrm{Y}\}$, then they also hold at $\mathrm{Y}$.

(ii) If $\mathrm{Y}$  lies on the ramp, then  $\wideparen{\mathrm{Y_{-}Y}}$ exists. As in the first step,  if the inequalities in (\ref{7402}) hold for every point in $\Sigma_{\mathrm{Y}}\setminus\{\mathrm{Y}\}$, then $(q, \tau, \sigma, S)(Y)\in \Xi_{\nu}$, $
\mathcal{Z}_{+}(\mathrm{Y})<-\frac{1}{2}\mathcal{C}_1$, and $\mathcal{R}_{+}(\mathrm{Y})>-2\mathcal{C}_2$.

According to the slip boundary condition, we have
$$
\bar{\partial}_{0}\sigma=0
\quad \mbox{on}\quad y=x\tan\theta_w, \quad x\geq 0.
$$
Moreover, by (\ref{81301}) and  (\ref{82900g}) we have
\begin{equation}\label{6202d}
\bar{\partial}_{+}S=0
\quad \mbox{on}\quad y=x\tan\theta_w, \quad x\geq 0.
\end{equation}
Then by $\bar{\partial}_{0}\sigma=\frac{1}{4}(\bar{\partial}_{+}+\bar{\partial}_{-})(\alpha+\beta)$ and (\ref{1})--(\ref{2}) we have
\begin{equation}\label{61204}
\bar{\partial}_{-}\rho=\bar{\partial}_{+}\rho\quad \mbox{on}\quad y=x\tan\theta_w, \quad x>0.
\end{equation}
 Thus, $\mathcal{Z}_{-}(\mathrm{Y})=\mathcal{Z}_{+}(\mathrm{Y})<-\frac{1}{2}\mathcal{C}_1$ and $\mathcal{R}_{-}(\mathrm{Y})=\mathcal{R}_{+}(\mathrm{Y})>-2\mathcal{C}_2$.
Therefore, by an argument of continuity we know that the solution satisfies (\ref{7402}) in $\Sigma^{\varepsilon}_3(\delta_1)$.

This completes the proof of Lemma \ref{7401}.
\end{proof}

We define the following constants:
$$
\mathcal{C}_4:=\max\limits_{y\in [\delta_0\tan\theta_w, \psi_{\varepsilon}(\delta_0)]}\big|\mathcal{R}_{-}(\delta_0, y)\big|\quad \mbox{and}\quad   \mathcal{C}_5:=\max\big\{\mathcal{C}_2,~\mathcal{C}_3,~\mathcal{C}_4\big\}.
$$

\begin{lem}\label{61205}
Suppose that the boundary value problem (\ref{PSEU}), (\ref{PBD}) admits a classical solution on $\Sigma^{\varepsilon}(\delta_1)$ for some $\delta_1\in (\delta_0, \delta)$. Then, in $\Sigma^{\varepsilon}(\delta_1)\setminus \Sigma^{\varepsilon}(\delta_0)$ the solution satisfies
\begin{equation}\label{52705}
 \mathcal{R}_{-}>-2\mathcal{C}_5.
\end{equation}
\end{lem}
\begin{proof}
Let ${\it l}_1=\{(x, y)\mid y=\psi_{\varepsilon}(x), \delta_0<x\leq\delta_1\}$,
${\it l}_2=\{(x, y)\mid x=\delta_0, \delta_0\tan\theta_w\leq y\leq  \psi_{\varepsilon}(\delta_0)\}$, and
${\it l}_3=\{(x, y)\mid y=x\tan\theta_w, \delta_0<x\leq \delta_1\}$.
By the definition of $\mathcal{C}_5$, (\ref{6202d}), and (\ref{61204}) we have $\mathcal{R}_{-}>-2\mathcal{C}_5$ on ${\it l}_1\cup{\it l}_2\cup{\it l}_3$.

For any point $\mathrm{Y}\in \Sigma^{\varepsilon}(\delta_1)\setminus \Sigma^{\varepsilon}(\delta_0)$, the backward $C_{+}$ characteristic curve issued from $\mathrm{Y}$ stays in $\Sigma^{\varepsilon}(\delta_1)\setminus \Sigma^{\varepsilon}(\delta_0)$ until it intersects ${\it l}_1\cup{\it l}_2\cup{\it l}_3$ at a point $\mathrm{Y}_{+}$.
We assert that if $\mathcal{R}_{-}>-2\mathcal{C}_5$ for every point in $\wideparen{\mathrm{\mathrm{Y_{+}Y}}}\setminus\{\mathrm{Y}\}$, then  $\mathcal{R}_{-}(\mathrm{Y})>-2\mathcal{C}_5$.
Suppose $\mathcal{R}_{-}(\mathrm{Y})=-2\mathcal{C}_5$.
Then by the assumption of the assertion we have $\bar{\partial}_{+}\mathcal{R}_{-}\leq 0$ at $\mathrm{Y}$.
While,  as in (\ref{8300b}) we have $\bar{\partial}_{+}\mathcal{R}_{-}>0$ at $\mathrm{Y}$. This leads to a contradiction.
So, we have $\mathcal{R}_{-}(\mathrm{Y})>-2\mathcal{C}_5$. This completes the proof of the assertion.

Therefore,  by the argument of continuity we get $\mathcal{R}_{-}>-2\mathcal{C}_5$ for every point in $\Sigma^{\varepsilon}(\delta_1)\setminus \Sigma^{\varepsilon}(\delta_0)$. This completes the proof of the Lemma.
\end{proof}

By (\ref{72802}) and Lemmas \ref{7401} and \ref{61205}, we obtain a uniform a priori $C^1$ norm estimate for the solution in $\Sigma^{\varepsilon}(\delta_1)\setminus \Sigma^{\varepsilon}(\delta_0)$ for $\delta_0<\delta_1<\delta$.  Then  the local solution can be extended to $\Sigma^{\varepsilon}(\delta)$. So, there is a fixed $\delta>0$ such that for any small $\varepsilon>0$, the problem (\ref{PSEU}), (\ref{PBD}) admits a classical solution on $\Sigma^{\varepsilon}(\delta)$ and the solution satisfies (\ref{61201}) in $\Sigma^{\varepsilon}(\delta)$.

For small $r>0$, we
define
$$
\Sigma_{r}(\delta):=\big\{(x, y)\mid
x\tan\theta_w\leq y\leq\psi(x)-r, ~\varsigma\leq x\leq \delta\big\},
$$
where $\varsigma$ is defined by $\psi(\varsigma)-r=\varsigma\tan\theta_w$.
We also define $$\Gamma_{\delta, r}:=\{(x, y)\mid y=\psi(x)-r, ~\varsigma\leq x\leq \delta\}.$$

\begin{lem}\label{61210}
For any small $r>0$, there exists a sufficiently small $\varepsilon_0>0$ such that when $\varepsilon<\varepsilon_0$ the solution of the problem (\ref{PSEU}), (\ref{PBD}) satisfies
\begin{equation}
-\frac{2\mathcal{N}_3}{r(\tau_d-\nu)^{n_2}}<\mathcal{R}_{-}<0\quad \mbox{in}\quad \Sigma_{r}(\delta).
\end{equation}
\end{lem}
\begin{proof}
By the result in the third step of the proof for Lemma \ref{7401} we know that when $r>0$ is sufficiently small,
$$
\mathcal{R}_{-}>-2\mathcal{C}_2>-\frac{2\mathcal{N}_3}{r(\tau_d-\nu)^{n_2}}\quad \mbox{in}\quad \Sigma^\varepsilon_3(\delta).
$$
So, it only remains to prove the estimate in $\Sigma_{r}(\delta)\setminus\Sigma_3^\varepsilon(\delta)$.

For any point in $\Gamma_{\delta, r}\setminus\Sigma_3^\varepsilon(\delta)$ the backward $C_{+}$ characteristic curve issued from this point intersects $\Gamma_{\delta}^{\varepsilon}$ or goes back to the origin. By
(\ref{cd4}), (\ref{51504}), and $(\frac{1}{\mathcal{W}_{-}})(\mathrm{O})=0$ we have that when $\varepsilon$ is sufficiently small,
$$
\frac{1}{\mathcal{W}_{-}}<-\frac{r}{2\mathcal{N}_3}\quad \mbox{on}\quad \Gamma_{\delta, r}.
$$
Furthermore, by $\tau>\tau_d-\nu$ we have
\begin{equation}\label{8300c}
\mathcal{R}_{-}>-\frac{2\mathcal{N}_3}{r(\tau_d-\nu)^{n_2}}\quad \mbox{on}\quad \Gamma_{\delta, r}.
\end{equation}

Using (\ref{61201}) and (\ref{8300b}) we have
\begin{equation}\label{8300d}
\bar{\partial}_{+}\mathcal{R}_{-}>0\quad \mbox{in} \quad \Sigma^\varepsilon(\delta).
\end{equation}
For any point in $\Sigma_{r}(\delta)\setminus\Sigma_3^\varepsilon(\delta)$, the backward $C_{+}$ characteristic issued from this point stays in
$\Sigma_{r}(\delta)\setminus\Sigma_3^\varepsilon(\delta)$ until it intersects $\Gamma_{\delta, r}\setminus\Sigma_3^\varepsilon(\delta)$.
Then by (\ref{8300c}) and (\ref{8300d}) we have
$$
\mathcal{R}_{-}>-\frac{2\mathcal{N}_3}{r(\tau_d-\nu)^{n_2}}\quad \mbox{in}\quad \Sigma_{r}(\delta)\setminus\Sigma_3^\varepsilon(\delta).
$$
This completes the proof.
\end{proof}

For convenience, we use $(u, v, \tau, S)=(u_{\varepsilon}, v_{\varepsilon}, \tau_{\varepsilon}, S_{\varepsilon})(x, y)$, $(x, y)\in \Sigma^{\varepsilon}(\delta)$ to denote the solution of the regularized problem (\ref{PSEU}), (\ref{PBD}).
Lemmas \ref{7401} and \ref{61210} provide a uniform interior $C^1$ norm estimate to the solutions of the regularized problem (\ref{PSEU}), (\ref{PBD}). Let  $\Sigma(\delta)=\{(x, y)\mid x\tan\theta_w\leq y<\psi(x), 0<x<\delta\}$.
So, by the classical Arzela-Ascoli theorem and a standard diagonal procedure,
there exists a sequence $\{\varepsilon_n\}_{n=1}^{\infty}$ and a function $(u_2, v_2, \tau_2, S_2)(x, y)\in C^{0, 1}(\Sigma(\delta))$ such that for any small $r>0$,
$$
(u_{\varepsilon_n}, v_{\varepsilon_n}, \tau_{\varepsilon_n}, S_{\varepsilon_n})\rightarrow (u_2, v_2, \tau_2, S_2)(x, t)\quad \mbox{in}\quad C^{0, 1}(\Sigma_{r}(\delta))\quad \mbox{as}\quad n\rightarrow +\infty.
$$
So, the function $(u, v, \tau, S)=(u_2, v_2, \tau_2, S_2)(x, y)$ satisfies
(\ref{PSEU}) and (\ref{61201}) a.e. in $\Sigma(\delta)$.
Moreover, by (\ref{8300c}) and $(q, \tau, \sigma, S)\in \Xi_{\nu}$ we also have $(u_2, v_2, \tau_2, S_2)(x, y)\in C(\overline{\Sigma(\delta)}\setminus\{(0,0)\})$ and
$$
(u_2, v_2, \tau_2, S_2)=(u_b, v_b, \tau_b, S_b)\quad \mbox{on}\quad \Gamma_{\delta}.
$$
This implies that $(u, v, \tau, S)=(u_2, v_2, \tau_2, S_2)(x, y)$ is a Lipschitz continuous solution of the singular boundary value problem (\ref{PSEU}), (\ref{BDD}) on $\Sigma(\delta)$.


Let
\begin{equation}\label{81406}
(u, v, \tau, S)=\left\{
                  \begin{array}{ll}
                    (u_1, v_1, \tau_1, S_1)(x, y), & \hbox{$\psi(x)<y<y_2(\delta)$, $0<x<\delta$;} \\[2pt]
                    (u_2, v_2, \tau_2, S_2)(x, y), & \hbox{$x\tan\theta_w<y<\psi(x)$, $0<x<\delta$.}
                  \end{array}
                \right.
\end{equation}
Then the function $(u, v, \tau, S)$ defined in (\ref{81406}) is a solution of the problem (\ref{PSEU}), (\ref{RBD1}).
This completes the proof of Theorem \ref{main}.

\vskip 32pt
\section*{Acknowledgement}
This work was partially supported by National Natural Science Foundation of China (No. 12071278) and Natural Science Foundation of Shanghai (No. 23ZR1422100).

\vskip 32pt
\noindent
{\bf  Declarations}
\vskip 4pt
\noindent
{\bf Conflicts of interests} ~This work does not have any conflicts of interest.
\vskip 4pt
\noindent
{\bf  Data availability}
~No data was used for the research described in the article.

\vskip 32pt


\small
\begin{thebibliography}{aa}


\bibitem{Bethe} {\sc H. A. Bethe},
{\em The theory of shock waves for an arbitatry equation of state}, Tech. Paper 45, Office of Scientific Research and Development, 1942.

 \bibitem{BBKN} {\sc A. A. Borisov, Al. A. Borisov, S. S. Kutateladze and V. E. Nakoryakov}, {\em Rarefaction shock wave near the critical liquid vapor point},
J. Fluid Mech., 126 (1983) 59--73.

\bibitem{ZH} {T. Chang and L. Hsiao}, {\em The Riemann Problem and Interaction of Waves in Gas Dynamics}, Longman, 1989.

\bibitem{CCX} {\sc G. Q. Chen, J. Chen, and W.  Xiang}, {\em Stability of attached transonic shocks in steady potential flow past three-dimensional wedges}, Comm. Math. Phys. 387 (2021) 111--138.

\bibitem{CFf} {\sc G. Q. Chen and B. X. Fang}, {\em Stability of transonic shocks in steady supersonic flow past multidimensional wedges}, Adv. Math. 314 (2017) 493--539.

\bibitem{CKXZ} {\sc G. Q. Chen, J. Kuang, W. Xiang, and Y. Zhang}, {\em Hypersonic similarity for steady compressible full Euler flows over two-dimensional Lipschitz wedges}, Adv. Math. 451 (2024), Paper No. 109782.


\bibitem{CKZ} {\sc G. Q. Chen, J. Kuang, and Y. Zhang}, {\em Stability of conical shocks in the three-dimensional steady supersonic isothermal flows past Lipschitz perturbed cones}, SIAM J. Math. Anal. 53 (2021) 2811--2862.

\bibitem{CZZ} {\sc G. Q. Chen, Y. Zhang, and D. Zhu}, {\em Existence and stability of supersionic euler flows past Lipschitz wedges}, Arch. Ration. Mech. Anal., 161 (2006) 261--310.

\bibitem{CLD1}	{\sc S. X. Chen and D. N. Li}, {\em Supersonic flow past a symmetrically curved cone}, Indiana Univ. Math. J. 49 (2000)  1411--1435.

    \bibitem{CLD2}	{\sc S. X. Chen and D. N. Li}, {\em Conical shock waves in supersonic flow}, J. Differential Equations 269 (2020) 595--611.


\bibitem{CXY}	{\sc S. X. Chen, Z. P. Xin, and H. C. Yin}, {\em Global shock waves for the supersonic flow past a perturbed cone}, Comm. Math. Phys. 228 (2002) 47--84.


\bibitem{Colombo}  {\sc R. M. Colombo and A. Corli}, {\em Sonic hyperbolic phase transitions and Chapman-Jouguet detonations}, J. Differential Equations 184 (2002) 321–347.

\bibitem{Corli} {\sc A. Corli and M. Sabl\'{e}-Tougeron}, {\em Kinetic stabilization of a nonlinear sonic phase
boundary}, Arch. Ration. Mech. Anal. 152 (2000) 1--63.

\bibitem{CaF} {\sc R. Courant and K. O. Friedrichs}, {\em Supersonic flow and shock
waves}, Interscience, New York, 1948.


 \bibitem{Cra} {\sc M. S. Cramer and R. Sen}, {\em Exact solutions for sonic shocks in van der Waals gases},
 Phys. Fluids 30 (1987) 377--385.



\bibitem{CY1} {\sc D. C. Cui and H. C. Yin}, {\em Global conic shock wave for the steady supersonic flow past a cone: isothermal case}, Pac. J. Math. 233 (2007) 257--289.

\bibitem{CY2} {\sc D. C. Cui and H. C. Yin}, {\em Global supersonic conic shock wave for the steady supersonic flow past a cone: polytropic gas}, J. Differential Equations 246 (2009)  641--669.



\bibitem{Fossati} {\sc M. Fossati and L. Quartapelle}, {\em The Rimann problem for hyperbolic equations under a nonconvex flux with two inflection points}, arXiv: 1402.5906.

 \bibitem{Gu} {\sc C. H. Gu}, {\em A method for solving the supersonic flow past a curved wedge}, Fudan J. (Nat. Sci.) 7 (1962) 11--14.


\bibitem{HD} {\sc D. Hu}, {\em The supersonic flow past a wedge with large curved boundary}, J. Math. Anal. Appl. 462 (2018) 380--389.


\bibitem{HQ} {\sc D. Hu and A. F. Qu}, {\em Hypersonic flow onto a large curved wedge and the dissipation of shock wave}, J. Lond. Math. Soc. (2) 111 (2025) Paper No. e70182.

\bibitem{HZ} {\sc D. Hu and Y. Q. Zhang}, {\em Global conic shock wave for the steady supersonic flow past a curved cone}, SIAM J. Math. Anal. 51 (2019) 2372--2389.

  \bibitem{Lai1}  {\sc G. Lai}, {\em Interaction of jump-fan composite waves in a two-dimensional jet for van der Waals gases}, J. Math. Phys. 56 (2015), no. 6, 061504.

 \bibitem{Lai2}  {\sc G. Lai}, {\em Interactions of composite waves of the two-dimensional full Euler equations for van der Waals gases}, SIAM J. Math. Anal. 50 (2018) 3535--3597.

\bibitem{LeFloch1} {\sc P. G. LeFloch}, {\em Hyperbolic systems of Conservation Laws: the theory of classical and nonclassical shock waves},
 Lectures in Mathematics--ETH Z\"{u}rich, Birkh\"{a}user Verlag, Basel, 2002.

 \bibitem{LeFloch2} {\sc P. G. LeFloch and M. D. Thanh}, {\em The Riemann problem in continuum physics}, Applied Mathematical Sciences 219, Springer, 2023.



\bibitem{LWY1}	{\sc J. Li, I. Witt, and H. C. Yin}, {\em On the global existence and stability of a multi-dimensional supersonic conic shock wave}, Comm. Math. Phys. 329 (2014) 609--640.








 \bibitem{Li-Zhang-Zheng} {\sc J. Q. Li, T. Zhang and Y. X. Zheng}, {\em Simple waves and a characteristic decomposition of the
 two dimensional compressible Euler equations},
Comm. Math. Phys., 267 (2006) 1--12.



\bibitem{Li-Yu} {\sc T. T. Li and W. C. Yu},
{\em Boundary value problem for quasilinear hyperbolic systems},
Duke University, 1985.

\bibitem{Li-Sheng1} {\sc S. R. Li and W. C. Sheng}, {\em On the composite waves of the two-dimensional pseudo-steady van der Waals gas satisfied Maxwell's law},  Appl. Math. Lett. 153 (2024), Paper No. 109066.


\bibitem{Li-Sheng2} {\sc S. R. Li and W. C. Sheng}, {\em Centered waves for the two-dimensional pseudo-steady van der Waals gas satisfied Maxwell's law around a sharp corner}, Chinese Ann. Math. Ser. B 45 (2024) 537--554.

 \bibitem{LL} {\sc W. C. Lien and T. P. Liu}, {\em Nonlinear stability of a self-similar 3-d gas flow}, Comm. Math. Phys., 304 (1999) 524--549.

\bibitem{Liu1} {\sc T. P. Liu},
{\em  The Riemann problem for $2\times 2$ conservation laws},
Trans. Amer. Math. Soc.  199 (1974) 89--112.

\bibitem{Liu2} {\sc T. P. Liu},
{\em  The Riemann problem for general systems of conservation laws},
J. Differential Equations 18 (1975) 218--234.


\bibitem {MP} {\sc R. Menikoff and B. J. Plohr},  {\em The Riemann problem for fluid flow of real materials}, Rew. Mod. Physics 61 (1989) 75--130.

     \bibitem{MV} {\sc S. M\"{u}ler and A. Vo{\ss}}, {\em The Riemann problem for the Euler equation with nonconvex and nonsmooth equation of state: construction of wave curves},
SIAM J. Sci. Comput., 28 (2006) 651--681.

\bibitem{NSMGC} {\sc  N. R. Nannan, C.  Sirianni, T. Mathijssen, A. Guardone and P. Colonna}, {\em The admissibility domain of rarefaction shock waves in the near-cirtical vapour-liquid equilibrium region of pure typical fluids}, J. Fluid Mech. 795 (2016) 241--261.




\bibitem{Sch} {\sc D. G. Schaeffer}, {\em Supersonic flow past a nearly straight wedge}, Duke Math. J. 43 (1976), 637--670.



\bibitem{Thompson1} {\sc P. A. Thompson}, {\em A fundamental derivative in gasdynamics}, Phys. Fluids 14 (1971) 1843--1849.


\bibitem{Thompson2} {\sc P. A. Thompson and K. C. Lambrakis}, {\em Negatie shock waves}, J. Fluid. Mech. 60 (1973) 187--208.


\bibitem{VKG1} {\sc  D. Vimercati, A. Kluwick and A. Guardone}, {\em Oblique waves in steady supersonic flows of Bethe-Zel'dovich-Thompson fluids}, J. Fluid Mech. 885 (2018) 445--468.

    \bibitem{VKG2} {\sc  D. Vimercati, A. Kluwick and A. Guardone}, {\em Shock interactions in two-dimensional steady flows of Bethe-Zel'dovich-Thompson fluids}, J. Fluid Mech. 887 (2020) A12.

\bibitem{WZ}{\sc Z. Wang and Y. Zhang}, {\em Stability supersonic flow past a curve zone}, J. Differential Equations 247 (2009) 1817--1850.



\bibitem{Wen1} {\sc  B. Wendroff}, {\em The Riemann problem for materials with nonconvex equation of state, I: Isentropic flow}, J. Math. Anal. Appl. 38 (1972) 454--466.

\bibitem{Wen2} {\sc  B. Wendroff}, {\em The Riemann problem for materials with nonconvex equation of state, II: General flow}, J. Math. Anal. Appl. 38 (1972) 640--658.

\bibitem{Weyl} {\sc H. Weyl}, {\em Shock waves in arbitrary fluids}, Comm. Pure Appl. Math. 2 (1949) 103--122.

\bibitem{XY1}{\sc G. Xu and H. C. Yin}, {\em Global transonic conic shock wave for the symmetrically perturbed supersonic flow past
a cone}, J. Differential Equations 245 (2008) 3389--3432.

\bibitem{XY2}{\sc G. Xu and H. C. Yin}, {\em Global multidimensional transonic conic shock wave for the perturbed supersonic flow past a cone}, SIAM J. Math. Anal. 41 (2009) 178--218.

\bibitem{Yin} {\sc H. C. Yin}, {\em Global existence of a shock for the supersonic flow past a curved wedge},   Acta Math. Sinica 22 (2006) 1425--1432.

\bibitem{ZGC} {\sc  C. Zamfirescu, A. Guardone and P. Colonna}, {\em Admissibility region for rarefaction shock waves in dense gases}, J. Fluid Mech. 599 (2008) 363--381.


    \bibitem{ZE} {\sc Y. B. Zel'dovich},
{\em On the possibility of rarefaction shock waves}, Zh. Eksp. Teor. Fiz. 4 (1946) 363--364.


 \bibitem{Zhang} {\sc Y. Q. Zhang}, {\em Global existence of steady supersonic potential flow past a curved wedge with a piecewise smooth boundary},
SIAM J. Math. Anal. 31 (1999) 166--183.



\end{thebibliography}
\end{document}